\documentclass[a4paper,reqno,11pt]{amsart}

\usepackage{amssymb,amsmath,amsfonts,amsthm,amscd,amstext,mathtools}
\usepackage[english]{babel}
\usepackage[utf8]{inputenc}
\usepackage{enumerate}
\usepackage{graphicx}
\usepackage[msc-links]{amsrefs}
\usepackage{url}
\usepackage{color}
\usepackage{caption}
\usepackage{subcaption}

\usepackage[a4paper,scale={0.72,0.74},marginratio={1:1},footskip=7mm,headsep=10mm]{geometry}

\usepackage[pdftex]{hyperref}

\numberwithin{equation}{section}

\newtheorem{theorem}{Theorem}[section]
\newtheorem{claim}[theorem]{Claim}
\newtheorem{lemma}[theorem]{Lemma}
\newtheorem{proposition}[theorem]{Proposition}
\newtheorem{corollary}[theorem]{Corollary}
\newtheorem{remark}[theorem]{Remark}
\newtheorem{definition}[theorem]{Definition}

\DeclareMathOperator{\Hess}{Hess}

\long\def\xcom#1{}

\newcommand{\cA}{{\ensuremath{\mathcal A}} }
\newcommand{\cB}{{\ensuremath{\mathcal B}} }
\newcommand{\cC}{{\ensuremath{\mathcal C}} }
\newcommand{\cD}{{\ensuremath{\mathcal D}} }
\newcommand{\cE}{{\ensuremath{\mathcal E}} }

\newcommand{\cL}{{\ensuremath{\mathcal L}} }

\newcommand{\cN}{{\ensuremath{\mathcal N}} }

\newcommand{\cP}{{\ensuremath{\mathcal P}} }

\newcommand{\cR}{{\ensuremath{\mathcal R}} }

\newcommand{\cT}{{\ensuremath{\mathcal T}} }
\newcommand{\cU}{{\ensuremath{\mathcal U}} }
\newcommand{\cV}{{\ensuremath{\mathcal V}} }

\newcommand{\gep}{\varepsilon}       

\renewcommand{\tilde}{\widetilde}          
\DeclareMathSymbol{\leqslant}{\mathalpha}{AMSa}{"36} 
\DeclareMathSymbol{\geqslant}{\mathalpha}{AMSa}{"3E} 
\DeclareMathSymbol{\eset}{\mathalpha}{AMSb}{"3F}     
\newcommand{\dd}{\text{\rm d}}             

\newcommand{\R}{\mathbb{R}}

\newcommand{\Z}{\mathbb{Z}}
\newcommand{\N}{\mathbb{N}}

\def\bs{\boldsymbol}

\newcommand\bP{\ensuremath{\bs{\mathrm{P}}}}
\newcommand\bE{\ensuremath{\bs{\mathrm{E}}}}

\newcommand{\ind}{{\sf 1}}

\renewcommand{\epsilon}{\varepsilon}
\renewcommand{\theta}{\vartheta}
\renewcommand{\phi}{\varphi}

 \newcommand{\be}[1]{\begin{equation}\label{#1}}
 \newcommand{\ee}{\end{equation}}

 \newcommand{\bl}[1]{\begin{lemma}\label{#1}}
 \newcommand{\el}{\end{lemma}}

 \newcommand{\br}[1]{\begin{remark}\label{#1}}
 \newcommand{\er}{\end{remark}}

 \newcommand{\bt}[1]{\begin{theorem}\label{#1}}
 \newcommand{\et}{\end{theorem}}

 \newcommand{\bd}[1]{\begin{definition}\label{#1}}
 \newcommand{\ed}{\end{definition}}

 \newcommand{\bcl}[1]{\begin{claim}\label{#1}}
 \newcommand{\ecl}{\end{claim}}

 \newcommand{\bp}[1]{\begin{proposition}\label{#1}}
 \newcommand{\ep}{\end{proposition}}

 \newcommand{\bc}[1]{\begin{corollary}\label{#1}}
 \newcommand{\ec}{\end{corollary}}

 \newcommand{\bpr}{\begin{proof}}
 \newcommand{\epr}{\end{proof}}

 \newcommand{\bi}{\begin{itemize}}
 \newcommand{\ei}{\end{itemize}}

\newcommand{\esp}[1]{\mathbb{E}\etc{#1}}
\newcommand{\ens}[1]{\left\{#1\right\}}

\newcommand{\prob}[1]{\mathbb{P}\etp{#1}}
\newcommand{\etc}[1]{\left [#1 \right ]}
\newcommand{\etp}[1]{\left (#1 \right )}
\newcommand{\valabs}[1]{\left|#1 \right|}
\newcommand{\norme}[1]{\left\Vert#1 \right\Vert}
\newcommand{\cvloi}[1]{\xrightarrow{d}{#1}}
\newcommand{\unsur}[1]{\frac{1}{#1}}
\newcommand{\floor}[1]{\lfloor#1\rfloor}
\newcommand{\espbeta}[1]{\bE_\beta\etc{#1}}

\newcommand{\probbeta}[1]{\bP_\beta\etp{#1}}
\newcommand{\probmubeta}[1]{\bP_{\beta,\mu_\beta}\etp{#1}}
\newcommand{\probbetax}[1]{\bP_{\beta,x}\etp{#1}}

\newcommand{\Nu}{\mathfrak{N}}
\newcommand{\Ju}{\mathfrak{J}}
\newcommand{\PPbeta}{\mathfrak{P}_\beta}
\newcommand{\EEbeta}{\mathfrak{E}_\beta}

\newcommand{\fnt}{{\floor{N\bar{t}}}}
\newcommand{\fns}{{\floor{N\bar{s}}}}
\newcommand{\Pnhn}{P_{N,H_N^{q}}}
\newcommand{\Enhn}{E_{N,H_N^{q}}}
\newcommand{\probnhn}[1]{\Pnhn\etp{#1}}
\newcommand{\espnhn}[1]{\Enhn\etp{#1}}
\newcommand{\blbln}{\bar{L}_{\bar{\Lambda}_N}}
\newcommand{\undemi}{\frac12}
\newcommand{\un}[1]{1_{\etp{#1}}}

\newcommand{\Ll}{\mathfrak{L}}
\newcommand{\Llam}{\Ll_\Lambda}
\newcommand{\Llamn}{\Ll_{\Lambda_n}}

\newcommand{\taufrak}{\mathfrak{T}}

\author{Philippe Carmona}

\address{Laboratoire de Math\'ematiques Jean Leray UMR 6629,
Universit\'e de Nantes, 2 Rue de la Houssini\`ere,
BP 92208, F-44322 Nantes Cedex 03, France}

\email{philippe.carmona@univ-nantes.fr}

\author{Nicolas P\'etr\'elis}

\address{Laboratoire de Math\'ematiques Jean Leray UMR 6629,
Universit\'e de Nantes, 2 Rue de la Houssini\`ere,
BP 92208, F-44322 Nantes Cedex 03, France}

\email{nicolas.petrelis@univ-nantes.fr}



\title[IPDSAW : scaling limits]
{Interacting partially directed self avoiding walk : Scaling Limits.}

\date{\today}

\begin{document}

\begin{abstract}
This paper is dedicated to the investigation of a $1+1$ dimensional
self-interacting and partially directed self-avoiding walk, usually
referred to by the acronym IPDSAW and  introduced in \cite{ZL68} by Zwanzig and Lauritzen to study the 
collapse transition of an homopolymer dipped in a poor solvant. The intensity  of the interaction between monomers is denoted by $\beta\in(0,\infty)$ and there exists a critical 
threshold  $\beta_c$ which determines the three regimes displayed by the polymer, i.e., \emph{extended} for $\beta<\beta_c$, \emph{critical} for $\beta=\beta_c$ and 
\emph{collapsed} for $\beta>\beta_c$.  

In  \cite{POBG93}, physicists displayed some numerical results concerning the typical growth rate of some geometric features of the path as its length  $L$ diverges. 
From this perspective the quantities of interest are the projection of the path onto the horizontal axis (also called horizontal extension) and the projection of the path onto 
the vertical axis for which it is useful to define the lower and the upper envelopes of the path.
With the help of a new random walk representation, we proved in  \cite{CNGP13} that the 
path grows horizontally like $\sqrt{L}$  in its collapsed regime  and that, once rescaled by $\sqrt{L}$ vertically and horizontally, its upper and lower envelopes converge 
towards some deterministic Wulff shapes. 

In the present paper, we bring the geometric investigation of the path several steps further. In the extended regime,  we prove a law of large number for the horizontal extension of the polymer rescaled by its total length $L$, we provide a precise asymptotics of the partition function and we show that  its lower and upper envelopes, once rescaled in time by $L$ and in space by $\sqrt{L}$,  converge towards the same Brownian motion.
In the critical regime we identify the limiting distribution of the horizontal extension rescaled by $L^{2/3}$ and we 
show that the excess  partition function decays as $L^{2/3}$ with an explicit prefactor.  In the collapsed regime,  we identify the joint limiting distribution of the fluctuations of the  upper and lower envelopes around their associated limiting Wulff shapes,  rescaled in time by $\sqrt{L}$ and in space by $L^{1/4}$.



\end{abstract}
\maketitle


\section{Introduction}
In this paper we consider a model of statistical mechanics introduced in \cite{ZL68} by Zwanzig and Lauritzen and refered to as
interacting partially directed self avoiding walk (IPDSAW). The model is a $(1+1)$-dimensional partially directed version of the interacting self-avoiding walk (ISAW)
introduced by Flory in \cite{Flory} as a model for an homopolymer in a poor solvent.

The aim of our paper  is to pursue the investigation of the IPDSAW initiated in \cite{NGP13} and \cite{CNGP13} and in particular to 
display the infinite volume limit of some features of the model when
the size of the system diverges for each of the  three regimes:     
collapsed, critical and extended.  The first object to be considered is the horizontal extension of the path.
Then, we will consider the whole path, properly rescaled and look at
its infinite volume limit in the extended phase and  in the collapsed
phase.

Let us point  that numerical simulations are difficult~\cite{POBG93} and have not led to theoretical results 
about the path properties of the polymer in the three regimes that we
establish in this paper.

\subsection{Model}
The model can be defined in a simple manner. An allowed configuration for the polymer is given by a family of oriented vertical stretches. To be more specific,  for a polymer made of $L\in \N$ monomers, the possible configurations are gathered in $\Omega_L:=\bigcup_{N=1}^L\mathcal{L}_{N,L}$, where $\mathcal{L}_{N,L}$ is the set 
consisting of all families made of  $N$ vertical stretches that have a total length $L-N$, that is
\begin{equation}\label{defLL}
\textstyle\mathcal{L}_{N,L}=\Bigl\{l\in\mathbb{Z}^N:\sum_{n=1}^N|l_n|+N=L\Bigr\}.
\end{equation}
Note that with such configurations, the modulus of a given stretch
corresponds to the number of monomers constituting this stretch (and
the sign gives the direction upwards or downwards). Moreover, any two consecutive vertical
stretches are separated by a monomer placed horizontally and this explains why $\sum_{n=1}^N |l_n|$ must equal $L-N$ in order for $l=(l_i)_{i=1}^N$ to be associated with a polymer made of $L$ monomers (see Fig. \ref{fig:stretches}).


\begin{figure}
        \centering
        \begin{subfigure}[b]{0.3\textwidth}
                \includegraphics[width=\textwidth]{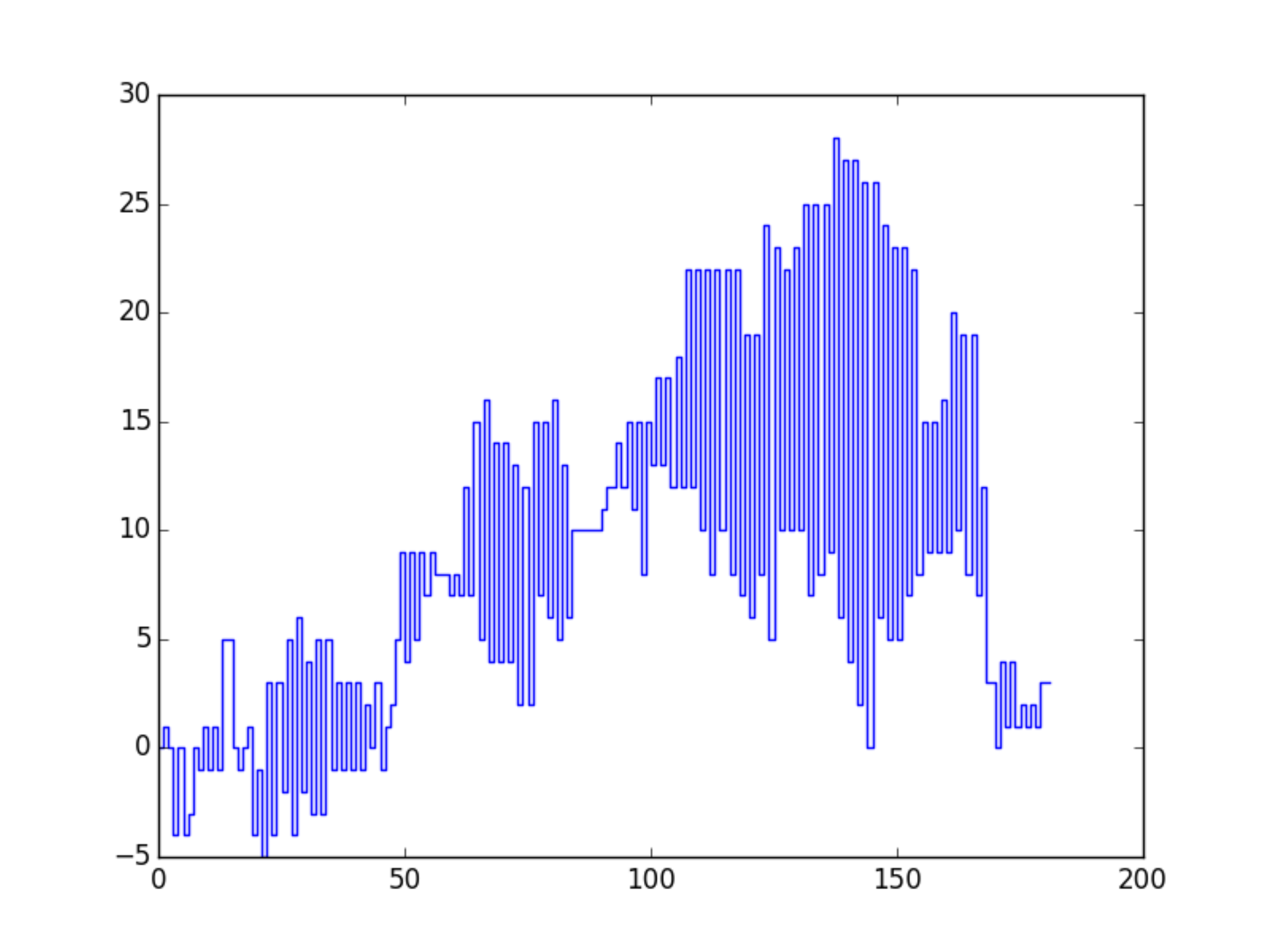}
        \end{subfigure}%
	\begin{subfigure}[b]{0.3\textwidth}
\includegraphics[width=\textwidth]{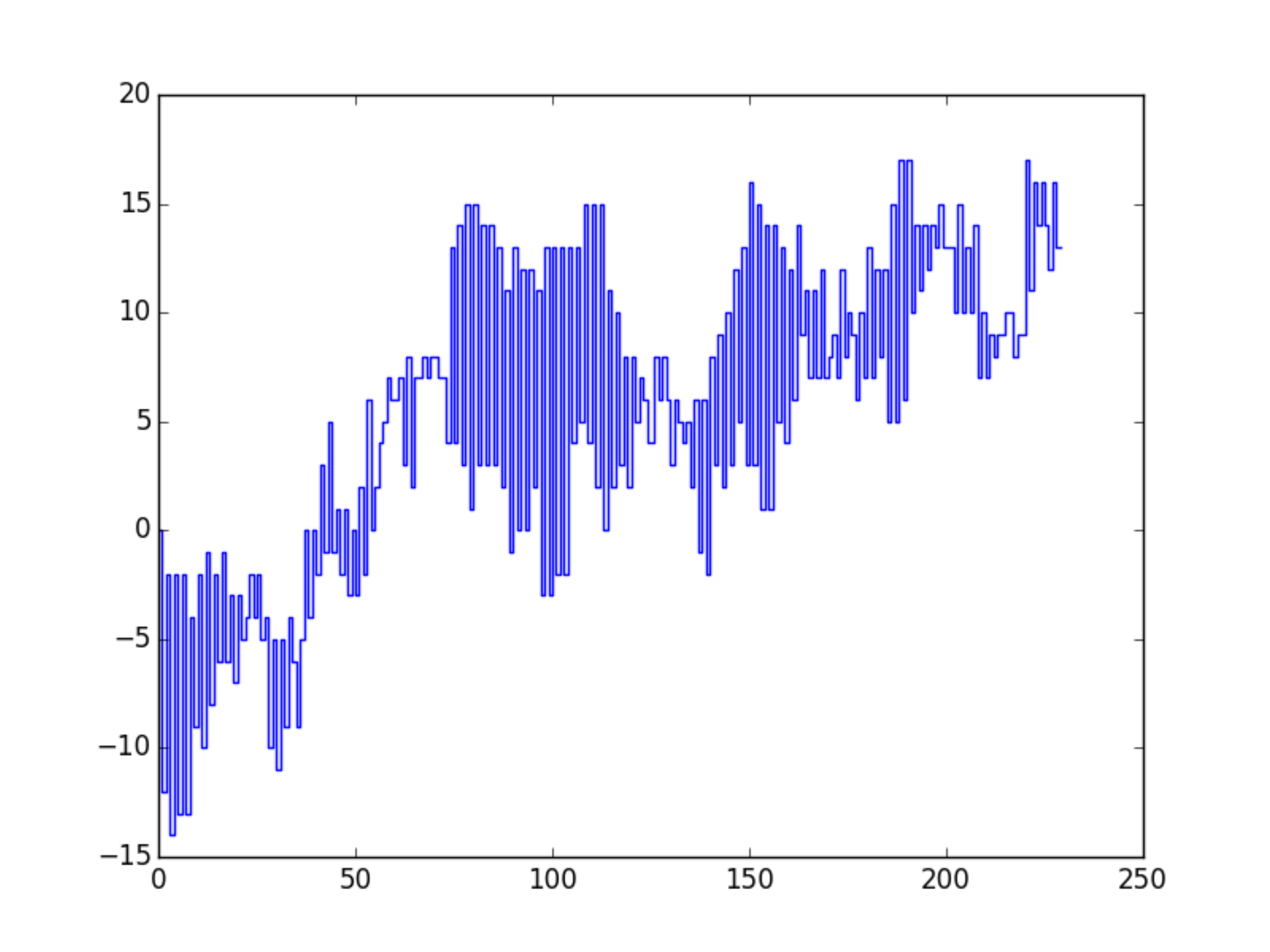}
\end{subfigure}
	\begin{subfigure}[b]{0.3\textwidth}
\includegraphics[width=\textwidth]{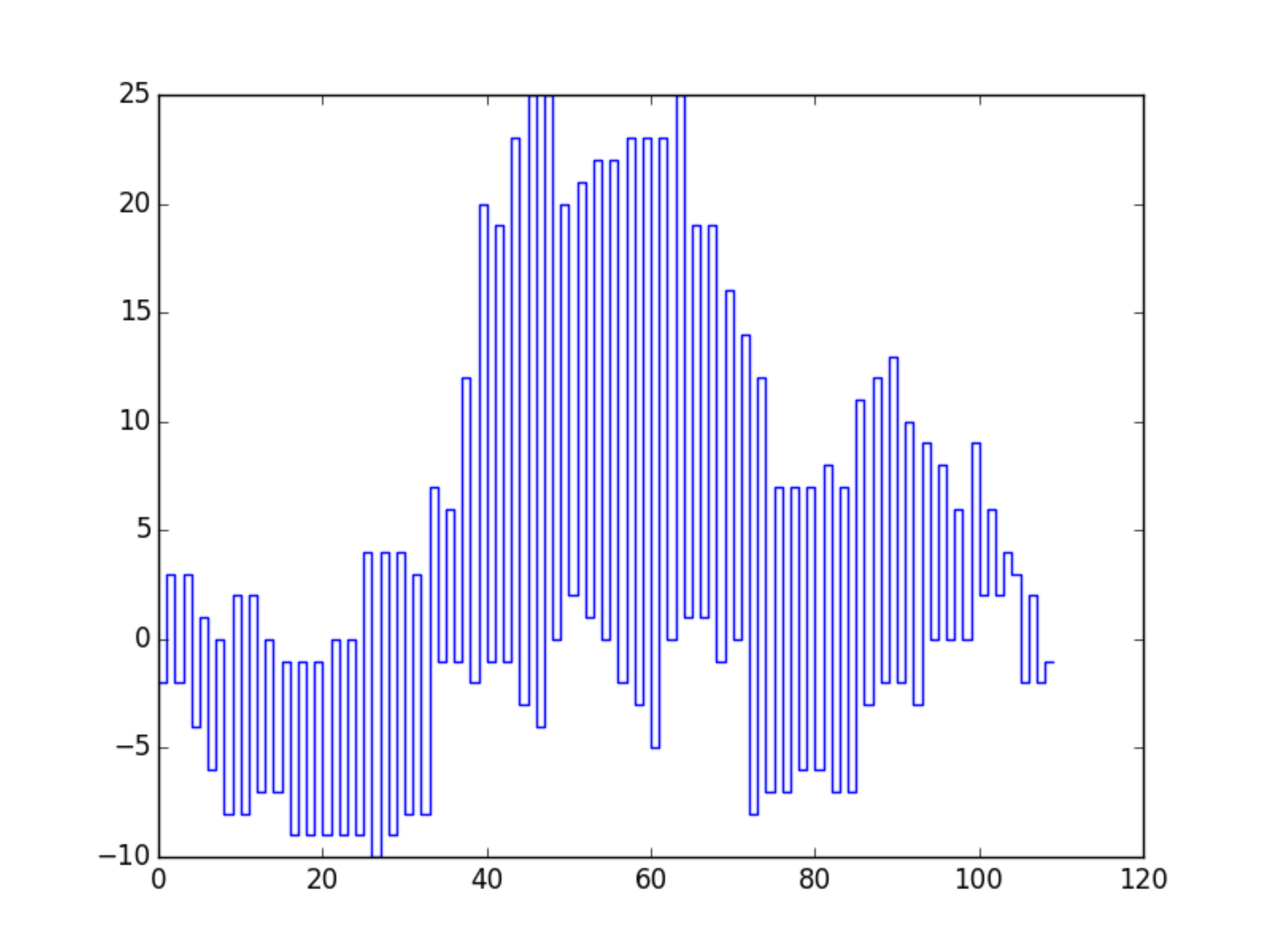}
\end{subfigure}
	\caption{Simulations of IPDSAW for critical temperature
          $\beta=\beta_c$ and length $L=1600$}
	\label{fig:simulations}
\end{figure}

The repulsion between the monomers constituting the polymer and the solvent around them is 
taken into account in the Hamiltonian associated with each path $l\in \Omega_L$ by rewarding 
energetically those pairs of consecutive stretches with opposite directions, i.e.,
\begin{equation}
\textstyle H_{L,\beta}(l_1,\ldots,l_N)=\beta\sum_{n=1}^{N-1}(l_n\;\tilde{\wedge}\;l_{n+1})
\end{equation}
where
\begin{equation}
x\, \tilde\wedge\, y=\begin{dcases*}
	|x|\wedge|y| & if $xy<0$,\\
  0 & otherwise.
  \end{dcases*}
\end{equation}

\begin{figure}[ht]\center
	\includegraphics[width=.3\textwidth]{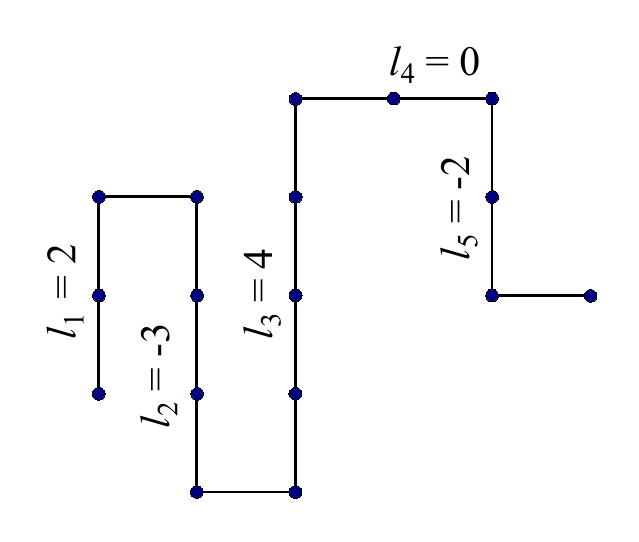}
	\caption{Example of a trajectory  $l\in \cL_{N,L}$ with $N=5$ vertical stretches, a total  length $L=16$ and 
	and Hamiltonian $H_{L,\beta}(\pi)=5 \beta$.}
	\label{fig:stretches}
\end{figure}

One can already note that large Hamiltonians will be assigned to trajectories made of few 
but long vertical stretches with alternating signs. Such paths will be referred to as collapsed configurations.
With the Hamiltonian  in hand we can define the polymer measure as
\be{defpolme}
P_{L,\beta}(l)=\frac{e^{H_{L,\beta}(l)}}{Z_{L,\beta}},\quad l\in \Omega_L
\ee
where $Z_{L,\beta}$ is  partition function of the model, i.e.,
\begin{equation}\label{pff}
Z_{L,\beta}=\sum_{N=1}^{L}\sum_{l\in\mathcal{L}_{N,L}} \,  e^{H_{L,\beta}(l)}.
\end{equation}

\subsection{Random walk representation and  collapse transition}\label{free}
It is custom for such statistical mechanical models to introduce the free energy per monomers $f(\beta)$ as the limiting exponential growth rate of the partition function,
i.e.
\begin{equation}\label{exis}
f(\beta)=\lim_{L\to\infty}\frac{1}{L}\log Z_{L,\beta}=\sup_{L\in\mathbb{N}}\frac{1}{L}\log Z_{L,\beta}<\infty.
\end{equation}
Both equalities in \eqref{exis} are straightforward consequences of the fact that  $(\log Z_{L,\beta})_{L=1}^\infty$
is a super-additive sequence.  
A phase transitions of such system is associated with a loss of analyticity of $\beta\to f(\beta)$ at some critical point.
In \cite{NGP13}, we displayed an alternative way of computing the partition function that turns out to 
simplify the investigation of the phase diagram. It is indeed possible to exhibit an 
auxiliary random walk $V=(V_i)_{i=0}^\infty$ with geometric increments and with law $\bP_\beta$ such that 
\be{partfun}
Z_{L,\beta}=e^{\beta L} \sum_{N=1}^L (\Gamma_\beta)^N \bP_{\beta} (V\in \cV_{N,L-N})
\ee
where $\cV_{N,L}:=\{V\colon G_N(V)=L-N, V_{N+1}=0\}$, where $G_N(V):=\sum_{i=1}^N |V_i|$ is the geometric area in between $V$ and the horizontal axis up to time $N$
and where 
\begin{align}\label{disti}
\Gamma_\beta\;=\frac{c_\beta}{e^{\beta}}
\end{align}
with  $c_\beta:=\frac{1+e^{-\beta/2}}{1-e^{-\beta/2}}$ that is simply the normalizing constant of $\bP_\beta$. In section \ref{sec:rep}, we will recall how to exhibit such a  random walk representation of $Z_{L,\beta}$ but let us mention already that, for $N\in \{1,\dots,L\}$, the contribution to the partition function of those trajectories in $\cL_{N,L}$ is given by the term indexed by $N$
in the sum of \eqref{partfun}.  
\smallskip

A useful feature of the random walk representation  is that it allows us to read the phase diagram on   \eqref{partfun} directly.  To that purpose, we note that \eqref{partfun} makes it  natural to define the
\emph{excess free energy} as  $\tilde{f}(\beta):=f(\beta)-\beta$ so that the exponential growth rate of the sum in the r.h.s. of \eqref{partfun}
equals $\tilde f(\beta)$.
Then,  $\beta\mapsto \Gamma_\beta$ being  a decreasing bijection from 
$(0,\infty)$ to $(0,\infty)$, we denote by $\beta_c$ the unique solution of $\Gamma_\beta=1$.  For $\beta\geq \beta_c$
the inequality $\Gamma_\beta\leq 1$ immediately yields that $\tilde f(\beta)=0$ since those terms indexed by  $N\sim\sqrt{L}$ in
\eqref{partfun} decay subexponentially. As a consequence the trajectories dominating $Z_{L,\beta}$ have a small 
horizontal extension, i.e., $N=o(L)$. When $\beta<\beta_c$ in turn, $\Gamma_\beta>1$ and since for $c\in (0,1]$ the quantity 
$P_\beta(\cV_{cL,(1-c)L})$ decays exponentially fast with a rate that vanishes as $c\to 0$
we can claim that the dominating 
trajectories in $Z_{L,\beta}$ have an horizontal extension of order $L$, and moreover that  $\tilde f(\beta)>0$.
Thus, the free energy is non analytic at $\beta_c$ and we can
partition $[0,\infty)$ into a collapsed phase denoted by $\mathcal{C}$ and an extended phase denoted by $\mathcal{E}$, i.e,
\begin{equation}\label{eq:colphase}
\mathcal{C}:=\{\beta:\tilde{f}(\beta)=0\}=\{\beta:\beta\geq\beta_c\}
\end{equation}
and
\begin{equation}\label{eq:extphase}
\mathcal{E}:=\{\beta:\tilde{f}(\beta)>0\}=\{\beta:\beta<\beta_c\}.
\end{equation}

We shall see that, in fact, there are three regimes; collapsed ($\beta>\beta_c$), critical ($\beta=\beta_c$)
and extended ($\beta<\beta_c$), in which the asymptotics of the partition function and the path  properties are radically different.

\br{satp}
Observe that the main difference between IPDSAW and wetting/copolymer models, comes from the fact that 
the saturated phase (where the free energy is trivial) corresponds to a maximization of energy for IPDSAW, and, on the opposite,
to a maximization of entropy for wetting/copolymer models. We refer to Giacomin \cite{cf:Gia} or den Hollander \cite{cf:dH} for a review on wetting/copolymer models.
\er

\section{Main results}

We mentioned in the preceding section that the excess free energy $\tilde f(\beta)$ is the exponential 
growth rate of the sum in the r.h.s. of \eqref{partfun}. For this reason we set  
\be{defZti}
\tilde Z_{L,\beta}=e^{-L \beta} Z_{L,\beta},
\ee
and we recall that the definition of the polymer measure in \eqref{defpolme} is left unchanged  if we replace the denominator by $\tilde{Z}_{L,\beta}$
and substract  $L \beta$ to the Hamiltonian.  


\subsection{Scaling  limit of  the horizontal extension}\label{nor}
Displaying sharp asymptotic estimates of the partition function as the system size diverges is a major issue in statistical mechanics.
Computing the probability mass of a certain subset of trajectories under the polymer measure indeed requires to have a good control on the denominator in 
\eqref{defpolme}.  For the extended and the critical regimes, we display in Theorem \ref{pfa} below an equivalent of the partition function allowing us e.g 
to exhibit the polynomial decay rate of the partition function at the critical point.  For the collapsed regime,
in turn, we recall the bounds on $\tilde{Z}_{L,\beta}$ that had been obtained in \cite{CNGP13} allowing us to identify its sub-exponential decay rate.

Note that in Remark \ref{recons} below, we provide some complements concerning Theorems \ref{pfa} and \ref{pfa2}  among which the exact value of some 
pre-factors when an expression  is available.  
We also denote by  $f_{ex}$ the density
of the area below a normalized Brownian excursion (see e.g. Janson~\cite{J07}) and we set
$C_\beta:=(\bE_\beta(V_1^2))^{-1/2}$. Thus, we can define  $w(x)=C_\beta\, f_{ex}(C_\beta\, x)$.

\begin{theorem}[Asymptotics of partition function]\label{pfa}
\begin{enumerate}[(1)]
\item if $\beta<\beta_c$, there exists a $c>0$ such that 
$$\tilde{Z}_{L,\beta}= c\, e^{\tilde{f}(\beta) L} (1+o(1)), $$
\item for $\beta=\beta_c$, there exists a $c>0$ such that 
$$\tilde{Z}_{L,\beta}= \frac{c}{L^{2/3}} (1+o(1)) \quad  \text{ with} \quad c=\frac{1+e^{\frac{\beta}{2}}}{(24\, \pi\, \bE_{\beta} (V_1^2))^{\frac12} \,\int_0^{+\infty}
  x^{-3} w(x^{-\frac32})\, dx},$$\\
\item for $\beta>\beta_c$, there exists $c_1,c_2,\kappa>0$ such that 
$$\frac{c_1}{L^\kappa} e^{\tilde{G}(a(\beta)) \sqrt{L}} \leq \tilde{Z}_{L,\beta} \leq \frac{c_2}{\sqrt{L}} e^{\tilde{G}(a(\beta)) \sqrt{L}}
\quad \text{for}\quad  L\in \N.$$
\end{enumerate}
\end{theorem}
\medskip

For each $l\in \Omega_L$, the variable $N_l$ denotes the horizontal extension of $l$, i.e., the integer $N\in \{1,\dots,L\}$
such that  $l\in \cL_{N,L}$.
With Theorem \ref{pfa2} below, we provide the scaling limit of the horizontal extension of a typical path $l$
sampled from $P_{L,\beta}$ and as $L\to \infty$. As for Theorem \ref{pfa} and for the sake of completeness, 
we integrate the collapsed regime into the theorem although this
regime was dealt with  in \cite[Theorem D]{CNGP13}.

\begin{theorem}[Horizontal extension]\label{pfa2}

  \begin{enumerate}[(1)]
  \item if $\beta<\beta_c$, there exists a real
constant $e(\beta)\in (0,1)$ such that 
\be{cvhex}
\lim_{L \to \infty} P_{L,\beta} \Big( \Big| \frac{N_l}{L}-e(\beta)\Big |\geq \gep\Big)=0.
\ee
\item if $\beta=\beta_c$, then 
$$\lim_{L\to \infty} \frac{N_l}{L^{2/3}}=_{\text{law}} g_1$$
where $g_a =\inf\ens{ t>0 \int_0^t \valabs{B_s}\, ds = a}$ is the
continuous inverse of the geometric Brownian area, and we consider
$g_1$ under the conditional law of the Brownian motion conditioned by
$B_{g_1}=0$.

\item  If $\beta>\beta_c$, there exists a unique
real number $a(\beta)$ such that 
\be{cvhco}
\lim_{L \to \infty} P_{L,\beta} \Big( \Big| \frac{N_l}{\sqrt{L}}-a(\beta)\Big |\geq \gep\Big)=0.
\ee
\end{enumerate}
\end{theorem}
\begin{remark}\label{recons}
\rm 
\noindent \begin{enumerate}[(1)]
 \item For the extended regime, in Section \ref{extended},  we will decompose each path into a succession of patterns (sub pieces) 
 and we will associate with our model an underlying regenerative process  $(\sigma_i,\nu_i,y_i)_{i\in \N}$  of law $\PPbeta$ 
 in such a way that   $\sigma_i$ (resp. $\nu_i$, resp. $y_i$) plays the role of the number of monomers constituting the 
$i$th pattern (resp. the horizontal extension of the $i$th pattern, resp. the vertical displacement of the $i$th pattern). Then, the constant $c$ in Theorem \ref{pfa} (1) and  the limiting rescaled horizontal 
extension in Theorem \ref{pfa2} (1) satisfy 
$$c=\tfrac{1}{\EEbeta(\sigma_1)}\quad \text{and}\quad
e(\beta)=\tfrac{\EEbeta(\nu_1)}{\EEbeta(\sigma_1)}.$$



\item For the critical regime $\beta=\beta_c$, the appearance of the distribution of $g_1$
  is explained at the end of  Section~\ref{sec:critic}.

\item For the collapsed regime, by inspecting closely (4.29)-(4.35) of \cite{CNGP13}, we see
  that this result can be easily generalized  to a large deviation
  principle of speed $\sqrt{L}$ for the sequence of random variables
  $(N_l/\sqrt{L})_{L\in \N}$  with the good rate function  $a\in
  (0,\infty)\to \tilde{G}(a(\beta))-\tilde G(a)$ which
  admits a unique minimum $a(\beta)$. This large deviation result
  holds under $P^{o}_{L,\beta}$ the polymer measure restricted to have only one bead.
  A rigorous definition of a bead is recalled in the paragraph of Section \ref{ver} that is dedicated to the collapsed phase.
  Note, however, that we are not able at this stage to prove the same LDP under $P_{L,\beta}$.

The good rate function $\tilde{G}(a)-\tilde{G}(a(\beta))$ was carefully investigated in \cite[Remark 5]{CNGP13} and an exact expression was provided, i.e.,
\begin{equation}\label{defg}
\tilde{G}(a):=a\log\Gamma(\beta)-\tfrac{1}{a}\,\tilde h_0\bigl(\tfrac{1}{a^2},0\bigr)+a\Llam\bigl(\tilde H\bigl(\tfrac{1}{a^2},0\bigr)\bigr),
\end{equation}
where $\Llam$ and $\tilde H$ are defined in \eqref{eq:LLam} and \eqref{eq:deftil}.
Note that, for $\beta>\beta_c$, the function  $a\mapsto\tilde{G}(a)$ is $\cC^\infty$, strictly concave, strictly negative and $a(\beta)$ is the unique zero of its derivative  on $(0,\infty)$.


\end{enumerate}

\end{remark}

\subsection{Scaling limit of the vertical extension}\label{ver}

The horizontal extension of $l\in \Omega_L$  can be viewed as the projection of $l$ onto the horizontal axis. Thus,  after providing the scaling limit of $N_l$ in each of the three phases, a natural issue consists in displaying  the scaling limit of the projection of the polymer onto the vertical axis. To be more specific, we will try to exhibit the scaling limit of the whole path rescaled horizontally by its horizontal extension
$N_l$  and vertically by some ad-hoc power of $N_l$. 

To that aim, the fact that each trajectory $l\in \Omega_L$ is made of a succession of vertical stretches  makes it convenient  to give a representation of the trajectory in terms of its upper and lower envelopes. Thus, we
pick $l \in \cL_{N,L}$ and we let $\cE_l^+=(\cE^+_{l,i})_{i=0}^{N+1}$ and $\cE_l^-=(\cE^-_{l,i})_{i=0}^{N+1}$  be  the upper and the lower envelopes of $l$, i.e., the $(1+N)$-step paths that link the 
top and the bottom of each stretch consecutively. Thus, $\cE^+_{l,0}=\cE^-_{l,0}=0$,
\begin{align}\label{trek}
\cE^+_{l,i}&=\max\{l_1+\dots+l_{i-1}, l_1+\dots+l_{i}\},\quad i\in \{1,\dots,N\},\\
\cE^-_{l,i}&=\min\{l_1+\dots+l_{i-1}, l_1+\dots+l_{i}\},\quad i\in \{1,\dots,N\},
\end{align}
and $\cE^+_{l,N+1}=\cE^-_{l,N+1}=l_1+\dots+l_{N}$ (see Fig. \ref{fig:upper}).
Note that the area in between these two envelopes is completely filled by the path and therefore, we will focus on the scaling limits of $\cE^+_l$ and $\cE^-_l$.

At this stage, we define $\tilde Y: [0,1]\to \R$  to be the time-space rescaled cadlag process of a given  $(Y_i)_{i=0}^{N+1}\in \Z^{N+1}$ satisfying  $Y_0=0$. Thus,  
\be{defti}
\tilde Y(t)= \frac{1}{N+1}\,  Y_{\lfloor t\,(N+1)\rfloor},\quad t\in [0,1],
\ee
and  for each $l \in \cL_{N,L}$ we let $\tilde\cE_l^+$, $\tilde\cE_l^-$ be the time-space rescaled processes associated with the upper envelope $\cE_l^+$ and with  the lower envelope $\cE_l^-$, respectively.

In this paper we will focus on the infinite volume limit of the whole path in the extended phase ($\beta<\beta_c$) and inside the collapsed phase ($\beta>\beta_c$). 
Concerning the critical regime ($\beta=\beta_c$) this limit will be discussed as an open problem in section \ref{disc} below. 

\begin{figure}[ht]\center
\begin{tabular}{cc}
\includegraphics[width=.30\textwidth]{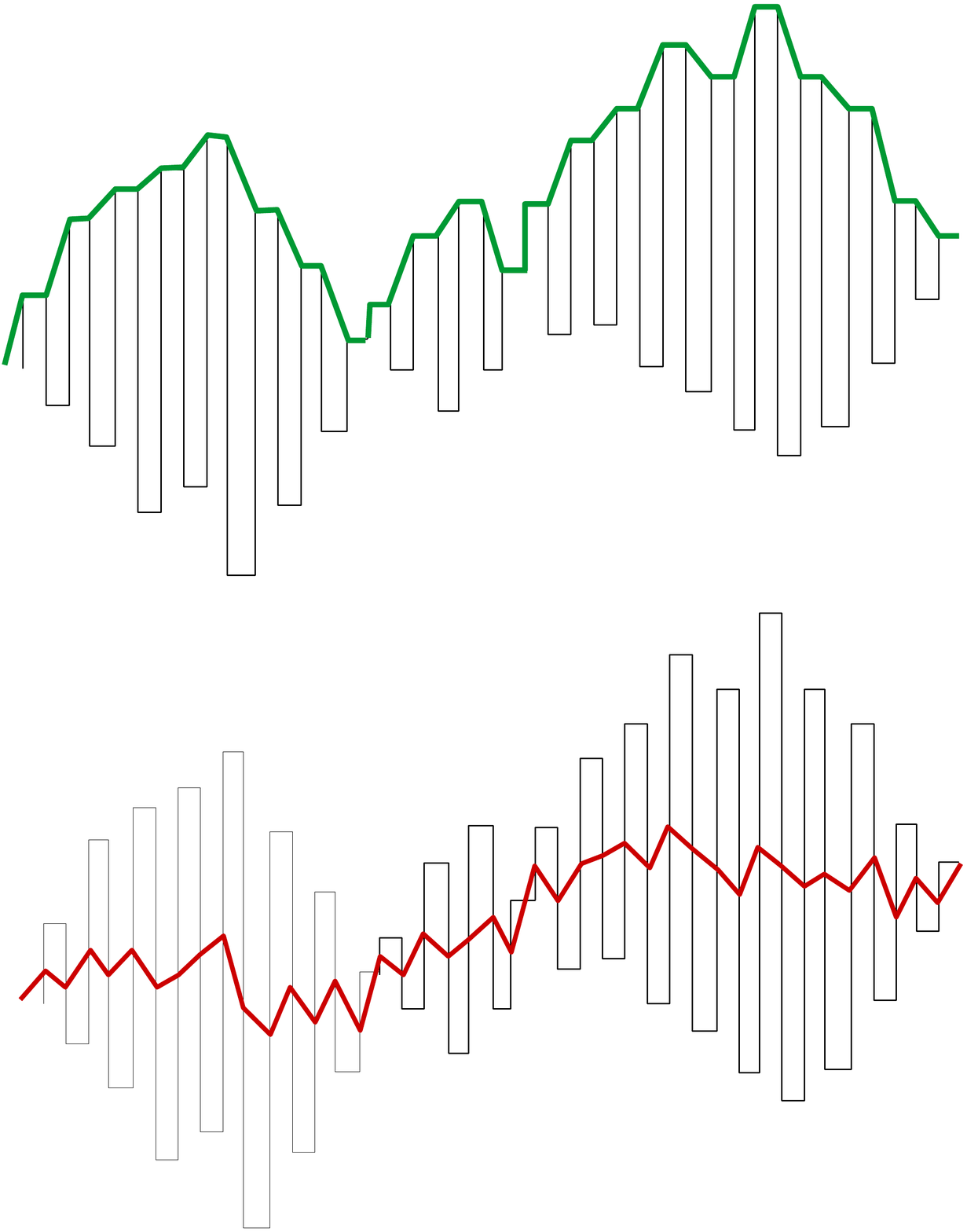} & \includegraphics[width=.30\textwidth]{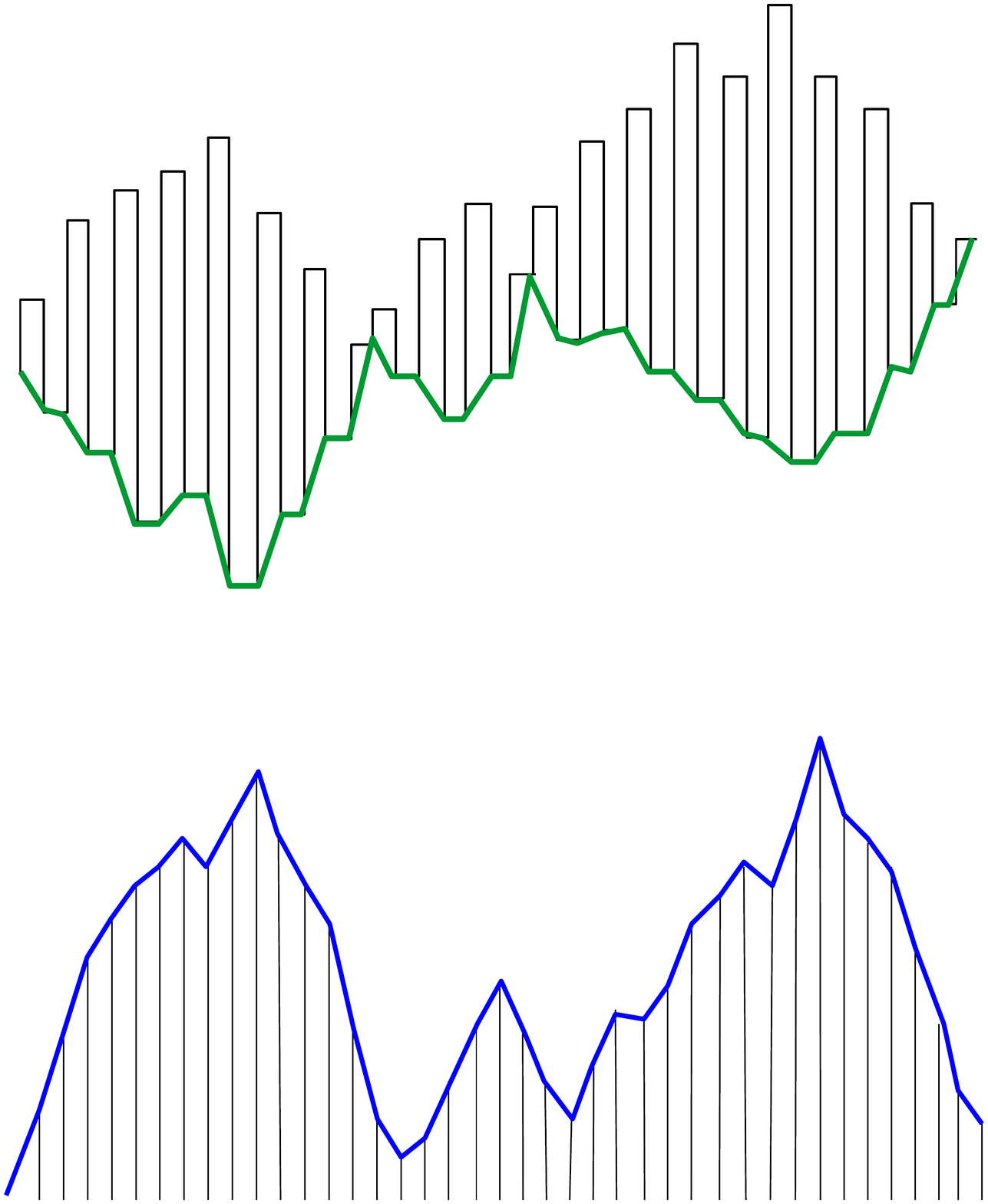}\\

\end{tabular}
	
	\caption{Example of the upper envelope (top-left picture), of the 
	lower envelope (top-right picture), of the middle line (bottom left picture) and of the profile (bottom right picture)
	of a given trajectory (in dashed line).}
	\label{fig:upper}
\end{figure}

\subsubsection{\bf The extended phase ($\beta<\beta_c$)}

When $\beta<\beta_c$ and under $P_{L,\beta}$, we have seen that a typical path $l$ adopts an extended configuration, characterized by a number of horizontal steps of order $L$. We let $Q_{L,\beta}$ be the law of $
\sqrt{N_l} \, \big(\tilde \cE_l^-(s), \tilde\cE_l^{+}(s)\big)_{s\in [0,1]}$ under
$P_{L,\beta}$. We let also $(B_s)_{s\in [0,1]}$ be a standard Brownian motion.

\begin{theorem} \label{extph}
For  $\beta<\beta_c$, there exists a  $\sigma_{\beta}>0$ such that 
\be{cvvp}
Q_{L,\beta}\xrightarrow[L\to \infty]{d} \sigma_{\beta} \big( B_s,B_s\big)_{s\in [0,1]}.
\ee
\end{theorem}

\begin{remark}
\rm The constant $\sigma_{\beta}$ takes value $\sqrt{\EEbeta(y_1^2)/\EEbeta(\nu_1)}$ where $y_1$ (resp.  $\nu_1$)
correspond to the vertical (resp. horizontal)  displacement of the 
path on one of the pattern mentioned in Remark \ref{recons} (1). These objects are defined rigorously in Section \ref{extended}
below.
\end{remark}

Although the upper and the lower envelopes of each trajectory $l\in \Omega_L$ seem to be the 
appropriate objects to consider when it comes to describing the geometry of the whole path, it turns out that it is simpler to prove Theorem \ref{extph} by recovering the envelopes from 
two auxiliary processes, i.e, the \emph{middle line} $M_l$
and the \emph{ profile} $|l|$. 
Thus, we  associate with each $l \in \cL_{N,L}$ the path $|l|=(|l_i|)_{i=0}^{N+1}$ (with $l_{N+1}=0$ by convention) and the path $M_l=\big(M_{l,i}\big)_{i=0}^{N+1}$  that links the middles of each stretch consecutively, i.e.,   $M_{l,0}=0$  and 
\be{droitmi}
M_{l,i}=l_1+\dots+l_{i-1}+\frac{l_i}{2},\quad i\in \{1,\dots,N\},
\ee
and $M_{l,N+1}=l_1+\dots+l_N$ (see Fig. \ref{fig:upper}). With the help of  \eqref{defti}, we let $\tilde M_l$ and $\tilde l$ be the time-space rescaled processes associated with $M_l$ and $l$
and one can easily check that 
\be{caract}
\textstyle \tilde \cE_l^+=\tilde M_l+\frac{|\tilde l |}{2}\quad \text{and}\quad \tilde \cE_l^-=\tilde M_l-\frac{|\tilde l|}{2}.
\ee
As a consequence, proving Theorem \ref{extph} is equivalent to proving that 
\be{cvp}
\widehat Q_{L,\beta}\xrightarrow[L\to \infty]{d}  \sigma_{\beta} \big(B_s,0\big)_{s\in [0,1]},
\ee
where $\widehat Q_{L,\beta}$ is the law of $
\sqrt{N_l} \big(\tilde M_l(s), |\tilde
l(s)|\big)_{s\in [0,1]}$ with $l$ sampled from 
$ P_{L,\beta}$.


%
%

\subsubsection{\bf The collapsed phase ($\beta>\beta_c$)}
The collapsed regime was studied in \cite{CNGP13}, where a particular decomposition of the path into \emph{beads}
has been introduced. A bead is a succession of non-zero vertical stretches with alternating signs which ends when two consecutive 
stretches have the same sign (or when a stretch is nul). Such a decomposition is meaningful geometrically and we proved in 
\cite[Theorem C]{CNGP13} that there is a unique macroscopic bead in the collapsed regime and that the number of monomers outside this bead
are at most of order $(\log L)^4$.   

The next step, in the geometric description of the path, consisted in determining the limiting shapes of the envelopes of this unique bead. This has been achieved in \cite{CNGP13} where we proved  that 
 the rescaled upper envelope (respectively lower envelope) converges in probability towards a \emph{deterministic Wulff shape} $\gamma^*_{\beta}$ (resp. $-\gamma^*_{\beta}$) 
 defined as follows
 \be{defwulf}
 \gamma^*_{\beta}(s)=\int_0^s L' \big[(\tfrac{1}{2}-x) \tilde h_0(\tfrac{1}{a(\beta)^2},0) \big] dx,\quad s\in [0,1],
 \ee
 where $L$ is defined in \eqref{defL} and $\tilde h_0$ in \eqref{eq:deftil}.
 Thus, we obtained
 \bt{Convenv}{\rm (\cite{CNGP13} Theorem E)} For $\beta>\beta_c$  and $\gep>0$,
\begin{align}\label{conven}
\nonumber \lim_{L\to \infty} P_{L,\beta}\Big(  \big\|\tilde\cE^+_{l}-\frac{\gamma^*_{\beta}}{2}\big\|_{\infty} >\gep \Big)&=0,\\
\lim_{L\to \infty} P_{L,\beta}\Big(  \big\|\tilde\cE^-_{l}+\frac{\gamma^*_{\beta}}{2}\big\|_{\infty} >\gep \Big)&=0.
\end{align}
\et
 
%
%
%
%

This Theorem has also been stated a a \emph{Shape Theorem} in
\cite{CNGP13}. 
The natural question that comes to the mind is : are we able to
identify the fluctuations around this shape ? 
For technical reasons
that will be discussed in Remark \ref{rtech} below, we are not able to
identify such a limiting distribution. However, we can prove a close
convergence result by working on a particular mixture of those measures $P^{}_{L',\beta}$ for $L'\in K_L:=L+[-\gep(L),\gep(L)]\cap \N$ with $\gep(L):=(\log L)^6$. Thus,  we define the extended set of trajectories  $\widetilde \Omega_L=\cup_{L'\in K_L} \Omega_{L'}$, and we let $\widetilde{P}_{L,\beta}$ be a mixture of those $\big\{P_{L',\beta}, \, L'\in K_L\big\}$ defined by 
\be{rp111}
\widetilde{P}_{L,\beta}\big(\cdot |\, \Omega_{L'}\big )= P_{L',\beta} (\cdot)\quad  \text{and}  \quad 
\widetilde{P}_{L,\beta}\big(\Omega_{L'}\big )=\frac{\widetilde{Z}_{L',\beta}}{\sum_{k\in K_L}  \widetilde{Z}_{k,\beta}}, \quad \text{for} \ L'\in K_L,
\ee
where we recall \eqref{defZti}. In other words,
$\widetilde{P}_{L,\beta}$ can be defined as
\be{defPtil}
\widetilde{P}_{L,\beta}\big(l \big )=\sum_{L'\in K_L} \frac{\widetilde Z_{L',\beta} }{\sum_{k\in K_L} \widetilde{Z}_{k,\beta}}\, P_{L',\beta} (l) \, \ind_{\{l\in \Omega_{L'}\}}, \quad  \text{for}\  l \in \widetilde{\Omega}_L.
\ee
We denote by $\tilde Q_{L,\beta}$ the law of the fluctuations of the envelopes around their limiting shapes, that is the law of the
random processes
\be{rp}
\sqrt{N_l}\,  \Big(\tilde\cE^+_{l}(s)-\tfrac{\gamma^*_{\beta}(s)}{2}, \tilde\cE^-_{l}(s)+\tfrac{\gamma^*_{\beta}(s)}{2}\Big)_{s\in [0,1]}
\ee
where $l$ is sampled from $\tilde P_{L,\beta}$ as $L\to \infty$.  We obtain the following limit. 
%
\begin{theorem}[Fluctuations of the convex envelopes around the Wulff shape]\label{Oscill}
For  $\beta>\beta_c$, and $H=\tilde{H}(q_\beta,0),
q_\beta=\unsur{a(\beta)^2}$ we have the convergence in
distribution 
\be{cv}
\tilde Q_{L,\beta}\xrightarrow[L\to \infty]{d} \Big(\xi_H+\frac{\xi^c_H}{2}, \xi_H-\frac{\xi^c_H}{2}\Big),
\ee
where for $H=(h_0,h_1)$ such that $\etc{h_0,h_0+h_1} \subset D=(-\beta/2,\beta/2)$, the process  $\xi_H=(\xi_H(t), 0\leq t\leq 1)$ is centered
and Gaussian with covariance 
$$\esp{\xi_H(s),\xi_H(t)} = \int_0^{s\wedge t} L''((1-x) h_0 + h_1)\,
dx,$$
and 
where $\xi_H^c:=(\xi^c_H(t), 0\leq t \leq 1)$
is a process independent of $\xi_H$ which has the law of $\xi_H$
conditionally on $\xi_H(1)=\int_0^1\xi_H(s)\, ds = 0$. 
\end{theorem}
From Theorem \ref{Oscill} we deduce that the fluctuations of both envelopes around their limiting shapes are of order $L^{1/4}$.
\medskip

\begin{remark}\label{rtech}\rm
The reason why we prove Theorem \ref{Oscill} under the mixture $\tilde P_{L,\beta}$ rather than $P_{L,\beta}$ is the following.  We need to establish a local limit theorem for the associated random walk of law $\bP_\beta$
conditionned on having a large geometric area $G_N(V)$ and we are unable to do it. Fortunately, we know how to condition the random walk on having a large algebraic area $A_N(V)$ and under the mixture $\tilde P_{L,\beta}$ we are able to compare quantitatively these two conditionings (see Step 2 of the proof of Proposition \ref{llt} in Section \ref{pr1}).

In the construction of the mixture law $\tilde P_{L,\beta}$ (cf. \eqref{defPtil}), the choice of the prefactors of those $P_{L^{'},\beta}$ with $L^{'} \in K_L$ may appear
artificial. However it is conjectured (see e.g. \cite[Section 8]{GT}) that our inequalities in Theorem \ref{pfa} (3) can be improved into
$$ \tilde{Z}_{L,\beta}\sim \frac{B}{L^{3/4}} e^{\tilde{G}(a(\beta)) \sqrt{L}} \quad \text{with}\ B>0,$$
so that the ratio of any two prefactors would converges to $1$ as 
$L\to \infty$ uniformly on the choice of the indices of the two prefactors in $K_L$. In other word, $\tilde P_{L,\beta}$ should, in first approximation, be  the uniform 
mixture of those $\{P_{L',\beta},\, L'\in K_L\}$.

\end{remark}

\begin{remark}\label{remalt}\rm
As for the extended phase with Theorem \ref{extph}, it will be easier to work with the middle line $\tilde{M}_l$ and with 
the profile $|\tilde l |$ defined in (\ref{droitmi}-\ref{caract}).
As a consequence, proving Theorem \ref{Oscill} is equivalent to proving that 
\be{cvpp}
\widehat Q_{L,\beta}\xrightarrow[L\to \infty]{d} \big(\xi_H,\xi^c_H\big),
\ee
where $\widehat Q_{L,\beta}$ is the law of $
\sqrt{N_l} \big(\tilde M_l(s), |\tilde
l(s)|-\gamma^*_{\beta}(s)\big)_{s\in [0,1]}$ with $l$ sampled from 
$\tilde P_{L,\beta}$. The convergence of $\tilde M_l$ in \eqref{cvpp} answers an open question
raised in \cite[Fig. 14 and Table II]{POBG93} where the process is referred to as the \emph{center-of-mass walk}.
\end{remark}

\subsection{Discussion and open problems}\label{disc}

Giving a path characterization of the phase transition is an important issue for polymer models in Statistical Mechanics. From that point of view, identifying in each regime the limiting distribution of the  whole path rescaled in time by its total length $N$ and in space by $\sqrt{N}$ is challenging and meaningful. 
This was studied   in \cite{DGZ05} and \cite{CGZ06} for $(1+1)$-dimensional wetting models
which  deal with a $N$-step  random walk (with continuous or discrete increments)  conditioned to remain non-negative and receiving an energetic reward $\gep$  every time it touches the $x$-axis (which plays the role of a hard wall). 
Such models exhibit a pinning transition at some  $\gep_c$ such that 
when $\gep>\gep_c$ the polymer is  \emph{localized}, meaning that the path typically  remains at distance $O(1)$ from the wall. Thus, the  rescaled path converges to the null function. 
When $\gep<\gep_c$ in turn, the polymer is \emph{delocalized} and visits the origin only $O(1)$ times. Then, the rescaled path  
converges towards a Brownian meander if it is  constrained to come back to the origin at its right extremity 
and converges towards a normalized Brownian excursion otherwise.  Finally, the critical regime $\gep=\gep_c$ is characterized by a number of contacts between the polymer and the $x$-axis that grows as $\sqrt{N}$. The rescaled 
path converges to a reflected Brownian motion when there is no constraint on its right extremity and towards a reflected Brownian bridge otherwise. We note finally that similar results have been obtained in \cite{S14} when the pinning of the path occurs  at a  layer of finite width  on top of the hard wall. 

Before  comparing the infinite volume limit description of the wetting transition with that of the 
collapse transition, 
 let us insist on the fact that the nature of these two phase transitions  are fundamentally different and this can be explained in a few words. 
For the wetting model, the saturated phase for which the free energy is trivial (=0)
corresponds to the polymer being fully delocalized off the interface which means that entropy completely takes over in the energy-entropy competition
that rules such systems. For the IPDSAW in turn, the saturated phase is characterized by a domination of trajectories that 
are maximizing the energy. In other words, we could say that both models display a saturated phase 
which in the pinning case is associated with a maximization of the entropy, whereas it is associated with a maximization of the energy for the polymer collapse.

 As a consequence, only the extended regime of the IPDSAW and the 
localized regime of the wetting model may be compared. 
In both cases, one can indeed decompose the trajectory into simple patterns, that do not interact with each other and are typically of finite length,
i.e., the excursions off the $x$-axis for the wetting model and the pieces of  path in between two consecutive vertical stretches of length $0$ for the IPDSAW: these patterns can 
be seen as independent building blocks of the path and can be associated with a positive recurrent renewal. However the comparison can not be brought any further, since even in this regime the envelopes of the 
IPDSAW display a Brownian limit whereas the limiting object is the null function for the wetting model.

Due to the convergence of both envelopes towards deterministic Wulff shapes, the collapsed IPDSAW  may be related to other models in Statistical Mechanics that are known to undergo convergence of interfaces 
towards deterministic Wulff shapes. This is the case for  instance when considering a 2 dimensional  bond percolation model in its percolation regime and  conditioned on the existence of an open  curve of the dual graph around the origin with 
a prescribed and large area enclosed inside the curve (see \cite{A01}).   A similar interface appears when considering  the $2$-dimensional Ising model in a big square box of size $N$ at low temperature with no external field and $-$ boundary conditions and when conditioning the total magnetization to deviate from its 
average (i.e., $-m^* N^2$ with $m*>0$) by a factor  $a_N\sim N^{4/3+\delta}$ ($\delta>0$). It has been proven in  \cite{DKS92}, \cite{I94}, \cite{I95}  and \cite{IS98}  that such a deviation    is typically due to a unique large droplet of $+$, whose boundary converges to a deterministic Wulff shape once 
rescaled by $\sqrt{a_N}$.

However, the closest relatives to the collapsed IPDSAW are probably  the 1-dimensional SOS model with a prescribed large area below the interface (see \cite{DH96}) and the 2-dimensional Ising interface separating the $+$ and $-$ phases in a
vertically infinite layer of finite width (again with a large area underneath the interface, see \cite{DH97}). For both models and in size $N\in \N$, the law of the interface  can be related to the law of an underlying random walk $V$ conditioned on describing an abnormally large \emph{algebraic} area ($q N^2$ with $q>0$). As a consequence, once rescaled in time and space by $N$ the interface converges in probability towards 
a Wulff shape, whose formula depends on $q$ and on the random walk law.  The fluctuations of the interface around this deterministic shape  are of order $\sqrt{N}$  and their limiting distribution is identified in \cite[theorem 2.1]{DH96} for SOS model and in \cite[theorem 3.2]{DH97}  for Ising interface at sufficiently low temperature. The proofs in \cite{DH96} and \cite{DH97} use 
an ad-hoc tilting of the random walk law (described in Section \ref{sec:lcltp}), so that the large area becomes typical under the tilted law. In this framework, a local limit theorem can be derived for any finite dimensional distribution of $V$ under
the tilted law.

In the present paper, our system also enjoys a random walk representation (see section \ref{sec:rep}) and we will use the "large area"   tilting of the random walk law as well to prove Theorem \ref{sec:rep}. However, our model displays three particular features 
that prevent  us from  applying the results of \cite{DH96} straightforwardly. 
First, the conditioning on the auxiliary random walk $V$ is, in our case, related to the \emph{geometric} area below the path 
rather than to the \emph{algebraic} area (see Remark \ref{rtech}). 
Second, the horizontal extension of an IPDSAW path fluctuates, which is not the case for SOS model. Thus,  the ratio $q$
of the area below the path divided by the square of its horizontal extension fluctuates as well, which forces us to 
display some uniformity in $q$ for every local limit theorem we state in Section \ref{sec:oscill}.  Third and last, the fact that an IPDSAW path is characterized by two envelopes 
makes it compulsory to study simultaneously the fluctuations of $V$ around the Wulff shape and the fluctuations of $M$ around the $x$-axis (recall \eqref{cvpp}). 
We recall that the increments of $M$ are obtained by switching the sign of every second increment of $V$. 
As a consequence, we need to adapt, in Section \ref{sec:lcltp}, the proofs of the finite dimensional convergence and of the tightness displayed in \cite{DH96}.

Let us conclude with the critical regime of IPDSAW that is studied in Section \ref{sec:critic}.  The random walk representation described in Section \ref{sec:rep} and the 
fact that $\Gamma_\beta=1$ when $\beta=\beta_c$ tells us that the horizontal extension of the path has the law of 
the stopping time  $\tau_L:=\min \{N\geq 1\colon G_N(V)\geq L-N\}$ when $V$ is sampled from $\bP_\beta(\cdot\, |\, V_{\tau_L}=0, G_{\tau_L}(V)=L-\tau_L)$. 
Studying the scaling of $\tau_L$ requires to build up a renewal process based on the successive  excursions made by  $V$ inside the lower half-plane  or inside the upper half-plane. 
A sharp local limit theorem is therefore required
for the area enclosed in between such an excursion and the $x$-axis and this is precisely the object of a recent paper by Denisov, Kolb and  Wachtel in \cite{DKW13}.  Note that,  the fact that the horizontal extension fluctuates
makes the scaling limits of the upper and lower envelopes of the path much harder to investigate at $\beta=\beta_c$.
As soon as 
the limiting distribution of the rescaled horizontal extension is not constant, which is the case for the critical IPDSAW,  
one should indeed consider the limiting distribution of  the rescaled horizontal extension and of $\tilde V$ and $\tilde M$ simultaneously. For instance, the asymptotic decorelation of $\tilde M$ and $\tilde V$ may well not be true anymore. For this reason, we will state the investigation of the limiting distribution of the upper and lower envelopes of the critical IPDSAW as an open problem.
Let us conclude by pointing out that the critical regime of a Laplacian (1+1)-dimensional pinning model that is investigated by Caravenna and Deuschel in \cite{CD09} has somehow a similar flavor. More precisely, when the pinning term $\gep$ is switched off, the path 
can be viewed as the bridge of an integrated random walk and therefore scales like $N^{3/2}$. This scaling persists until
$\gep$ reaches a critical value $\gep_c$.  At criticality,   
and once rescaled in time by $N$ and in space $N^{\frac32}/\log^{\frac52}(N)$ the path is seen as the density of a signed measure $\mu_N$ 
on $[0,1]$. Then, $\mu_N$ that is build with a path sampled from the polymer measure, converges in distribution  towards a random atomic  measure on $[0,1]$. The atomes of the limiting distribution are generated by the longest excursions of the integrated  
walk, very much in the spirit of the limiting distribution in Theorem \ref{pfa2} (2) where each contribution to the sum constituting the limiting distribution is associated to a long excursion of the auxiliary random walk.
\subsubsection{\bf Computer Simulations} As explained in the Appendix~\ref{sec:appendice}, the
representation formula \eqref{partfun} provides an exact simulation
algorithm for the law of a path under the polymer measure
$P_{\beta,L}$. However, this algorithm is very efficient only for
$\beta=\beta_c$, and looses all efficiency when $\beta$ is not close
to $\beta_c$.
\subsubsection{\bf Open problems}
\begin{itemize} 
\item Find the scaling limit of the envelopes of the path in the critical regime.
\item Establish the fluctuations of the envelopes around the Wulff shapes (Theorem \ref{Oscill})
for the true polymer measure $P_{L,\beta}$ rather that for the mixture $\tilde P_{L,\beta}$.
\item Establish a Central Limit Theorem and a Large Deviation Principle for the horizontal extensions in the collapsed
and extended regimes.
\item Devise a dynamic scheme of convergence of measures on paths such
  that the equilibrium measure is the polymer measure, and with a
  sharp control on the mixing time to equilibrium similar to the one
  devised for S.O.S by Caputo, Martinelli and Toninelli in \cite{CapMarTon12}.

\end{itemize}

\section{Preparation}\label{sec:prep}
We begin this section by recalling  the proof of the probabilistic representation of the partition function (recall \ref{partfun}). This proof was already displayed in \cite{NGP13}  but since it constitutes  the starting point of our analysis it is worth reproducing  it here briefly. Moreover, we obtain as a by product the auxiliary random walk $V$ of law $\bP_\beta$, which, under an appropriate conditioning, 
can be used to derive some path properties under the polymer measure.
In Section \ref{sec:lcltp}, we recall how to work with the random walk $V$ (of law $\bP_\beta$)  conditioned on describing an abnormally large algebraic area. 
To that aim, we introduce a strategy initially displayed in \cite{DH96} and subsequently used in \cite{CNGP13} which consists in tilting  
$\bP_\beta$ appropriately such that the path typically describes a large area. This framework will be of crucial importance
to study the collapse regime.

\subsection{Probabilistic representation of the  partition function}\label{sec:rep}
 Let $\mathbf{P}_{\beta}$ be the law of the random walk $V:=(V_n)_{n\in\mathbb{N}}$ satisfying
$V_0=0$, $V_n=\sum_{i=1}^n U_i$ for $n\in\mathbb{N}$ and $(U_i)_{i\in\mathbb{N}}$ is an i.i.d sequence of geometric increments, 
i.e.,  
\begin{equation}\label{lawP}
\mathbf{P}_{\beta}(U_1=k)=\tfrac{e^{-\frac{\beta}{2}|k|}}{c_{\beta}}\quad\forall k\in\mathbb{Z}\quad\text{with}\quad c_{\beta}:=\tfrac{1+e^{-\beta/2}}{1-e^{-\beta/2}}.
\end{equation}
For $L\in \N$ and $N\in \{1,\dots,L\}$ we recall that  
$$\cV_{N,L}:=\{V\colon G_N(V)=L-N, V_{N+1}=0\} \quad \text{with} \quad  G_N(V)=\textstyle \sum_{i=0}^{N} |V_i|$$
 and (see Fig. \ref{fig:transform})
we denote by $T_N$ the  one-to-one correspondence that maps 
$\cV_{N+1,L-N}$ onto $\cL_{N,L}$ as 
\be{defTN}
T_N(V)_i=(-1)^{i-1} V_i\quad  \text{for all} \quad i\in \{1,\dots N\}.
\ee
%
%

Coming back to the proof of \eqref{partfun} we recall (\ref{defLL}--\ref{pff}) and we note that the $\tilde{\wedge}$ operator can be written as
\begin{equation}
x\;\tilde{\wedge}\;y=\left(|x|+|y|-|x+y|\right)/2,\quad\forall x,y\in\mathbb{Z}.
\end{equation}
Hence, for $\beta>0$ and $L\in\mathbb{N}$, the partition function in \eqref{pff}  becomes
\begin{align}\label{ls}
\nonumber Z_{L,\beta}
&=\sum_{N=1}^{L}\sum_{\substack{l\in\mathcal{L}_{N,L}\\l_0=l_{N+1}=0}}\exp{\Bigl(\beta\sum_{n=1}^N{|l_n|}-\tfrac{\beta}{2}\sum_{n=0}^N{|l_n+l_{n+1}|}\Bigr)}\\
&=c_\beta\, e^{\beta L} \sum_{N=1}^{L}\left(\tfrac{c_\beta}{e^\beta}\right)^N\sum_{\substack{l\in\mathcal{L}_{N,L}
\\ l_0=l_{N+1}=0}}\prod_{n=0}^{N}\frac{\exp{\Bigl(-\tfrac{\beta}{2}|l_n+l_{n+1}|\Bigr)}}{c_\beta}.
\end{align}

Then, since for $l\in \cL_{N,L}$  the increments $(U_i)_{i=1}^{N+1}$ of $V=(T_N)^{-1}(l)$  in \eqref{defTN} necessarily satisfy $U_i:=(-1)^{i-1}(l_{i-1}+l_i)$, one can 
rewrite  \eqref{ls} as
\begin{equation}\label{tgh}
Z_{L,\beta}=c_\beta e^{\beta L} \sum_{N=1}^{L} \left(\tfrac{c_\beta}{e^\beta}\right)^N \sum_{V\in \cV_{N+1,L-N}} \mathbf{P}_{\beta}(V),
\end{equation}
which immediately implies \eqref{partfun}. 
A useful consequence of formula \eqref{tgh} is that, once conditioned on taking a given number of horizontal steps $N$, the polymer measure is exactly the image measure by the $T_N-$transformation of the geometric random walk $V$ conditioned to return to the origin after N+1 steps and 
to make a geometric area $L-N$, i.e.,  
\begin{equation}\label{transf2}
P_{L,\beta}\bigl(l\in\cdot \mid N_l=N\bigr)=\mathbf{P}_\beta\bigl(T_N(V)\in\cdot \mid V_{N+1}=0, G_N=L-N\bigr).
\end{equation}

\medskip

\begin{figure}[ht]\center
	\includegraphics[width=.65\textwidth]{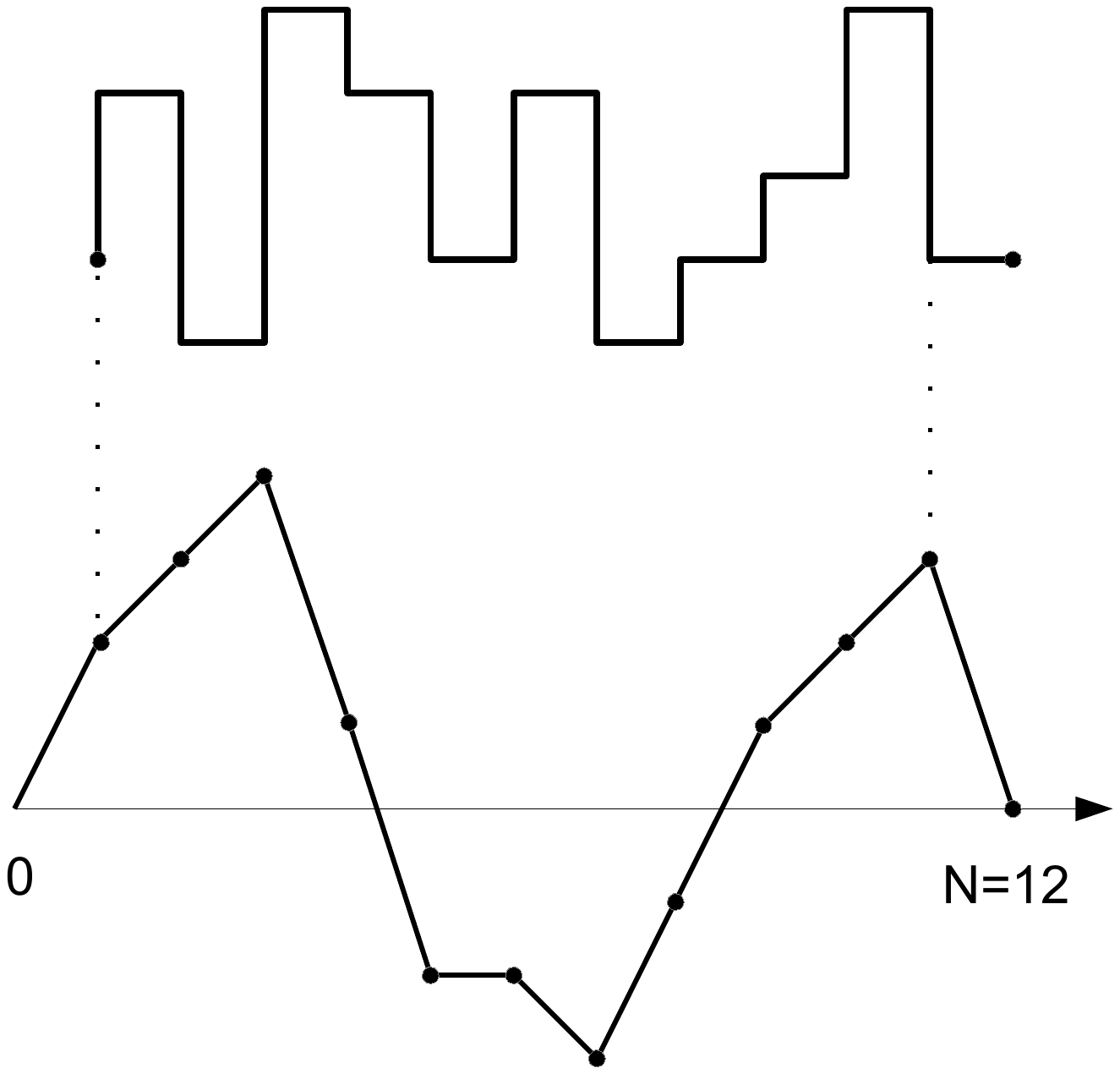}
	\vspace{-8cm}
	\caption{An example of a trajectory $l=(l_i)_{i=1}^{11}$ with $l_1=2, l_2=-3, l_4=4$\dots is drawn on the  upper picture.  The auxiliary random walk $V$ associated with $l$, i.e., $(V_i)_{i=0}^{12}=(T_{11})^{-1}(l)$  is drawn on the lower picture with 
	$V_1=2, V_2=3, V_4=4\dots$}
	\label{fig:transform}
\end{figure}

\medskip

\subsection{Large deviation estimates}\label{sec:lcltp}
In this section, we  introduce  an \textit{exponential tilting} of the probability measure $\mathbf{P}_\beta$ (through the Cramer transform), in order to study $\bP_\beta$ conditioned on the large deviation event $\{A_n(V)=q n^2, V_n=0\}$. Following  Dobrushin and Hryniv in \cite{DH96}, for $n\in\mathbb{N}$, we use
\begin{equation}\label{eq:Yn}
\tfrac{1}{n} A_n(V):=\tfrac{1}{n}(V_0+V_1+\dots+V_{n-1}),
\end{equation}
 Under the tilted probability measure the large deviation event $\{A_n=n^2 q,\,V_n=0\}$ becomes typical.
First, we denote by $L(h),h\in\mathbb{R}$ the logarithmic moment generating function of the random walk $V$, i.e,
\begin{equation}\label{defL}
\Ll(h):=\log\mathbf{E}_\beta[e^{h U_1}].
\end{equation}
From the definition of the law $\mathbf{P}_\beta$ in \eqref{lawP}, we obviously have $L(h)<\infty$ for all $h\in(-\beta/2,\beta/2)$. 
For the ease of notations, we set $\Lambda_n:=(\frac{A_n}{n},V_n)$ and we denote its logarithmic moment generating function by $\Llamn(H)$ for $H:=(h_0,h_1)\in\mathbb{R}^2$, i.e.,
\begin{equation}\label{eq:LlambdaN}
{\textstyle \Llamn(H):=\log\mathbf{E}_\beta\bigl[e^{h_0\frac{A_n}{n}+h_1V_n}\bigr]=\sum_{i=1}^n \Ll\Bigl(\bigl(1-\tfrac{i}{n}\bigr)h_0+h_1\Bigr).}
\end{equation}
Clearly, $\Llamn(H)$ is finite for all $H\in \cD_n$ with
\begin{equation}
{\textstyle\mathcal{D}_n:=\Bigl\{(h_0,h_1)\in\mathbb{R}^2\colon h_1\in\bigl(-\tfrac{\beta}{2},\tfrac{\beta}{2}\bigr),\ (1-\tfrac{1}{n})h_0+h_1\in\bigl(-\tfrac{\beta}{2},\tfrac{\beta}{2}\bigr)\Bigr\}.}
\end{equation}
We also introduce $\Llam$ the continuous counterpart of $L_{\Lambda_n}$ as
\begin{equation}\label{eq:LLam}
{\textstyle \Llam(H):=\int_0^1 \Ll(xh_0+h_1)dx,}
\end{equation}
which is defined on
\begin{equation}\label{eq:Llambda}
{\textstyle\mathcal{D}:=\Bigl\{(h_0,h_1)\in\mathbb{R}^2\colon h_1\in\bigl(-\tfrac{\beta}{2},\tfrac{\beta}{2}\bigr),\ h_0+h_1\in\bigl(-\tfrac{\beta}{2},\tfrac{\beta}{2}\bigr)\Bigr\}.}
\end{equation}
With the help of \eqref{eq:LlambdaN} and for $H=(h_{0},h_{1})\in\mathcal{D}_n$, we define the $H$-tilted distribution by
\begin{equation}\label{eq:defH}
\frac{\text{d}\mathbf{P}_{n,H}}{\text{d}\mathbf{P}_{\beta}}(V)=e^{h_{0}\frac{A_n}{n}+h_{1}V_n-\Llamn(H)}.
\end{equation}
For a given $n\in \N$ and $q\in \frac{\N}{n}$, the exponential tilt is
given by $H_n^q:=(h_{n,0}^q,h_{n,1}^q)$ which, by Lemma 5.5 in Section
5.1 of \cite{CNGP13}, is the unique solution of  
\begin{equation}\label{eq:tildeEC}
\mathbf{E}_{n,H}(\tfrac{\Lambda_n}{n})=\nabla\bigl[\tfrac{1}{n}\Llamn\bigr](H)=(q,0),
\end{equation}

An important feature of this exponential tilting is that $P_{n,H_n^q}=\otimes_{i=1}^n \mathbf{\nu}_{\,h_{n}^{i}}$ with 
$h_n^i=(1-\frac{i}{n}) h_{n,0}^q+h_{n,1}^q$. Consequently, under $P_{n,H_n^q}$, the random walk increments
$(U_i)_{i=1}^n$ are independent and for each $j\in \{1,\dots,n\}$,  the law of $U_j$ is 
$\nu_{\,h_j^n}$ where
\begin{equation}\label{defnu}
\frac{\text{d} \nu_{\,h}}{\text{d}\mathbf{P}_{\beta}}(v_1)=e^{h v_1-\Ll(h)}.
\end{equation}

\noindent Then, we define the continuous counterpart of $H_n^q$ by 
 $\tilde{H}(q,0):=(\tilde{h}_0(q,0),\tilde{h}_1(q,0))$ which is  the unique solution of the equation
\begin{equation}\label{eq:deftil}
\nabla \Llam(H)=(q,0),
\end{equation}
and we state a Proposition that allows us to remove the $n$ dependence of the exponential decay rate.

\section{Scaling Limits in the extended phase}\label{extended}

In this section, we will first display an alternative representation
of the model in terms of an auxiliary regenerative process. Lemmas
\ref{compfhatftilde} and \ref{dishexpattern} are linking the original
model and the regenerative process. The proof of Lemma
\ref{compfhatftilde} is omitted for sake of conciseness. Section \ref{prt1} is dedicated to the proof of 
Theorems \ref{pfa} (1) and \ref{pfa2} (1) and of Theorem \ref{extph}.

\medskip

We shall restrict ourselves to paths of length $L$ whose last stretch has a zero vertical
length, i.e.,   $\Omega_L^c =\ens{l \in
  \Omega_L : l_{N_l}=0}$. Note that the natural one-to-one correspondence between $\Omega_L$
and $\Omega_{L+1}^c$ conserves the Hamiltonian and therefore, proving Theorem \ref{pfa} (1) or \ref{pfa2} (1) or Theorem 
\ref{extph} with the constraint immediately entails the same results without this constraint.

Let us define a \emph{pattern} as a path whose first zero length vertical stretch
occurs only at the end of the path. We shall decompose
a path into a finite number of patterns.  That is if for
$l\in\Omega_L^c$ we consider the successive indices corresponding to vertical stretches
of zero length, i.e.,  
$$ T_0=0, T_{k+1}(l) =\inf\ens{j \ge 1+ T_k : l_j=0}\,.$$
Then $\Nu_k=T_k -T_{k-1}$ is the horizontal extension of the $k$-th
pattern,  $S_k=\Nu_k + \valabs{l_{T_{k-1}+1}} + \cdots
+\valabs{l_{T_k}}$ is the length of the $k$-th pattern and $\Ju_k=l_{T_{k-1}+1}+\dots+l_{T_k}$ is the vertical displacement on the $k$-th pattern. If
$\pi_L(l)=r$ is the number of patterns, then the horizontal extension
is $N_l=\Nu_1 + \cdots + \Nu_r$, the total length is of course
$L= S_1 + \cdots + S_r$ and the total vertical displacement is $\Ju_1 + \cdots + \Ju_r$.

The key observation that will lead to the construction of the renewal
structure, is that  the Hamiltonian of the path is the
sum of the Hamiltonian of the patterns, since the separating two
horizontal steps prevent any interaction between the patterns.

Let us define the constrained excess partition function as
\be{defzcp} 
 \tilde{Z}^{c}_{L,\beta} := e^{-\beta L}
\sum_{l\in \Omega_L^c} \mathbf{P}_L(l)\,  e^{H_{L,\beta}(l)}.
\ee
We shall apply to $\tilde{Z}^{c}_{L,\beta} $ the probabilistic representation displayed in (\ref{ls}--\ref{tgh}).  The only additional constraint is that
$l_{N_l}=0$, which we translate immediately as the added constraint $V_N=0$ on the associated random walk $V$, i.e.,
\begin{align}\label{eq:impa}
\nonumber   \tilde{Z}^{c}_{L,\beta}&= c_\beta \sum_{N=1}^L (\Gamma_\beta)^N\, 
  \probbeta{G_N(V)=L-N, V_N=V_{N+1}=0}\\
&= \sum_{N=1}^L (\Gamma_\beta)^N\, 
  \probbeta{G_N(V)=L-N, V_N=0},
\end{align}
since 
$\probbeta{V_{N+1}-V_N=0}=\probbeta{V_1=0}=1/c_\beta$.
Accordingly the pattern excess partition function is defined as
\begin{align}\label{defzcpat} 
\nonumber\hat{Z}^{c}_{L,\beta} :&= \tilde{Z}^{c}_{L,\beta}(T_1(l)=N_l)= 
e^{-\beta L}
\sum_{l\in \Omega_L^c} \mathbf{P}_L(l)\,  e^{H_{L,\beta}(l)} \, 1_{\{T_1(l)=N_l\}}\\
&= \sum_{N=1}^L (\Gamma_\beta)^N\, 
  \probbeta{G_N(V)=L-N, T(V)=N}
\end{align}
where, for $V\in \Z^{\N_{0}}$ such that $V_0=0$, we  set
$T(V)=\inf\{i\geq 1\colon\, V_i=0\}$.
For the associated random walk trajectory $V$, the  vertical displacement is given by
 $Y_n(V):=\sum_{i=1}^n
(-1)^{i-1} V_i$ for $n\in \N$.

We will use the decomposition into patterns to generate an auxiliary renewal process, whose inter-arrivals are associated with the successive lengths of the patterns. Thus, it is natural to consider the series
\be{defphialpha}
\phi(\alpha) := \sum_{t\ge 1} \hat{Z}^{c}_{t,\beta}e^{-\alpha t} \in
]0,+\infty]\,
\ee
and the convergence abcissa
$\hat{f}(\beta) := \inf\ens{\alpha : \phi(\alpha) < +\infty}$. An important observation at this stage is the link between $\phi$ and $\tilde f(\beta)$ that is stated in the following lemma. 
\bl{compfhatftilde}
In the extended phase we have $0 < \hat{f}(\beta) <
\tilde{f}(\beta)$ and moreover $\phi(\tilde{f}(\beta))=1$.
\el

Lemma \ref{compfhatftilde} allows us to  define rigorously the renewal process. We even 
 enlarge the probability space on which this renewal process is defined to take into account the horizontal extension and the vertical displacement 
on each pattern.  We finally obtain an auxiliary regenerative process 
that will be the cornerstone of our study of the extended phase. To that aim, 
we let    $(\sigma_i,\nu_i,y_i)_{i\ge 1}$ be an IID sequence of random
variables of law $\PPbeta$. The law of $(\sigma_1,n_1,y_1)$ is given by 
\begin{itemize}\label{defpb}
\item $ \PPbeta(\sigma_1=s) = \hat{Z}^{c}_{s,\beta}\,  e^{-s
  \tilde{f}(\beta)}$, \quad for $s\geq 1\,.$\\ 
\item The conditional distribution of $\nu_1$ given $\sigma_1=s$ is (recall \eqref{defzcpat})
$$ \PPbeta(\nu_1=n\mid \sigma_1=s)=
\frac{1}{\hat{Z}_{s,\beta}^{c}} \, (\Gamma_\beta)^n\, \probbeta{G_n(V)=s-n,T(V)=n} \qquad(1\le n\le s)\,.$$
\item The conditional distribution of $y_1$ given $\sigma_1=s, \nu_1=n$ is 
$$ \PPbeta(y_1=t \mid \sigma_1=s,\nu_1=n)=
\probbeta{Y_n(V)=t \mid  G_n(V)=s-n,T(V)=n}
\qquad(t\in \Z).$$
\end{itemize}
\medskip

The link between the latter regenerative process and the polymer law is stated in Lemma \ref{dishexpattern} below.
We let  $\cT$ be  the set of renewal times associated to $\sigma$, i.e., $\cT=\{\sigma_1 + \cdots + \sigma_r,r\in \N\}$. 

\bl{dishexpattern}
Given integers $r, s_1, \ldots, s_r, n_1, \ldots, n_r,t_1,\dots,t_r$ such that
$s_i\ge 1$, $1\le n_i\le s_i$, $s_1 + \cdots + s_r=L$, we have
$$P_{L,\beta}^{c}\etp{(S_i,\Nu_i,\Ju_i)=(s_i,n_i,t_i), 1\le i\le r} =
\PPbeta\big((\sigma_i,\nu_i,y_i)=(s_i,n_i,t_i), 1\le i\le r\mid L\in \cT\big)\,.
$$
\el
\begin{proof}
We disintegrate $\tilde{Z}^{c}_{L,\beta}$ with respect to the number of patterns $r$ and to
$s_1, \ldots, s_r$ the respective lengths of these patterns:
\be{decompozcontraint}
\tilde{Z}^{c}_{L,\beta} = \sum_{r=1}^{L/2} \sum_{\substack{s_i\ge 1\\ s_1 +
  \cdots+s_r=L}} \prod_{i=1}^r \hat{Z}^{c}_{s_i,\beta}\,.
\ee
It is now folklore in Probability theory (see \cite[Chapter 1]{cf:Gia}, for an application of this technique to the linear pinning model) to multiply and divide the r.h.s. by $e^{L \tilde f(\beta)}$ and obtain
\be{tre}
\tilde{Z}^{c}_{L,\beta} = e^{\tilde f(\beta)  L}\,  \sum_{r=1}^{L/2} \sum_{\substack{s_i\ge 1\\ s_1 +
  \cdots+s_r=L}} \prod_{i=1}^r  \PPbeta(\sigma_1=s_i)=\PPbeta(L\in \cT) \, e^{\tilde{f}(\beta) L}.
\ee
We use the probabilistic representation of the partition function, and we let 
$\cL^{n_.,s_.,t_.}$ be the subset of $\Omega_L^c$ containing those configurations
forming $r$ patterns of respective horizontal extensions, lengths and  horizontal displacements $n_1,
  \ldots, n_r, s_1, \ldots ,s_r, t_1,\dots,t_r$. 
We obtain   
\begin{align*}
    P_{L,\beta}^{c} ((S_i,\Nu_i,&\Ju_i)=(s_i,n_i,t_i), 1\le i\le r)=
    \unsur{\tilde{Z}^{c}_{L,\beta}} c_\beta\,  (\Gamma_\beta)^{n_1 +
      \cdots + n_r} \sum_{l\in \cL^{n_.,s_.,t_.}} \prod_{i=0}^{n_1 + \cdots +n_r} \frac{e^{-
        \frac{\beta}{2} \valabs{l_i+ l_{i+1}}}}{c_\beta} 
        \end{align*}
and we  recall that  for $l\in   \cL^{n_.,s_.,t_.}$ and $i\in \{1,\dots,r\}$,  the vertical displacement on the $j$th pattern
is $l_{n_1+\dots+n_{i-1}+1}+\dots+l_{n_1+\dots+n_{i}}$.  For the associated random walk trajectory $V$, the  vertical displacement is given by
$(-1)^{n_1+\dots+n_{i-1}} Y_{n_i}(V)$. It remains to use the symetry of the $V$-random walk to get
 \begin{align*}       
 P_{L,\beta}^{c} & ((S_i,\Nu_i,\Ju_i)=(s_i,n_i,t_i), 1\le i\le r) \\
 &=\unsur{\tilde{Z}^{c}_{L,\beta}}  (\Gamma_\beta)^{n_1 +
      \cdots + n_r} \prod_{i=1}^r
    \probbeta{Y_{n_i}(V)=(-1)^{n_1+\dots+n_{i-1}} t_i, G_{n_i}(V)=s_i-n_i,   T(V)=n_i},\\
 &   =\unsur{\tilde{Z}^{c}_{L,\beta}}  (\Gamma_\beta)^{n_1 +
      \cdots + n_r} \prod_{i=1}^r
    \probbeta{Y_{n_i}(V)= t_i, G_{n_i}(V)=s_i-n_i,   T(V)=n_i}.
    \end{align*}
We multiply the numerator and the denominator by $e^{-\tilde{f}(\beta) L}$ and we use 
\eqref{defzcpat}, \eqref{tre} and the definition of $\PPbeta$ to get
 \begin{align*}   
  P_{L,\beta}^{c} ((S_i,\Nu_i,\Ju_i)=(s_i,n_i,t_i), 1\le i\le r)&= \unsur{\PPbeta(L\in\cT)}\prod_{i=1}^r \PPbeta((\sigma_1,\nu_1,y_1)=(s_i,n_i,t_i)).
  \end{align*}
\end{proof}

\subsection{Proof of Theorems \ref{pfa} (1), \ref{pfa2} (1) and of Theorem  \ref{extph}}\label{prt1}

The proof of Theorem \ref{pfa} (1) is a straightforward application of formula  \eqref{tre} and of the renewal Theorem which ensures us that $\lim_{L\to \infty} \PPbeta(L\in \tau)\to 1/\mu_\beta>0$ with $\mu_\beta:=\EEbeta(\sigma_1)$. The finiteness of $\mu_\beta$ is an easy consequence of the definition of $\PPbeta$ and of the fact that $\hat{f}(\beta)< \tilde{f}(\beta)$. 

The proof of \eqref{cvhex} (i.e., Theorem \ref{pfa2} (1)) is performed as follows. We let $\pi_L(l)$ be the \emph{number of patterns} in a sequence
$l\in\Omega_L^c$ (and thus $T_{\pi_L(l)}=N_l$). We set also $\pi_L(\sigma):=\max\{i\geq 1\colon\,\sigma_1+\dots+\sigma_i\leq L\}$, such that $\pi_L(\sigma)$ is the counterpart of $\pi_L(l)$ for the  renewal process associated with the interarrivals  $(\sigma_i)_{i\in \N}$. 
By Lemma \ref{dishexpattern}, \eqref{cvhex} will be proven once we show that, under $\PPbeta(\cdot\mid L\in \cT)$, we have 
\be{sqfun}
\lim_{L\to \infty} \frac{\nu_1 + \cdots +
  \nu_{\pi_L(\sigma)}}{L}=_{\text{Prob}}
\frac{\EEbeta(\nu_1)}{\mu_\beta},
\ee
 and therefore the quantity $e(\beta)$ in \eqref{cvhex} is given by  $\frac{\EEbeta(\nu_1)}{\mu_\beta}$.

To prove \eqref{sqfun}  we note that, under $\PPbeta$, a straightforward application of the law of large number gives the almost sure convergence of  $(\nu_1 + \cdots +
  \nu_{n})/n$  to $\EEbeta(\nu_1)$ and of $\pi_L(\sigma)/L$ to 
  $1/\mu_\beta$.  Thus $(\nu_1 + \cdots +
  \nu_{\pi_L(\sigma)})/L$ tends almost surely to $\EEbeta(\nu_1)/\mu_\beta$ and this convergence
  also holds in probability.  Moreover, we have just seen that $\lim_{L\to \infty} \PPbeta(L\in \tau)\to 1/\mu_\beta>0$ so that the latter convergence in probability also holds under $\PPbeta(\cdot\mid L\in \cT)$ and this completes the proof of \eqref{cvhex}.


\medskip

It remains to prove \eqref{cvp}. To begin with, we show that, under $P_{L,\beta}^{c}$
the largest stretch of a given configuration $l\in \Omega_L^c$ is typically not larger than $c\log L$.
This will imply the convergence in probability of $(\sqrt{N_l} \ \tilde{l}_s)_{s\in [0,1]}$ to $0$, which is the convergence of the second coordinate in \eqref{cvp}. 

\bl{lbound} 
For $\beta< \beta_c$, there exists a $c>0$ such that 
\be{lboun}
\lim_{L\to \infty} P_{L,\beta}^{c} \big(\max\{|l_i|, i=1,\dots,N_l\} \geq c \log L\big)=0.
\ee
\el
\begin{proof}
We will prove a slightly stronger property, that is there exists a $c>0$ such that
\be{lboun2}
\lim_{L\to \infty} P_{L,\beta}^{c} \big(\max\{|S_i|, i=1,\dots,\pi_L(l)\} \geq c\log L\big)=0.
\ee
The fact that each stretch of a given configuration $l\in \Omega_L^c$ belongs to one of the $\pi_L(l)$
patterns of $l$ will then be sufficient to obtain \eqref{lboun}. With the help of Lemma \ref{dishexpattern}, we can state that
\begin{align}\label{tdf}
\nonumber P_{L,\beta}^{c}(\max\{|S_i|, i=1,\dots,&\pi_L(l)\} \geq c\log L) \\
\nonumber &=
\PPbeta(\max\{|\sigma_i|, i=1,\dots,\pi_L(\sigma)\} \geq c\log L \mid L\in \cT)\\
 &\leq \tfrac{1}{\PPbeta(L\in \cT)} \PPbeta(\max\{|\sigma_i|, i=1,\dots,L\} \geq c\log L).
\end{align}
By the renewal theorem, we know that  $\lim_{L\to \infty} \PPbeta(L\in\cT)=
1/\mu_\beta>0$. Moreover, since $\hat{f}(\beta)< \tilde{f}(\beta)$, under $\PPbeta$, the IID sequence $(\sigma_i)_{i=1}^\infty$ 
has finite small exponential moments and therefore, we can choose $c>0$ large enough so that
$\lim_{L\to \infty} \PPbeta(\max\{|\sigma_i|, i=1,\dots,L\} \geq c\log L)=0$. Thus, by choosing $c>0$ large enough, the r.h.s. in \eqref{tdf} vanishes as $L\to \infty$ and the proof of Lemma \ref{lbound} is complete.
\end{proof}

At this stage, the proof of \eqref{cvp} will be complete once we  prove that, with $l$ sampled from $P_{L,\beta}^{c}$, we have
\be{convdim}
\sqrt{N_l}\,  \tilde M_l \xrightarrow[L\to \infty]{d}  \sqrt{\frac{v_\beta}{\EEbeta(\nu_1)}} \,  \big( B_s\big)_{s\in [0,1]},
\ee 
with $v_\beta=\EEbeta(y_1^2)$. For simplicity we will rather deal with the process $\widehat M_l= \frac{N_l}{\sqrt{L}} \tilde M_l$ and since, by Theorem \ref{pfa2} (1) we know that $N_l/L$ converges in probability to $\frac{\EEbeta(\nu_1)}{\mu_\beta}$, we can claim that \eqref{cvp} will be proven once we show that, with $l$ sampled from $P_{L,\beta}^{c}$, the process
$\widehat{M}_l$ converges in law  towards $\sqrt{v_\beta/\mu_\beta} \, B$. 

Because of Lemma \ref{lbound}, we do not change the limit of $\widehat M_l$ under $P_{L,\beta}^{c}$ if, 
for $l\in \Omega_L^c$ and $i\in \{1,\dots, \pi_L(l)\}$, we redefine 
$M_{l,i}$ in \eqref{droitmi} as $M_{l,i}=l_1+\dots+l_i$. We will use this later definition until the end of this proof only. 
Then, we let $\bar M_l$ be the cadlag process defined  as
\be{defmbar}
\bar M_{l}(t):=\frac{1}{\sqrt L} \sum_{i=1}^{C_{1,l}(t)} \Ju_i= \frac{l_1+\dots+l_{T_{C_{1,l}(t)}}}{\sqrt L} , \quad t\in [0,1],
\ee
where $C_{1,l}(t)$ simply counts how many patterns have been completed by the trajectory $l$
before its $\lfloor t N_l\rfloor$th horizontal step, i.e., 
\be{defcunlt}
C_{1,l}(t):=\sum_{i=1}^{\pi_L(l)} 1_{ \{\Nu_1+\dots+\Nu_i\leq t N_l\}}.
\ee

In this section, all random processes are viewed as elements of  $D_{[0,1]}$ the set of cadlag processes on $[0,1]$ endowed with the Skorohod topology $\cD$ and we refer to \cite{Bi} for an overview on the subject.  At this stage we note that, for $t\in [0,1]$, we have
\begin{align}\label{diffpro}
|\bar M_{l}(t)-\widehat M_l(t)|&=\tfrac{1}{\sqrt{L}} \bigg| \sum_{i=T_{C_{1,l}(t)}+1}^{\lfloor t N_l\rfloor} l_i \bigg|\leq  \tfrac{1}{\sqrt{L}}  \sum_{i=T_{C_{1,l}(t)}+1}^{T_{C_{1,l}(t)+1}} |l_i \big|\leq \tfrac{1}{\sqrt{L}} \max\{|S_i|, i=1,\dots,\pi_L(l)\},
\end{align}
and therefore, \eqref{lboun2} ensures that $\widehat M_l$ and $\bar M_l$ have the same limit in law under $P_{L,\beta}^{c}$.

Let  $(\sigma_i,\nu_i,y_i)_{i\ge 1}$ be an IID sequence of random
variables under $\PPbeta$ and let $B:=(B_s)_{s\in [0,1]}$ be a standard $1$ dimensional Brownian motion independant from 
$(\sigma_i,\nu_i,y_i)_{i\ge 1}$. We note that, because of Lemma \ref{dishexpattern}, the cadlag process $\bar M_l$ under $P_{L,\beta}^{c}$ has the same law as
the process $\bar W_{L}$ under $\PPbeta(\cdot \mid L\in \cT)$ which is defined by 
\be{defwbar}
\bar W_{L}(t):=\frac{1}{\sqrt L} \sum_{i=1}^{C_{2,L}(t)} y_i , \quad t\in [0,1],
\ee
where $C_{2,L}(t)$ is the counterpart of $C_{1,l}(t)$ in the framework of the associated regenerative process, i.e.,
\be{defcdeuxlt}
C_{2,L}(t):=\sum_{i=1}^{\pi_L(\sigma)} 1_{ \{\nu_1+\dots+\nu_i\leq t \, \cV_L\}},
\ee
where we recall $\pi_L(\sigma)=\max\{i\geq 1\colon\,\sigma_1+\dots+\sigma_i\leq L\}$ and $\cV_L=\nu_1+\dots+\nu_{\pi_L(\sigma)}$.
Thus, the proof of \eqref{cvp} will be complete once we show that, under $\PPbeta(\cdot \mid L\in \cT), L\in \N$
\be{donskk}
\lim_{L\to \infty} \bar W_L=_{\text{Law}} \textstyle \sqrt{\tfrac{v_\beta}{\mu_\beta}} \, B\quad \text{with} \quad v_\beta=\EEbeta(y_1^2).
\ee
\medskip

 The proof of \eqref{donskk} is standard in regenerative process
 theory, see e.g. Section 5.10 of Serfozo \cite{Ser09}. We just need
 to be careful when transporting results from $\PPbeta$ to $\PPbeta(\cdot \mid L\in \cT)$.

\section{Fluctuation of the convex envelopes around the Wulff shape: proof of Theorem \ref{Oscill}}\label{sec:oscill}

Let us first recall some notations.  For each 
$l \in \cL_{N,L}$ we defined in  \eqref{droitmi} the middle line  $M_l$. 
We also defined in Section \ref{sec:rep}, the $T_N$ transformation that associates with each $l\in \cL_{N,L}$ the path $V_l=(T_N)^{-1}(l)$
such that $V_{l,0}=0$,  $V_{l,i}=(-1)^{i-1} l_i$ for all $i\in\{1,\dots,N\}$ and $V_{l,N+1}=0$.
Finally, note that 
the path $M_l$ can be rewritten with the increments $(U_i)_{i=1}^{N+1}$ of the $V_l$ random walk as
\be{defM}
M_{l,i}=\sum_{j=1}^{i} (-1)^{j+1} \frac{U_j}{2},\quad i\in \{1,\dots,N\}.\\
\ee
In the same spirit, we will need to work with the $V$ random walk sampled from $\bP_\beta$ directly. We associate with each 
$V$ trajectory the process $M$ that is obtained exactly  as $M_l$ is obtained from $V_l$ in \eqref{defM}, i.e.,
\be{defMM}
M_{i}=\sum_{j=1}^{i} (-1)^{j+1} \frac{U_j}{2},\quad i\in \N.
\ee
We recall finally that, for any trajectory $V=(V_i)_{i=0}^\infty\in \Z^\N$ and any $N\in \N$, $A_N(V)$ is the algebraic area  below the $V$-trajectory
up to time $N$ and  $G_N(V)$  is its geometric counterpart, i.e.,  
$A_N(V)=\sum_{i=1}^N V_i$  and $G_N(V)=\sum_{i=1}^N |V_i|$.

\medskip

\subsection{Outline of the proof}
We recall the definition of $\tilde P_{L,\beta}$ in \eqref{defPtil}. As stated in Remark \ref{remalt}, the proof of Theorem \ref{Oscill}
will be complete once we show the convergence in law \eqref{cvpp}.
To that aim we will prove Proposition \ref{llt} and \ref{tensip} below, that are a finite dimensional convergence and a tension argument 
which will be sufficient to prove \eqref{cvpp}

\medskip

Given $\bar{t}=(t_1, \ldots, t_{r_1})$ with $0< t_1 < \cdots < t_{r_1}
<1$, we let $g_{H,\bar{t}}(\bar{x})$, $\bar{x} \in \R^{r_1}$, be the
density of the Gaussian vector $\xi_H(\bar{t})=(\xi_H(t_1), \ldots,
\xi_H(t_{r_1}))$ and let $f_{H,\bar{t}}(z_0,z_1,\bar{x})$ be the
density of the Gaussian vector $(\int_0^1\xi_H(s)\, ds, \xi_H(1),\xi_H(t_1), \ldots,
\xi_H(t_{r_1}))$. Finally let 
$$f^c_{H,\bar{t}}(\bar{y}) = \frac{f_{H,\bar{t}}(0,0,\bar{y})}{\int
  f_{H,\bar{t}}(0,0,\bar{x})\, d\bar{x}}$$ be the density of the law
of $\xi_H(\bar{t})$ conditional on $\int_0^1\xi_H(s)\, ds=0=\xi_H(1)$.

\bp{llt}
For $\beta>\beta_c$ and $(r_1,r_2)\in \N^2$, consider $\bar s=(s_i)_{i=1}^{r_1}$ and $\bar t=(t_i)_{i=1}^{r_2}$  two ordered sequences in $[0,1]$. Set $m_L=a(\beta) \sqrt{L}$. We have that 
\begin{align}\label{llt1}
\lim_{L\to \infty} \sup_{(\bar x, \bar y) \in \Z^{r_1}\times\N^{r_2}} \Big| (m_L)^{\tfrac{r_1+r_2}{2}}  
 \widetilde{P}_{L,\beta}\big[ H_{\bar s,\bar t}(\bar x,\bar
 y)\big]-g_{\beta,\bar s}\Big(\tfrac{\bar
   x}{\sqrt{m_L}}\Big)\, f_{\beta,\bar t}\Big(\tfrac{\bar y}{\sqrt{m_L}},m_L \gamma^*_{\beta}\Big)\Big|=0
\end{align}
with  
\be{defH}
H_{\bar s,\bar t}(\bar x,\bar y)=\big\{  N_l\, \big(2\tilde M_l(\bar s),|\tilde V_l(\bar t)|\big)=(\bar x, \bar y)        \big\}\,,
\ee
\be{defgfbeta}
g_{\beta,\bar s}(\bar x) =
g_{\tilde{H}(q_\beta,0),\bar s}\quad\text{and}\quad f_{\beta,\bar t}(\bar
y,\phi)=\frac12(f^c_{\tilde{H}(q_\beta,0),\bar t}(\bar y -\phi(\bar
t))+ f^c_{\tilde{H}(q_\beta,0),\bar t}(\bar y +\phi(\bar t)))
\ee
\ep

\bp{tensip}
For $\beta>\beta_c$, the sequence of probability laws $(\widehat Q_{L,\beta})_{L\in \N}$
is tight.  
\ep

Let us give here the key idea behind the proofs of Propositions \ref{llt} and \ref{tensip}. We will first prove the counterpart of 
Propositions \ref{llt} and \ref{tensip} with the processes $\widetilde M$ and $\widetilde V$ sampled from 
$\bP_\beta$ (cf. \ref{defMM})   conditional on $V_N=0, A_N(V)= q N^2$  ($q>0$). The reason for these two intermediate  
results is that they can be obtained with the tilting of $\bP_\beta$ exposed in Section \ref{sec:lcltp} and first introduced by Dobrushin and Hryniv in \cite{DH96}. Then, we will translate these two results in terms of the two processes $\widetilde M_l$ and $\widetilde V_l$ (see \eqref{defM}) obtained with  $l$ sampled from $\tilde P_{L,\beta}$. However, this last step is difficult because the conditioning 
that emerges from the $T_N$ transformation (see section  \eqref{sec:rep}) involves the geometric area below the $V_l$ random walk rather than its algebraic counterpart. This is the reason why we state in Section \ref{prle}
a handful of preparatory Lemmas indicating that, when the algebraic area below $V$ is abnormally large ($q N^2$ instead of the typical $N^{3/2}$) then the geometric area below $V$ is not only also abnormally large but is fairly close to the algebraic area. These Lemmas will be proven in Section \ref{proofan}, except for Lemma \ref{ldev} that was already proven in \cite{CNGP13}.

In Section \ref{ad} we state and prove a local limit theorem for any finite dimensional joint distribution of the
middle line $M$ and  the profile $|V|$ with $V$ sampled from $\bP_\beta(\cdot | V_N=0, |A_N(V)|=q N^2)$
as $N\to \infty$. As a by product of this local limit  theorem we will observe that 
asymptotically the rescaled profile $|\tilde V| $ and the rescaled middle line $\tilde M$ decorrelate.
Section \ref{ad} can be seen as the first part of the proof of Proposition \ref{llt}. We will indeed 
prove in Section \ref{pr1} that the latter asymptotic decorelation still holds true 
with $|\tilde V_l |$ and $\tilde M_l$ when $l$ is sampled from $\tilde P_{L,\beta}$ and $\beta>\beta_c$.
This will complete the proof of Proposition \ref{llt}. Similarly, Section \ref{prtensip} can be seen as the first part of the proof of 
Proposition \ref{tensip} since we prove the tightness of $\tilde M$ and $\tilde V$ under $\bP_\beta(\cdot | V_N=0, A_N(V)=q N^2)$
as $N\to \infty$. However, we will not display the part of the proof showing that this tightness is still satisfied
under $\tilde P_{L,\beta}$ since this can be done by mimicking the proof in Section \ref{pr1}.

\subsection{Preparations}\label{prle} 
Lemma \ref{lem:refine} shows that  the probability, under the polymer measure, that the rescaled 
horizontal excursion deviates from $a(\beta)$ by more than a given vanishing quantity decays faster than any given polynomial
provided the vanishing quantity decreases slowly enough.

\bl{lem:refine}
We set $\eta_L=L^{-1/8}$. For $\beta>\beta_c$ and $\alpha>0$ we have 
\be{eq:refine1}
\lim_{L\to \infty} L^{\alpha} P_{L,\beta}\big[ \big( B_{\eta_L, L}\big)^c  \, \big]=0,
\ee
with 
\be{refine2}
B_{\eta,L}=
\Big\{\tfrac{N_l}{\sqrt L} \in a(\beta)+[-\eta,\eta]  \Big\}.
\ee
\el

Lemma \ref{refine3} indicates that, provided we choose a constant $c>0$ large enough, the probability that the algebraic area and the geometric area described by $V_l$ differ from each other by more than $c (\log L)^4$ tends to $0$ faster than any polynomial.

\bl{refine3}
For $\beta>\beta_c$ and $\alpha>0$ there exists a $c>0$ such that 
\be{refine0}
\lim_{L\to \infty} L^{\alpha} \, P_{L,\beta}\big( | A_{N_l} (V_l) | \notin [L-N_l-c(\log L)^4,L-N_l] \big)=0.
\ee
\el

Lemma \ref{ldev} was proven in \cite[proposition 2.4]{CNGP13}. It 
identifies the sub-exponential decay rate of the event $\{V_N=0, A_N(V)=q N^2\}$ when $V$ is sampled from $\bP_\beta$. 

\begin{lemma}\label{ldev} {\rm [Proposition 2.4 in \cite{CNGP13}]}
For  $\etc{q_1,q_2} \subset (0,\infty)$, there exists $C_1>C_2>0$ and $n_0\in\N$ such that for all 
$q\in [q_1,q_2]\cap \frac{\N}{N^2}$ we have
\be{ldevi}
\frac{C_2}{N^2} e^{-\rho_\beta(q) N}\leq \bP_\beta (A_N(V)=q N^2, V_N=0)\leq \frac{C_1}{N^2} e^{-\rho_\beta(q) N}\,,
\ee
where $\rho_\beta(q):= L_\Lambda(\tilde{H}(q,0)) -q \tilde{h}_0(q,0)$,
so that $\rho_\beta$ is $\cC^\infty$ and we have $\tilde{G}(a)= a (\log\Gamma(\beta) - \rho_\beta(\unsur{a^2}))$ 
(recall \eqref{defg}).
\end{lemma}

Lemma \ref{refine4} insures us that, when $V$ is sampled from $\bP_\beta$ conditional on $ V_N=0, \valabs{A_N(V)}=q N^2$,
the probability that the geometric area described by $V$ differs from the algebraic area by more than $(\log N)^4$ tends to $0$ faster than any polynomial.
\bl{refine4}
For any $[q_1,q_2]\subset (0,\infty)$ and $\alpha,c>0$ we have 
\be{unip}
\lim_{N\to \infty} \sup_{q\in [q_1,q_2]} N^{\alpha}\,  \bP_\beta\big(G_N(V)\geq q N^2+c (\log N)^4\, |\, V_N=0, \valabs{A_N(V)}=q N^2\big) =0.
\ee
\el

We recall that all Lemmas of this Section are proven in Section \ref{proofan}. 

\subsection{Asymptotic decorrelation of the middle line and of the
  profile}\label{ad}
  
  In this section, we prove the following Lemma
that gives us a local limit theorem 
for the paths $V=(V_i)_{i=0}^N$ and $M=(M_i)_{i=0}^N$ simultaneously when $V$ is sampled from $\bP_\beta(\cdot\, |\, V_N=0, A_N(V)=qN^2)$. This local limit theorem is reinforced by the fact that it is uniform in $q$ belonging to any compact set of $]0,\infty[$.

\begin{lemma}\label{refine5}
For  $\etc{q_1,q_2} \subset (0,\infty)$ and  $(r_1,r_2) \in \N^{r_1+r_2}$,  we have
\begin{multline}
  \label{equaconv}
  \lim_{N\to +\infty} \sup_{q\in\etc{q_1,q_2}}
  \sup_{(\bar{x},\bar{y})\in\Z^{r_1+r_2}} \left|N^{
      \frac{r_1+r_2}{2}}\,  \bP_\beta(H_{\bar s,\bar t}(\bar x,\bar
    y)\, | \, W_{N,L,q N^2}) - \right.\\
\left. \hat{f}_{\tilde{H}(q,0),\bar{t}}\etp{\frac{\bar{y}}{N^{1/2}},N^{1/2}\gamma^*_q} g_{\tilde{H}(q,0),\bar{s}}\etp{\frac{\bar{x}}{N^{1/2}}}
 \right|= 0\,
\end{multline}
with 
$$W_{N,L,b}=\big\{\valabs{A_N(V)}=b,\, V_{N+1}=0,\, G_N(V)\in \{b,\dots, b+c (\log L)^4 \}\big\},$$
and $$ \hat{f}_{H,\bar{t}}(\bar{y},\phi) =
\frac12\etp{f^c_{H,\bar{t}}(\bar{y}-\phi(\bar{t})) + f^c_{H,\bar{t}}(\bar{y}+\phi(\bar{t}))}.$$
\end{lemma}
\medskip

\begin{remark}
For a given $\bar s$ and $\bar t$, and $[q_1,q_2]\in (0,\infty)$, there exists a $M>0$ such that 
\be{bound}
\sup_{q\in [q_1,q_2]} \sup_{(\bar x, \bar y)\in \R^{r_1+r_2}}f^c_{\tilde{H}(q,0),\bar{t}}(\bar{y})\,  g_{\tilde{H}(q,0),\bar{s}}(\bar{x})\leq M
\ee

\end{remark}

\subsubsection{Proof of Lemma \ref{refine5}}

First of all we note that, thanks to Lemma \ref{refine4}, it is sufficient to prove Lemma \ref{refine5} 
with the conditioning  $\{V_N=0, |A_N(V)|=q N^2\}$ instead of $W_{N,L,q N^2}$.  We will first prove Lemma 
\ref{refine5} subject to Theorem \ref{lltilted} that is stated below and then, we will  prove Theorem \ref{lltilted}.

\bt{lltilted}
For any $\etc{q_1,q_2} \subset\mathbb{R}$ we have
\begin{align*}
  \lim_{N\to +\infty} &\sup_{q\in\etc{q_1,q_2}}
  \sup_{(z_0,z_1,\bar{x},\bar{y})\in\R^{2+r_1+r_2}}\\
 &  \left|N^{2 +
      \frac{r_1+r_2}{2}} \probnhn{A_N=N^2q + z_0,V_N=z_1,V_{\fnt}=\bar{y},2\tilde{M}_{\fns}=\bar{x}} -\right.
    \\
&\hspace{4cm} \left. f_{\tilde{H}(q,0),\bar{t}}\etp{\frac{z_0}{N^{3/2}},\frac{z_1}{N^{1/2}},\frac{\bar{y}-N\gamma^*_q(\bar{t})}{N^{1/2}}} g_{\tilde{H}(q,0),\bar{t}}\etp{\frac{\bar{x}}{N^{1/2}}}\right|= 0\,.
\end{align*}

\et

Let $k_H(z_0,z_1)$ be the density of the law of $(\int_0^1\xi_H(s)\,
ds, \xi_H(1))$. Then we have $k_H(z_0,z_1) = \int
f_{H,\bar{t}}(z_0,z_1,\bar{y})\, d\bar{y}$ and we have the  limit
theorem (Proposition 6.1 of \cite{CNGP13})
$$ \lim_{N\to \infty} N^2 P_{N,h^q_N}\etp{V_N=0,A_N = q N^2} =
k_{\tilde{H}(q,0)}(0,0)\,.$$

Therefore, since we can write
\begin{multline*}
  N^{\frac{r_1+r_2}{2}}  \bP_\beta(H_{\bar s,\bar y}(\bar x,\bar y)\,
| \,V_N=0, \valabs{A_N(V)}=q N^2) = \\
\frac{N^{2+\frac{r_1+r_2}{2}}\probnhn{A_N=N^2q ,V_N=0,V_{\fnt}=\bar{y},2\tilde{M}_{\fns}=\bar{x}}}
{N^2 \etp{P_{N,H^q_N}\etp{V_N=0,A_N = q N^2}+P_{N,H^q_N}\etp{V_N=0,A_N
      = -q N^2}}}\\
 + \frac{N^{2+\frac{r_1+r_2}{2}}\probnhn{A_N=-N^2q +
      ,V_N=0,V_{\fnt}=\bar{y},2\tilde{M}_{\fns}=\bar{x}}}
{N^2 \etp{P_{N,H^q_N}\etp{V_N=0,A_N = q N^2}+P_{N,H^q_N}\etp{V_N=0,A_N
      = -q N^2}}}
\end{multline*}
We now can use some symmetry argument (the symmetry of the distribution
of the increments of the geometric random walk) to say that
\begin{gather*}
\probnhn{V_N=0,A_N = q N^2} = \probnhn{V_N=0,A_N = -qN^2}\\
\gamma^*_{-q}(t) = -\gamma^*_q(t)\,, \quad
\tilde{H}(-q,0)=-\tilde{H}(q,0), \quad f^c_{H,\bar{t}}(\bar{y})=f^c_{-H,\bar{t}}(\bar{y})=f^c_{H,\bar{t}}(-\bar{y})\,.
\end{gather*}
Therefore we obtain the desired result by applying Theorem \ref{lltilted} and
combining it with the definition of $f^c_{H,t}$.

\subsubsection{Proof of theorem \ref{lltilted}}


Before starting the proof of Theorem \ref{lltilted},  we shall give a
flavour of its nature by looking at a toy model.

Let us consider a random walk $S_0=0, S_n = X_1 + \cdots + X_n$ with
IID increment and let us build an alternating sign random walk
$\bar{S}_n := \sum_{i=1}^n (-1)^{i+1} X_i$.  If $X_1$ is square
integrable and, say,  $\esp{X_0}=0$, $\esp{X_0^2}=1$, then by the
Central Limit Theorem, $\frac{S_n}{\sqrt{n}} \cvloi Z \sim
\cN(0,1)$. Furthermore, we have asymptotic decorrelation
$$\unsur{\sqrt{n}}(S_n, \bar{S_n}) \cvloi (Z,\bar{Z})$$ a pair of
independent $\cN(0,1)$ distributed random variables, and this
convergence can be lifted to the level of processes : 
$$ \unsur{\sqrt{n}}(S_{\floor{nt}},\bar{S}_{\floor{nt}},; t\ge 1) \cvloi
 (B_t,\bar{B}_t, t\ge 1)$$
a pair of independent standard Brownian motions. Theorem \ref{lltilted} shows
that we can extend this decorrelation result to properly conditioned
processes, in the sense of finite distributions.
Recall that $V_{\floor{N\bar{t}}} =
(V_{\floor{Nt_1}}, \ldots, V_{\floor{Nt_{r_1}}})$ and for
$\bar{s}\in(0,1)^{r_2}$, $\tilde{M}_\fns = (\tilde{M}_{\floor{N s_1}},
\ldots, \tilde{M}_{\floor{N s_{r_2}}})$.
%

Let us now start the proof of Theorem \ref{lltilted}.
  The proof is a copy of the classic proof of the local central limit
  theorem, and very similar to the proof given in \cite{DH96}. We set
  $$X_N
  :=\etp{\frac{z_0}{N^{3/2}},\frac{z_1}{N^{1/2}},\frac{\bar{y}-\espnhn{V_{\fnt}}}{N^{1/2}},\frac{\bar{x}-2\espnhn{\tilde{M}_{\fns}}}{N^{1/2}}}\,.$$
 Recall that by construction of $H^q_N$, we have $\espnhn{A_N}=N^2q$. Hence,
  by Fourier inversion formula
\begin{multline*}
 N^{2 +
      \frac{r_1+r_2}{2}}\probnhn{A_N=N^2q +
  z_0, V_N=z_1, V_{\fnt}=\bar{y},\tilde{M}_{\fns}=\bar{x}}
= \\(2\pi)^{-(2+r_1+r_2)} \int_{\cA_N} e^{-i <T,X_N>}
\hat{\phi}_{N,H^q_N}(T)\, dT
\end{multline*}
with $T=(\tau_0,\tau_1,\bar{\kappa},\bar{\eta})$, $\cA_N$ the domain
$$ \cA_N :=\ens{T : \valabs{\tau_0} \le N^{3/2}\pi,\valabs{\tau_1}\le
  \sqrt{N} \pi, \valabs{\kappa_i}\le \sqrt{N} \pi,\valabs{\eta_j}\le
  \sqrt{N} \pi}$$
and $\hat{\phi}_{N,H^q_N}(T)$ the characteristic function 
$$ \hat{\phi}_{N,H^q_N}(T):=E_{N, H_N^q}\bigg[ e^{i\etp{\frac{\tau_0}{N^{3/2}}(A_N
    -N^2 q) + \unsur{N^{1/2}}(\tau_1 V_N +<\bar{\kappa},V_{\fnt}-\espnhn{V_{\fnt}}> +
    <\bar{\eta}, 2\tilde{M}_{\fns}-\espnhn{2\tilde{M}_{\fns}}>}}\bigg] $$

Therefore 
\begin{multline*}
  R_N := N^{2 +
      \frac{r_1+r_2}{2}} \probnhn{A_N=N^2q +
      z_0,V_N=z_1,V_{\fnt}=\bar{y},\tilde{M}_{\fns}=\bar{x}} \\
-
f_{\tilde{H}(q,0),\bar{t}}\etp{\frac{z_0}{N^{3/2}},\frac{z_1}{N^{1/2}},\frac{\bar{y}-\espnhn{V_{\fnt}}}{N^{1/2}}}
g_{\tilde{H}(q,0),\bar{s}}\etp{\frac{\bar{x}-\espnhn{2\tilde{M}_{\fns}}}{N^{1/2}}} =\\
C \int_{\cA_N} e^{-i <T,X_N>}
\etp{\hat{\phi}_{N,H^q_N}(T)-\bar{\phi}_{\tilde{H}(q,0)}(T)}\, dT
\end{multline*}
with $\bar{\phi}_{\tilde{H}(q,0)}(T)$ the characteristic function of
the Gaussian vector with density
$f_{H,\bar{t}}(z_0,z_1,\bar{y})g_{H,\bar{s}}(\bar{x})$ that is the
Gaussian vector $(\int_0^1\xi_H(s)\, ds, \xi_H(1),\xi_H(\bar{t}),
\bar{\xi}_H(\bar{s}))$ where $\bar{\xi}_H$ is an independent copy of $\xi_H$.

Hence,
$$ \valabs{R_N} \le C \int_{\cA_N}\valabs{
\hat{\phi}_{N,H^q_N}(T)-\bar{\phi}_{\tilde{H}(q,0)}(T)}\, dT =
\sum_{i=1}^4 J^{(q)}_i$$
 with
 \begin{align*}
   J^{(q)}_1 &= \int_{\Gamma_1} \valabs{
\hat{\phi}_{N,H^q_N}(T)-\bar{\phi}_{\tilde{H}(q,0)}(T)}\, dT,
&\Gamma_1=\etc{-A,A}^{2+r_1+r_2}\\
J^{(q)}_2 &= \valabs{\int_{\Gamma_2} \bar{\phi}_{\tilde{H}(q,0)}(T)\,
  dT}\,, & \Gamma_2 = \R^{2+r_1+r_2} \backslash \Gamma_1\,,\\
J^{(q)}_3 &= \int_{\Gamma_3} \valabs{ \hat{\phi}_{N,H^q_N}(T)}\, dT
&\Gamma_3 = \ens{T : \valabs{t_i} \le \Delta \sqrt{N}} \backslash
\Gamma_1 \\
J^{(q)}_4 &= \int_{\Gamma_4} \valabs{ \hat{\phi}_{N,H^q_N}(T)}\, dT
&\Gamma_4 = \cA_N \backslash (\Gamma_1 \cup \Gamma_3), 
 \end{align*}
where $A,\Delta>0$ are positive constants. We can bound $J^{(q)}_i$ for
$i=2,3,4$ exactly with the same procedure used in \cite{CNGP13},
Proposition 6.1. We shall focus on proving that 
$$ \forall A>0\,,\quad \lim_{N\to +\infty} \sup_{q \in\etc{q_1,q_2}}
J^{(q)}_1 = 0\,.$$
It is enough to prove that 
$$ \lim_{N\to +\infty} \sup_{q \in\etc{q_1,q_2}}  \sup_{T\in\Gamma_1}\valabs{
\hat{\phi}_{N,H^q_N}(T)-\bar{\phi}_{\tilde{H}(q,0)}(T)}=0\,.$$

To this end we shall note $\bar{\Lambda}_N = ( \frac{A_N}{N}, V_N,
V_\fnt, 2 M_{\fns})$ and consider the moment generating function
$$ \blbln(H,\bar{\kappa},\bar{\eta}) := \log \espbeta{\exp\etp{h_0 \frac{A_N}{
N} + h_1 V_N + <\bar{\kappa}, V_\fnt> + < \bar{\eta},M_{\fns}>}}$$
Observe that
$$ \log \hat{\phi}_{N,H^q_N}(T) = \blbln(H^q_N + \frac{i}{N^{1/2}}
(\tau_0,\tau_1), \frac{i}{N^{1/2}} \bar{\kappa},
\frac{i}{N^{1/2}}\bar{\eta}) - \blbln(H^q_N,0,0) - \frac{i}{N^{1/2}} < T,
\espnhn{\bar{\Lambda}_N}>$$
Therefore, an order 2 Taylor expansion gives
$$ \log \hat{\phi}_{N,H^q_N}(T) = -\undemi < \unsur{N} \Hess
\blbln(H^q_N,0,0) T,T> + \alpha_N$$
 with $\sup_N \sup_{q \in\etc{q_1,q_2}} N^{1/2} \alpha_N < +\infty$.

We can write the moment generating function explicitly as 
$$ \blbln(H,\bar{\kappa},\bar{\eta}) = \sum_{1\le i\le N}
L\etp{(1-\frac{i}{N})h_0+h_1 + \sum_{k:\floor{N t_k}\ge i} t_k +
  \sum_{m: \floor{N s_m}\ge i} (-1)^{i+1} \eta_m}$$
 Therefore, thanks to Proposition 2.3 and Lemma 5.1 of \cite{CNGP13},
 we have
 for $i,j \in \ens{0,1}$
$$ \unsur{N}\partial^2_{h_ih_j} \blbln(H^q_N,0,0)=\unsur{N} \partial^2_{h_ih_j} L_{\Lambda_N}(H^q_N) 
\to \partial^2_{h_ih_j}
\Llam(\tilde{H}(q,0))$$
Furthermore,
$$ \unsur{N}\partial^2_{\kappa_l^2} \blbln(H^q_N,0,0)=\unsur{N}
\sum_{i=1}^{\floor{Nt_l}} L''((1-\frac{i}{N}) h^{N,q}_0 + h^{N,q}_1)
\to \int_0^{t_l} L''((1-x)\tilde{h}_0, + \tilde{h}_1)\, dx\,.$$
If $l<m$ then 
$$\unsur{N}\partial^2_{\kappa_l\kappa_m}
\blbln(H^q_N,0,0)=\unsur{N}\partial^2_{\kappa_l^2} \blbln(H^q_N,0,0)$$
and converges to the preceding limit.
We have similarly,
$$ \unsur{N}\partial^2_{\eta_m^2} \blbln(H^q_N,0,0)=\unsur{N}
\sum_{i=1}^{\floor{Ns_m}} L''((1-\frac{i}{N}) h^{N,q}_0 + h^{N,q}_1)
\to \int_0^{s_m} L''((1-x)\tilde{h}_0, + \tilde{h}_1)\, dx\,.$$
 There remains to understand the cross term,
$$ \unsur{N}\partial^2_{\kappa_l\eta_m} \blbln(H^q_N,0,0) =
\unsur{N} \sum_{i=1}^{\floor{Nt_l}\wedge \floor{Ns_m}} (-1)^{i+1}
L''((1-\frac{i}{N}) h^{N,q}_0 + h^{N,q}_1) \to 0$$
since by regrouping the alternating  signs two by two, 
we have
$$ \valabs{L''((1-\frac{i}{N}) h^{N,q}_0 + h^{N,q}_1)  -
  L''((1-\frac{i-1}{N}) h^{N,q}_0 + h^{N,q}_1)} \le \frac{C}{N}$$
(thanks to the boundedness of the third derivative $L'''(z)$ on a
compact of $(-\beta/2,\beta/2)$).

Therefore, as a whole, the Hessian matrix $\unsur{N} \Hess
\blbln(H^q_N,0,0)$ converges uniformly on $\etc{q_1,q_2}$ to the
covariance matrix of the vector $(\int_0^1\xi_H(s)\, ds, \xi_H(1),\xi_H(\bar{t}),
\bar{\xi}_H(\bar{s}))$ and this establishes the convergence
$\lim_{N\to +\infty} R_N =0$.

To conclude, we need to prove now that in the expression of $R_N$, we
can replace the quantities $\espnhn{V_{\fnt}}$ and
$\espnhn{\tilde{M}_{\fns}}$ by respectively $N \gamma^*_q(\bar{t})$
and $0$.

We first observe that when $H=(h_0,h_1)$ lives in a compact set of
$(-\beta/2,\beta/2)$, which is the case if $q\in\etc{q_1,q_2}$ and
$H=H^q_N$or $H=\tilde{H}(q,0)$, then the variances of  $\xi_H(t)$ have a positive lower
bound, and therefore there exists a uniform Lipschitz constant $C$
such that for any $\bar{x},\bar{x}' \in \R^{r_1}$,
$\bar{y},\bar{y}'\in \R^{r_2}$, $z_0,z_1,z_0',z_1'$ and any
$q\in\etc{q_1,q_2}$
\begin{align*} \valabs{f_{\tilde{H}(q,0),\bar{t}}(z_0,z_1,\bar{y})-
f_{\tilde{H}(q,0),\bar{t}}(z_0',z_1',\bar{y}')} &\le C
(\norme{\bar{y}-\bar{y}'} +\norme{z-z'})\,,\\
\valabs{g_{\tilde{H}(q,0),\bar{s}}(\bar{x}) -
    g_{\tilde{H}(q,0),\bar{s}}(\bar{x}')} &\le C \norme{\bar{x} -\bar{x}'}\,.
\end{align*}

The second observation is that with $h_{N,i}= (1- \frac{i}{N})
h^{N,q}_0 + h^{N,q}_1$ we have
$$ \unsur{N} \espnhn{V_{\fnt}} = \unsur{N} \sum_{i=1}^{\floor{Nt}}
L'(h_{N,i}) \xrightarrow{N\to +\infty} \int_0^{t}
L'((1-x)\tilde{h}_0(q,0) + \tilde{h}_1(q,0))\,dx  = \gamma^*_q(t)\,,$$
and that, thanks to the boundedness on compact sets
 of $L''()$, this
convergence holds uniformly.

Similarly
$$ 2 \unsur{N} \espnhn{\tilde{M}_{\fns}} = \unsur{N}  \sum_{i=1}^{\floor{Nt}}
(-1)^{i+1}L'(h_{N,i}) \to 0\,,$$
 and  this convergence holds uniformly, since by grouping the
 alternating signs two by two, we have 
$$ \valabs{L'(h_{N,i}) -L'(h_{N,i+1})} \le \frac{C}{N}\,.$$

\subsection{Proof of Proposition \ref{llt}}\label{pr1}

Our aim is to  use the local limit Theorem stated in Proposition  \ref{llt}. To that aim, we define an equivalence relation between functions of type $\psi,\widetilde \psi \colon\, \N\times \Z^{r_1+r_2} \to  [0,\infty)$ so that 
$\psi\sim \widetilde \psi$ if
\be{equiv}
\lim_{L\to \infty} \sup_{(\bar x,\bar y)\in \Z^{r_1+r_2}} L^{\frac{r_1+r_2}{4}} \big| \psi_{L}(\bar x ,\bar y)-\widetilde\psi_{L}(\bar x ,\bar y)|=0.
\ee
We set 
\begin{align}\label{defpsi5}
\nonumber & \psi_{1,L}(\bar x, \bar y)= \widetilde{P}_{L,\beta}\big[ H_{\bar s,\bar t}(\bar x,\bar y)\big],\\
& \psi_{5,L}(\bar x, \bar y)=m_L^{-\frac{r_1+r_2}{2}}g_{\beta,\bar s}\Big(\tfrac{\bar
   x}{\sqrt{m_L}}\Big)\, f_{\beta,\bar t}\Big(\tfrac{\bar y}{\sqrt{m_L}},\sqrt{m_L} \gamma^*_{\beta}\Big)
\end{align}
and the proof of Theorem \ref{llt} will be complete once we show that
$\psi_1\sim \psi_5$. To achieve this equivalence we introduce $3$
intermediate functions $\psi_2$,$\psi_3$ and $\psi_4$ and we divide the proof into $4$ steps.
For $i\in \{1,2,3,4\}$, the i-th step consists in proving that $\psi_i\sim \psi_{i+1}$ so that at the end of the fourth step
we can state that $\psi_1\sim \psi_5$.   

 In steps 1 and 2, we  will use the fact that for $i\in \{1,2,3\}$, the $\psi_i$ function is of the form $\psi_i=\frac{A_i}{B_i}$ such that an equivalence of type \eqref{equiv} between $\psi_j$ and $\psi_{k}$  will be proven once we show that

\begin{align}\label{perms}
\nonumber &(i) \quad \lim_{L\to \infty} \sup_{(\bar x,\bar y)\in \Z^{r_1+r_2}} L^{\frac{r_1+r_2}{4}} \frac{|A_{j,L}-A_{k,L}|}{B_{j,L}}=0,\\
&(ii)\quad  \lim_{L\to \infty} \sup_{(\bar x,\bar y)\in \Z^{r_1+r_2}} L^{\frac{r_1+r_2}{4}} \frac{|A_{k,L}|}{B_{j,L}}\, 
\frac{|B_{j,L}-B_{k,L}|}{B_{k,L}}=0.
\end{align}
For the ease of notation, we will write $H$ instead of $H_{\bar s,\bar t}(\bar x,\bar y)$ until the end of the proof.

In the first step, we work under the polymer measure $\tilde P_{L,\beta}$ and we restrict the trajectories of $\tilde \Omega_L$   to those having an horizontal extension  $\approx a(\beta) \sqrt{L}$ and an algebraic area $A_N(V_l)\approx L-a(\beta)\sqrt{L}$.  
With the second step, we use the random walk representation in Section \ref{sec:rep} to switch from  
$\tilde P_{L,\beta}$ to $\bP_\beta$. Finally, in steps 3 and 4, we apply  the local limit theorem stated in Lemma \ref{refine5} to
complete the proof. 

\subsubsection{\bf Step 1}

We rewrite $\psi_1$ under the form  
\be{eq:reec}
\psi_{1,L}(\bar x, \bar y)=\frac{\sum_{L'\in K_L} \widetilde Z_{L',\beta}(H)}{\sum_{L'\in K_L} \widetilde Z_{L',\beta}},
\ee
where, for  $B\subset \widetilde \Omega_L$, the quantity  $\widetilde Z_{L',\beta}(B)$ is the restriction of the partition function 
$\widetilde Z_{L',\beta}$ to those trajectories $l\in B\cap \Omega_{L'}$,  i.e.,
$$\widetilde Z_{L',\beta}(B)=\big(\tfrac{e^\beta}{2}\big)^{-L'}\sum_{N=1}^{L'}\sum_{l \in\mathcal{L}_{N,L'}\cap B}  \mathbf{P}_{L'}(l)\   e^{\beta\sum_{i=1}^{N-1}(l_i\;\tilde{\wedge}\;l_{i+1})}.$$
At this stage we set $\eta=\eta_L=L^{-1/8}$ as in Lemma \ref{lem:refine}. We will note $I_{\eta,L}=\big\{(a(\beta)-\eta_L) \sqrt L, \dots, (a(\beta)+\eta_L) \sqrt L\big\}$ and we introduce the first intermediate function $\psi_2$, defined as 
\be{reec}
\psi_{2,L}(\bar x, \bar y)=\frac{\sum_{L'\in K_L} \widetilde Z_{L',\beta}(H\cap \cA_{L,L',\eta})}{\sum_{L'\in K_L} \widetilde Z_{L',\beta}(\cA_{L,L',\eta})}.
\ee
where
\begin{align}
\nonumber \cA_{L,L',\eta}&=\bigcup_{N\in I_{\eta,L}} \big\{l \in \cL_{N,L'}\colon\; |A_N(V_l)| \in L'+[-N-c (\log L)^4,-N]\cap \N\big\}.
\end{align}
For simplicity, we will omit the $L,\eta$ dependency of $\cA_{L,L',\eta}$ in what follows. The equivalence $\psi_1\sim\psi_2$ will be proven once we show that $(i)$ and $(ii)$ in \eqref{perms} are satisfied with $j=1,k=2$. We note that 

\begin{align}\label{prec}
\frac{|A_{1,L}-A_{2,L}|}{B_{1,L}}=\frac{\sum_{L'\in K_L} \widetilde Z_{L',\beta}(H\cap \cA_{L'}^c)}{\sum_{L'\in K_L} \widetilde Z_{L',\beta}}\leq \frac{\sum_{L'\in K_L}  \widetilde Z_{L',\beta} P_{L',\beta}(\cA_{L'}^c)}{\sum_{L'\in K_L} \widetilde Z_{L',\beta}},
\end{align}
with $\cA_{L'}^c=\Omega_{L'}\setminus \cA_{L'}$. We recall that $K_L=L+[-\gep(L),\gep(L)]\cap\N$  and we use, on the one hand, Lemma \ref{lem:refine} and the convergence $\lim_{L\to \infty} \frac{\gep(L)}{L}= 0$ to  claim that for all $\eta>0$
\be{tbc}
\lim_{L\to \infty}  L^{\frac{r_1+r_2}{2}}  \sup_{L'\in K_L} P_{L',\beta}(N_l \notin I_{\eta,L})=0,
\ee
and, on the other hand, Lemma \ref{refine3} and  $\lim_{L\to \infty} \frac{\gep(L)}{L}= 0$ to assert that there exists $c>0$ such that 
\be{tbc3}
\lim_{L\to \infty}  L^{\frac{r_1+r_2}{2}}  \sup_{L'\in K_L} P_{L',\beta}(|A_{N_l}(V_l)|\notin L'+[-N_{l}-c (\log L)^4,-N_{l}])=0,
\ee
 
We combine \eqref{tbc} and \eqref{tbc3} to claim that 
\be{tbc2}
\lim_{L\to \infty}  L^{\frac{r_1+r_2}{2}}  \sup_{L'\in K_L} P_{L',\beta}(\cA_{L'}^c)=0.
\ee
Thus, \eqref{prec} and \eqref{tbc2} are sufficient to prove (i).  It remains to show that (ii) is satisfied.  We note that  
\begin{align}\label{eq:pre}
\frac{A_{2,L}}{B_{1,L}}\, 
\frac{|B_{1,L}-B_{2,L}|}{B_{2,L}}&=\frac{\sum_{L'\in K_L} \widetilde Z_{L',\beta}(H\cap \cA_{L'})}{\sum_{L'\in K_L} \, \widetilde Z_{L',\beta}(\cA_{L'})}\, \frac{\sum_{L'\in K_L} \widetilde Z_{L',\beta}(\cA_{L'}^c)}{\sum_{L'\in K_L} \widetilde Z_{L',\beta}},\\
&\leq  \frac{\sum_{L'\in K_L} \widetilde Z_{L',\beta} P_{L',\beta}(\cA_{L'}^c)}{\sum_{L'\in K_L} \widetilde Z_{L',\beta}},
\end{align}
and then, we can use directly \eqref{tbc2} to obtain (ii) and this completes the proof of step 1.

%

\subsubsection{\bf Step 2}
To begin with, we set $J_{N,L',L}=\{L'-N-c(\log L)^4,\dots, L'-N\}$ and we note that, with the help of the random walk representation, we can rewrite
\begin{align}\label{reec2}
\psi_{2,L}(\bar x, &\bar y)=\\
\nonumber &\frac{\sum_{L' \in K_L}\sum_{N\in I_{\eta,L}} (\Gamma_\beta)^N \sum_{b\in J_{N,L',L}} \bP_\beta\big(H, 
|A_N(V)|=b, V_{N+1}=0, G_N(V)=L'-N\big)  }{\sum_{L' \in K_L} \sum_{N\in I_{\eta,L}} (\Gamma_\beta)^N \sum_{b\in J_{N,L',L}} \bP_\beta\big( |A_N(V)|=b, V_{N+1}=0, G_N(V)=L'-N\big) }.
\end{align}
We switch the order of summation in \eqref{reec2} and  we obtain  
\be{reec3}
\psi_{2,L}(\bar x, \bar y)=\frac{\sum_{N\in I_{\eta,L}} (\Gamma_\beta)^N \sum_{b\in D_{N,L}} \bP_\beta \big(H,|A_N(V)|=b, V_{N+1}=0,\,  G_N(V)\in  U_{b,N,L}\big)  }{\sum_{N\in I_{\eta,L}} (\Gamma_\beta)^N \sum_{b\in D_{N,L}} \bP_\beta\big( |A_N(V)|=b, V_{N+1}=0, G_N(V)\in U_{b,N,L} \big)},
\ee
with 
\begin{align}\label{defDU}
D_{N,L}&=\{L-\gep(L)-N-c(\log L)^4,\dots, L+\gep(L)-N\},\\
U_{b,N,L}&=\{b\vee [L-\gep(L)-N],\dots, [b+c(\log L)^4] \wedge [L+\gep(L)-N]\}.
\end{align}
We define the third intermediate function 
\be{3fun}
\psi_{3,L}(\bar x, \bar y)=\frac{\sum_{N\in I_{\eta,L}} (\Gamma_\beta)^N \sum_{b\in D_{N,L}} \bP_\beta\big(H, |A_N(V)|=b, V_{N+1}=0, G_N(V)\in  \widetilde U_{b,L}\big)  }{\sum_{N\in I_{\eta,L}} (\Gamma_\beta)^N \sum_{b\in D_{N,L}} \bP_\beta\big( 
|A_N(V)|=b, V_{N+1}=0, G_N(V)\in \widetilde U_{b,L} \big)},
\ee
with 
\begin{align}\label{defDU2}
\widetilde U_{b,L}&=\{b,\dots,b+c(\log L)^4\}.
\end{align}
The equivalence $\psi_2\sim\psi_3$ will be proven once we show that $(i)$ and $(ii)$ in \eqref{perms} are satisfied with $j=3,k=2$. 

For $(i)$, we note that for all $b\in D_{N,L}$ satisfying $b\geq L-\gep(L)-N$ and $b\leq L+\gep(L)-N-c (\log L)^4$ we have $U_{b,L}=\widetilde U_{b,L}$ and
therefore
\begin{align}\label{preci}
\frac{|A_{2,L}-A_{3,L}|}{B_{3,L}}&\leq \frac{\sum_{N\in I_{\eta,L}} (\Gamma_\beta)^N \sum_{b\in \widetilde D_{N,L}} \bP_\beta \big(H, |A_N(V)|=b, V_{N+1}=0, G_N(V)\in  \widetilde U_{b,L}\big)  }{\sum_{N\in I_{\eta,L}} (\Gamma_\beta)^N \sum_{b\in D_{N,L}} \bP_\beta \big( |A_N(V)|=b, V_{N+1}=0, G_N(V)\in \widetilde U_{b,L} \big)},\\
\nonumber &=\frac{\sum_{N\in I_{\eta,L}} (\Gamma_\beta)^N \sum_{b\in \widetilde D_{N,L}}  \bP_\beta \big( W_{N,L,b} \big)\ \bP_\beta\big(H\, |\,  W_{N,L,b}\big)  }{\sum_{N\in I_{\eta,L}} (\Gamma_\beta)^N \sum_{b\in D_{N,L}} \bP_\beta\big( W_{N,L,b} \big)},
\end{align}
with 
\begin{align}\label{eq:defDU}
\nonumber \widetilde D_{N,L}&=\{b\in D_{N,L}\colon b < L-\gep(L)-N\  \text{or}\   b> L+\gep(L)-N-c(\log L)^4\},\\
W_{N,L,b}&=\{|A_N(V)|=b,\, V_{N+1}=0,\, G_N(V)\in \widetilde U_{b,L}\}.
\end{align}

For $(ii)$, in turn, we note that, 
\begin{align}\label{precis}
\nonumber \frac{A_{2,L}}{B_{3,L}}\leq \frac{A_{3,L}}{B_{3,L}}&=\frac{\sum_{N\in I_{\eta,L}} (\Gamma_\beta)^N \sum_{b\in D_{N,L}} \bP_\beta(H, |A_N(V)|=b, V_{N+1}=0, G_N(V)\in  \widetilde U_{b,L})  }{\sum_{N\in I_{\eta,L}} (\Gamma_\beta)^N \sum_{b\in D_{N,L}} \bP_\beta( |A_N(V)|=b, V_{N+1}=0, G_N(V)\in \widetilde U_{b,L} )}.\\
&=\frac{\sum_{N\in I_{\eta,L}} (\Gamma_\beta)^N \sum_{b\in  D_{N,L}}  \bP_\beta \big( W_{N,L,b} \big)\ \bP_\beta\big(H\, |\,  W_{N,L,b}\big)  }{\sum_{N\in I_{\eta,L}} (\Gamma_\beta)^N \sum_{b\in D_{N,L}} \bP_\beta\big( W_{N,L,b} \big)},
\end{align}
and 
\begin{align}\label{pre}
\nonumber \frac{|B_{2,L}-B_{3,L}|}{B_{2,L}}&\leq\frac{\sum_{N\in I_{\eta,L}} (\Gamma_\beta)^N \sum_{b\in \widetilde D_{N,L}} \bP_\beta(|A_N(V)|=b, V_{N+1}=0, G_N(V)\in  \widetilde U_{b,L})  }{\sum_{N\in I_{\eta,L}} (\Gamma_\beta)^N \sum_{b\in D_{N,L}} \bP_\beta( |A_N(V)|=b, V_{N+1}=0, G_N(V)\in U_{b,L} )},\\
&\leq  \frac{\sum_{N\in I_{\eta,L}} (\Gamma_\beta)^N \sum_{b\in \widetilde  D_{N,L}}  \bP_\beta \big( W_{N,L,b} \big)\ }{\sum_{N\in I_{\eta,L}} (\Gamma_\beta)^N \sum_{b\in D_{N,L}\setminus \widetilde D_{N,L}} \bP_\beta\big( W_{N,L,b} \big)},
\end{align}
where we have used again  that  $\widetilde U_{b,L}=U_{b,L}$ for $b\in D_{N,L}\setminus \widetilde D_{N,L}$.
At this stage, we state two claims that will be sufficient to complete this step.
\begin{claim}\label{deuxi} For $\eta>0$, there exists a $C>0$ such that 
$$ \limsup_{L\to \infty} \, \sup_{N\in I_{\eta,L}}\, \sup_{b\in D_{N,L}}\,  \sup_{(\bar x,\bar y)\in \Z^{r_1+r_2}} L^{\frac{r_1+r_2}{2}} \bP_\beta\big(H\, |\,  W_{N,L,b}\big)\leq C.$$
\end{claim}

\begin{claim}\label{premie} For $\eta>0$, we have
\be{dernch}
\lim_{L\to \infty} \sup_{N\in I_{\eta,L}} \frac{\sum_{b\in \widetilde{D}_{N,L}} \bP_\beta\big(W_{N,L,b}\big)}{
\sum_{b\in D_{N,L}} \bP_\beta\big(W_{N,L,b}\big)}=0.
\ee
\end{claim}

Claims \ref{deuxi} and \ref{premie}, together with \eqref{preci}, easily imply that $(i)$  is satisfied. 
Moreover, \eqref{precis} and Claim \ref{deuxi} allow us to state that 
\be{limsup}
\limsup_{L\to \infty} \sup_{(\bar x,\bar y)\in \Z^{r_1+r_2}} L^{\frac{r_1+r_2}{2}}  \tfrac{A_{2,L}}{B_{3,L}} \leq
C
\ee 
while Claim \ref{premie} and \eqref{pre} yield that $\lim_{L\to \infty}  \frac{|B_{2,L}-B_{3,L}|}{B_{2,L}} =0$
which proves $(ii)$ and completes the proof of $\psi_2\sim\psi_3$.

It remains to display a proof for  Claims \ref{deuxi} and \ref{premie}.

\subsubsection{Proof of Claim \ref{deuxi}}

 Since $\eta_L\to 0$, we easily infer that for $L$ large enough and for $N\in  I_{\eta,L}$ and $b\in D_{N,L}$ we have
$$ \frac{b}{N^2}\in \Big[\frac{1}{2 a(\beta)^2},\frac{2}{a(\beta)^2}\Big]:=[R_1,R_2].$$ 
Thus, we can use  Lemma \ref{refine5} and \eqref{bound} to assert that, for $L$ large enough, $ L^{\frac{r_1+r_2}{2}} \bP_\beta\big(H\, |\,  W_{N,L,b}\big)$ is bounded 
above uniformly in $N\in  I_{\eta,L}$ and $b\in \widetilde D_{N,L}$. The $\beta$ dependency of $R_1$ and $R_2$ is omitted for simplicity.

\subsubsection{Proof of Claim \ref{premie}}
By using again the fact that for $N\in  I_{\eta,L}$ and $b\in D_{N,L}$ we have, for $L$ large enough, that
$ \frac{b}{N^2}\in [R_1,R_2]$,
we can apply  Lemma \ref{refine4}, to assert that for $L$ large enough 
\be{iuds}
\inf_{N\in I_{\eta,L}} \inf_{b\in D_{N,L}} \bP_\beta (G_N(V)\leq b+ c (\log L)^4 \,\big | \, |A_N(V)|=b, V_{N+1}=0)\geq \frac12.
\ee
We recall the definition of $W_{N,L,b}$ in \eqref{eq:defDU} and we set
$T_{N,b}:=\{|A_N(V)|=b, V_{N+1}=0\}$. We can bound from above the ratio in the l.h.s. of \eqref{dernch} as \begin{align}\label{dernch1}
\frac{\sum_{b\in \widetilde{D}_{N,L}} \bP_\beta\big(W_{N,L,b}\big)}{
\sum_{b\in D_{N,L}} \bP_\beta\big(W_{N,L,b}\big)}&\leq \frac{\sum_{b\in \widetilde{D}_{N,L}} \bP_\beta\big( T_{N,b}\big)}{
\sum_{b\in D_{N,L}} \bP_\beta\big(W_{N,L,b}\big)}\leq 2\,  \frac{\sum_{b\in \widetilde{D}_{N,L}} \bP_\beta\big(T_{N,b}\big)}{
\sum_{b\in D_{N,L}} \bP_\beta\big(T_{N,b}\big)}
\end{align}
where we have used \eqref{iuds} to obtain that for $L$ large enough, $N\in  I_{\eta,L}$ and $b\in D_{N,L}$ we have
$$\bP_\beta\big(W_{N,L,b}\big)\geq \bP_\beta\big(T_{N,b}\big) (\bP_\beta\big[G_N(V)\leq b+ c (\log L)^4 | T_{N,b}\big]\big)\geq \tfrac12 \bP_\beta\big(T_{N,b}\big).$$

At this stage, we need to use the fact that $\gep(L)=\log(L)^6$ and  we set
$$d^-_{N,L}= L-\gep(L)-N\quad \text{and}\quad d^+_{N,L}= L+\gep(L)-N$$ and also 
\begin{align}\label{defd}
D^1_{N,L}=\{d^{-}_{N,L},\dots,d^-_{N,L}+(\log L)^5\}&\quad\text{and} \quad D^2_{N,L}=\{d^+_{N,L}-(\log L)^5,\dots,d^+_{N,L}\}\\
\widetilde D^1_{N,L}=\{d^{-}_{N,L}-c( \log L)^4,\dots,d^-_{N,L}\}&\quad\text{and} \quad \widetilde D^2_{N,L}=\{d^+_{N,L}-c(\log L)^4,\dots,d^+_{N,L}\}
\end{align}
such that $D^1_{N,L}$ and $D^2_{N,L}$ are disjoint subsets of $D_{N,L}$ and $\widetilde D^1_{N,L}\cup
\widetilde D^2_{N,L}$ is a partition of $\widetilde D_{N,L}$.
We note that all $b\in D^1_{N,L}\cup \widetilde D^1_{N,L}$ satisfies $|b-d^{-}_{N,L}|\leq (\log L)^5$ and 
similarly that  all $b\in D^2_{N,L}\cup \widetilde D^2_{N,L}$
satisfies $|b-d^{+}_{N,L}|\leq (\log L)^5$. In fact, for $L$ large
enough, all the numbers
$\frac{b}{N^2}, \frac{d^{-}_{N,L}}{N^2},\frac{d^{+}_{N,L}}{N^2}$
belong to the compact $\etc{R_1,R_2}$ on which the function
  $\rho_\beta$ is differentiable. Therefore, there exists a $C>0$ such
  that 
\begin{align}\label{2in}
&|\rho_\beta(\tfrac{b}{N^2})-\rho_\beta(\tfrac{d^{-}_{N,L}}{N^2})|\leq
C \unsur{N^2} (\log L)^5, \quad b \in D^1_{N,L}\cup \widetilde D^1_{N,L},\\
&|\rho_\beta(\tfrac{b}{N^2})-\rho_\beta(\tfrac{d^{-}_{N,L}}{N^2})|\leq
C \unsur{N^2} (\log L)^5, \quad b \in D^2_{N,L}\cup \widetilde D^2_{N,L}.
\end{align}
Thus, we can apply Lemma \ref{ldev} to assert that there exists $M_1>M_2>0$ such that for $L$ large enough and 
for $N\in I_{\eta,L}$ we have 

\begin{align}\label{pbe}
\nonumber \frac{M_2}{N^2} e^{N\, \rho_\beta\big(\frac{d^{-}_{N,L}}{N^2}\big)}& \leq \bP_\beta(T_{N,b}) \leq \frac{M_1}{N^2} e^{N \rho_\beta\, \big(\frac{d^{-}_{N,L}}{N^2}\big)},\quad b\in D^1_{N,L}\cup \widetilde D^1_{N,L}\\
\frac{M_2}{N^2} e^{N\, \rho_\beta\big(\frac{d^{+}_{N,L}}{N^2}\big)}& \leq \bP_\beta(T_{N,b}) \leq \frac{M_1}{N^2} e^{N \,\rho_\beta\big(\frac{d^{+}_{N,L}}{N^2}\big)},\quad b\in D^2_{N,L}\cup \widetilde D^2_{N,L}.
\end{align}
It suffices  to combine \eqref{dernch1} and \eqref{pbe} and to note that 
$$|D^1_{N,L}|/|\widetilde D^1_{N,L}|=|D^2_{N,L}|/|\widetilde D^2_{N,L}|=
c (\log L)^4/(\log L)^5,$$
to complete the proof of Claim \ref{premie}.

\subsubsection{\bf Step 3}
In this step, we note first that $\psi_3$ can be written as  
\be{3fun2}
\psi_{3,L}(\bar x, \bar y)=\frac{\sum_{N\in I_{\eta,L}} (\Gamma_\beta)^N \sum_{b\in D_{N,L}}  \bP_\beta\big( W_{N,L,b}\big) \,  \bP_\beta\big( H\, |\, W_{N,L,b}\big)}{ \sum_{N\in I_{\eta,L}} (\Gamma_\beta)^N\,   \bP_\beta\big( W_{N,L,b} \big)}  ,
\ee
and  we set 

\begin{align}\label{4fun}
\nonumber &\psi_{4,L}(\bar x, \bar y)=\\
&\frac{\sum_{N\in I_{\eta,L}} (\Gamma_\beta)^N \sum_{b\in D_{N,L}}  \bP_\beta\big(
  W_{N,L,b}\big) \,  N^{-\frac{r_1+r_2}{2}} \hat{f}_{\tilde{H}(\frac{b}{N^2},0),\bar{t}}\etp{\frac{\bar{y}}{N^{1/2}},N^{1/2}\gamma^*_q} g_{\tilde{H}(\frac{b}{N^2},0),\bar{s}}\etp{\frac{\bar{x}}{N^{1/2}}}}{ \sum_{N\in I_{\eta,L}} (\Gamma_\beta)^N\,  \sum_{b\in D_{N,L}}  \bP_\beta\big( W_{N,L,b} \big)},
\end{align}
We have, thanks to Lemma \ref{refine5}, the existence of a sequence
$R_N\to 0$ such that
$$ \sup_{(\bar x,\bar y)\in \Z^{r_1+r_2}} L^{\tfrac{r_1+r_2}{4}}
\valabs{\psi_{3,L} -\psi_{4,L}} \le  L^{\tfrac{r_1+r_2}{4}}\sup_{N\in
  I_{\eta,L}} (R_N N^{-\frac{r_1+r_2}{2}}) \to 0\,.$$

\subsubsection{\bf Step 4}
Obviously we have,
$$ \psi_{5,L}(\bar x,\bar y) = 
\frac{\sum_{N\in I_{\eta,L}}
  (\Gamma_\beta)^N \sum_{b\in D_{N,L}}  \bP_\beta\big(
  W_{N,L,b}\big) \,  \psi_{5,L}(\bar x,\bar y)}{ \sum_{N\in
    I_{\eta,L}} (\Gamma_\beta)^N\,  \sum_{b\in D_{N,L}}
  \bP_\beta\big( W_{N,L,b} \big)}\,.
$$

Therefore,
\begin{multline*}
 \valabs{\psi_{4,L} -\psi_{5,L}} \le \sup_{N\in I_{\eta,L}, b\in
  D_{N,L}}
\left| m_L^{-\frac{r_1+r_2}{2}}g_{\beta,\bar s}\Big(\tfrac{\bar
   x}{\sqrt{m_L}}\Big)\, f_{\beta,\bar t}\Big(\tfrac{\bar
   y}{\sqrt{m_L}},\sqrt{m_L} \gamma^*_{\beta}\Big)\right.\\
 \left. -N^{-\frac{r_1+r_2}{2}}
 \hat{f}_{\tilde{H}(\frac{b}{N^2},0),\bar{t}}\etp{\frac{\bar{y}}{N^{1/2}},N^{1/2}\gamma^*_q}
 g_{\tilde{H}(\frac{b}{N^2},0),\bar{s}}\etp{\frac{\bar{x}}{N^{1/2}}}\right|\,.$$
\end{multline*}
Obviously, since $\eta=\eta_L \to 0$,
$$ \lim_{L\to \infty} L^{\frac{r_1+r_2}{4}} \sup_{N \in I_{\eta,L}}
  \valabs{m_L^{-\frac{r_1+r_2}{2}}-N^{-\frac{r_1+r_2}{2}}} =0\,.$$

By the implicit function theorem applied to the definition~\eqref{eq:deftil} of
$\tilde{H}$, the function $q\to \tilde{H}(q,0)$ is globally Lipschitz
on compact sets, and thus there exists a constant $C>0$ such that 
$$ \sup_{L\ge L_0} \eta_L^{-1}\sup_{N\in I_{\eta,L}, b\in
  D_{N,L}} \valabs{\tilde{H}(q_\beta,0)-\tilde{H}(\tfrac{b}{N^2},0)}
\le  C \,.$$
By the global Lipschitz properties in $(H,\bar x,\bar y)$ of the Gaussian densities
$f^c_{H,\bar t}(\bar y)$ and $g_{H,\bar s}(\bar x)$ we conclude that
$\psi_{4,L} \sim \psi_{5,l}$.

\subsection{Proof of Proposition  \ref{tensip}}\label{prtensip}
The proof of the
tigthness of the sequence of distributions
$(Q_{L,\beta})_{L\ge 1}$ is obtained by combining arguments of \cite[Section 6]{DH96} with the steps 1, 2 and 3 of 
the proof of Proposition \ref{llt}.  We focus on proving the tightness of the
first coordinate of the process, i.e.,  $(\sqrt{N_l}
\tilde{M}_l(s))_{s\in\etc{0,1}}$ (the proof for the second
coordinate is completely similar and  even closer to the proof displayed in \cite[Section 6]{DH96}).

Let us denote by
$(\widehat{M}_l(s))_{s\in\etc{0,1}}$  the polygonal
interpolation of the middle line. Then, see for example \cite{DZ}
proof of Lemma 5.1.4, the distribution under $\tilde{P}_{L,\beta}$ of
$(\widehat{M}_l(s))_{s\in\etc{0,1}}$ and
$(\tilde{M}_l(s))_{s\in\etc{0,1}}$ are exponentially close, and so
we shall restrict ourselves to proving the tightness of the sequence
of continous processes
$(\sqrt{N_l}\widehat{M}_l(s))_{s\in\etc{0,1}}$ using the
criterion of Theorem 7.3, (ii) of \cite{Bi}:
\begin{gather}
  \forall \epsilon>0,\eta>0, \exists L_0\in\N,\exists \delta\in (0,1)
  \\
\tilde{P}_{L,\beta}
\etp{w(\sqrt{N_l}\widehat{M}_l(s),0\le s\le 1; \delta) \ge \epsilon}
\le \eta \qquad (\forall L \ge L_0)\,.
\end{gather}
where $w(f,\delta)=\sup_{\valabs{x-y}\le \delta} \valabs{f(x)-f(y)}$
denotes the modulus of continuity.

\medskip
We then inspect closely  the steps 1, 2 and 3 taken in  Proposition \ref{llt}. It is tedious,
but straightforward, to see that we only need to prove that 
\begin{gather}
  \label{eq:modcontipbetamtilde}
  \forall \epsilon>0,\forall \eta>0, \, \exists N_0\in\N,\exists \delta\in (0,1)
  \\
\probbeta{w(\tfrac{1}{\sqrt{N}}M_{\floor{Ns}},0\le
  s\le 1; \delta) \ge \epsilon\middle| A_N=qN^2, V_N=0}
\le \eta \qquad (\forall N \ge N_0,\forall q \in \etc{q_1,q_2})\,.
\end{gather}

To this end, we shall use Kolmogorov's tightness criterion (see
Theorem 1.8 Chap. XIII of \cite{MR1083357}) and show that for all
$\etc{q_1,q_2}$ there exists
$\alpha,\gamma,C ,N_0>0$ such that $\forall N\ge N_0, \forall q \in
  \etc{q_1,q_2}$, $\forall 0\le s,t\le 1$
\begin{equation}
  \label{eq:ktcrit}
  \espbeta{\valabs{\frac{M_{\floor{Ns}}
        -M_{\floor{Nt}}}{\sqrt{N}}}^\alpha \Bigm|  A_N=qN^2, V_N=0} \le
  C \valabs{t-s}^{1+\gamma} \,.
\end{equation}

Eventually, we shall prove that
\begin{equation}
  \label{eq:evmomentktc}
  \sup_{N\ge N_0,q\in\etc{q_1,q_2},0\le s,t\le 1} \espbeta{\valabs{\frac{M_{\floor{Ns}}
        -M_{\floor{Nt}}}{\sqrt{N}}}^4 \Bigm|
    A_N=qN^2, V_N=0} < C   \valabs{t-s}^{\frac{7}{4}}. 
\end{equation}

At this stage, the proof is completed by mimicking the  proof of the weak compactness 
exposed in \cite[section 6]{DH96}.

\subsection{Proof of Lemmas \ref{lem:refine}--\ref{refine4}}\label{proofan}

\subsubsection{Proof of Lemma \ref{lem:refine}}
We just need to recall from \cite{CNGP13} the formula  (4.43) and the inequality 
(4.50) with $\epsilon'=\eta_L$  to see that since
$\tilde{G}(a)$ reaches its maximum at $a(\beta)$ and since the second derivative of 
$\tilde{G}(a)$ is strictly negative, $\eta_L=L^{-1/8}$ is suitable.

\subsubsection{Proof of Lemma \ref{refine3}}

Let us recall the notation $I_{j_{\text{max}}}$ of
\cite[Section 1]{CNGP13} which
is the set of indexes of  stretches that occur in the largest bead. In
other word it is the largest set of consecutive indices for which
$V_{l,i}$ keeps the same sign.
To begin with, we can show, following \emph{mutatis mutandis} the proof of Theorem C,  that for $\alpha>0$, there exists a $c>0$ such that 
\be{refine1}
\lim_{L\to \infty} L^{\alpha}\,  P_{L,\beta}\big[ I_{j_\text{max}}\leq L-c (\log L)^4 \big]=0.
\ee

Note that for $l \in \Omega_L$, we have $\sum_{i=1}^{N_l} |V_{l,i}|=\sum_{i=1}^{N_l} |l_i|= L-N_l$ and thus, by the definition of $I_{j_{\text{max}}}$ we have also
\begin{align}\label{aira}
& \sum_{i=1}^{N_l}  |V_{l,i}|-2 \sum_{i\notin  I_{\text{max}}} |V_{l,i}|        \leq \big|\sum_{i=1}^{N_l} V_{l,i}\big| 
\leq \sum_{i=1}^{N_l} |V_{l,i}|.
\end{align}
Moreover, we note that $l\in \{ I_{j_\text{max}}\geq L-c (\log L)^4\}$ yields 
\be{aira2}
\sum_{i\notin I_{j_\text{max}}} |V_{l,i}| =\sum_{i\notin I_{j_\text{max}}} |l_i| \leq c (\log L)^4.
\ee
At this stage, we recall that $A_N(V)=\sum_{i=1}^N V_i$ and we  use \eqref{aira} and \eqref{aira2} to assert that $l\in \{ I_{j_\text{max}}\geq L-c (\log L)^4\}$ implies $ \big| A_{N_l} (V_l)\big| \in [L-N_l-2c(\log L)^4,L-N_l]$.  It remains to use \eqref{refine1} to complete the proof of Lemma \ref{refine3}.


\subsubsection{Proof of Lemma \ref{refine4}}

For the sake of conciseness, we will use, in this proof only, the notations 
\begin{align}\label{inte}
\nonumber W_{c,N,q}&=\{G_N(V)\geq q N^2+c (\log N)^4,\, V_N=0, \, A_N(V)=q N^2\},\\
T_{N,q}&=\{V_N=0,\, A_N(V)=q N^2\},
\end{align}
with $c>0$ and $q\in [q_1,q_2]\cap \tfrac{\N}{N^2}$.
Thus, the proof of Lemma \ref{refine4} will be complete once we show that for $\alpha,c>0$  we have 
\be{tpp}
\lim_{N\to \infty} \sup_{q\in [q_1,q_2]} N^{\alpha} \bP_\beta\big(W_{c,N,q}\, |\, T_{N,q}\big) =0,
\ee
where the intersection of $[q_1,q_2]$ with $\tfrac{\N}{N^2}$ is omitted for simplicity.
Since the equality $P_{N,H_N^q} \big(W_{c,N,q}\, |\, T_{N,q}\big) =  \bP_\beta\big(W_{c,N,q}\, |\, T_{N,q}\big)$ holds for all $N\in \N$ and $q\in [q_1,q_2]$, we can use \cite[Proposition 2.2]{CNGP13} to ensure that  \eqref{tpp}
will be proven once we show that  for $c,\alpha>0$ we have
\be{tpp1}
\lim_{N\to \infty} \sup_{q\in [q_1,q_2]} N^{\alpha} P_{N,H_N^q}\big(W_{c,N,q} \big) =0.
\ee
For $N\in \N$ and $\tilde c>0$, we let $U_{N, \tilde c}$ be the set containing those trajectories that do not remain strictly positive on  the interval
$[\tilde c \log N,N-\tilde c \log N]\cap\N$ and satisfy $V_N=0$, i.e.,
\begin{align}\label{uc2}
\nonumber U_{N,\tilde  c}=\{\max\{i\leq \tfrac{N}{2} \colon\, V_i\leq 0\}& \geq  \tilde c \log N,\, V_N=0\}\\
&\cup \{\max\{i\leq \tfrac{N}{2} \colon\, V_{N-i}\leq 0\}\geq \tilde c \log N, \, V_N=0\}
\end{align}
so that we can write the upper bound
\be{up}
P_{N,H_N^q}\big(W_{c,N,q})\leq P_{N,H_N^q}(U_{N,\tilde c})+ P_{N,H_N^q}(W_{c,N,q}\cap ( U_{N,\tilde c})^c).
\ee

At this stage, we need  to distinguish between the positive part $A_N^+(V)$ and the negative part $A_N^-(V)$ of the algebraic area below the $V$ trajectory, i.e.,
$${\textstyle A_N^-(V)=-\sum_{i=1}^N V_i\, \ind_{\{V_i\leq 0\}}\quad \text{and} \quad A_N^+(V)=\sum_{i=1}^N V_i , \ind_{\{V_i\geq 0\}}.}$$
As a consequence, the geometric and the algebraic areas below the $V$ trajectory can be written as  $A_N(V)= A_N^+(V)-A_N^-(V)$ and $G_N(V)= A_N^+(V)+A_N^-(V)$ and therefore, under the event $A_N(V)=q N^2$ we have $G_N(V)=q N^2+2 A_N^-(V)$. Thus, under the event $W_{n,q}\cap (U_{N,\tilde c})^c$ we have necessarily that $A^-_{N}(V)\geq \tfrac{c}{2} (\log N)^4$ and since $V$ is strictly positive between $\tilde c \log N$ and $N-\tilde c \log N$ we can write

\begin{align}\label{up1}
\nonumber W_{c,N,q}\cap (U_{N,\tilde c})^c\, \subset \,  \big\{ A^-_{\,\tilde c \log N}(V)\geq\,  &\tfrac{c}{4} (\log N)^4, \, V_N=0\big\}\\
&\cup\big\{ A^-_{\,N-\tilde c \log N,N}(V)\geq  \tfrac{c}{4} (\log N)^4,\, V_N=0\big\}
\end{align}
where $A^-_{\,N-\tilde c \log N,N}=-\sum_{i=N-\tilde c \log N}^{N} V_i \ind_{\{V_i\geq 0\}}$.
Moreover, under   $P_{N,H_N^q}$, the sequences of random variables $(U_i)_{i=1}^N$ and   
$(-U_{N-i})_{i=0}^{N-1}$ have  the same law (by symmetry) and therefore the equality 
\be{frs}
P_{N,H_N^q} \big( (V_1,\dots,V_k)\in A, V_N=0\big)= P_{N, H_N^q} \big( (V_{N-1},\dots,V_{N-k})\in A, V_N=0\big)
\ee
holds  true for all $k\in \{1,\dots,N-1\}$ and $A\in \text{Bor}(\R^k)$. A straightforward application of \eqref{frs} tells us that under $P_{N,H_N^q}$, both sets in the r.h.s. of \eqref{uc2} have the same probability and similarly both sets in the r.h.s. of \eqref{up1}
have the same probability.  Therefore we can combine \eqref{uc2}, \eqref{up} and \eqref{up1} to obtain
\begin{align}\label{ineqi}
\nonumber  P_{N,H_N^q}\big(W_{N,q})\leq 2 P_{N,H_N^q}\big(\exists i\in \{\tilde c \log N,&\dots,\tfrac{N}{2}\}\colon V_i\leq 0\big)\\
 &+ 2 P_{N,H_N^q}\big(A^-_{\,\tilde c \log N}(V)\geq \tfrac{c}{4} (\log N)^4\big).
\end{align}
As a consequence, the proof of Lemma \ref{refine4} will be complete once we show, on the one hand,  that for all  $\alpha>0$
there exists a $\tilde c>0$ such that 
\be{tpp3}
\lim_{N\to \infty} \sup_{q\in [q_1,q_2]} N^{\alpha} P_{N,H_N^q}\big(\exists i\in \{\tilde c \log N,\dots,\tfrac{N}{2}\}\colon\, V_i\leq 0 \big) =0,
\ee
and, on the other hand, that for all $\alpha,x, y>0$ we have 
\be{tpp4}
\lim_{N\to \infty} \sup_{q\in [q_1,q_2]} N^{\alpha} P_{N,H_N^q}\big(A^-_{\,x \log N}(V)\geq y (\log N)^4 \big) =0.
\ee

In order to prove \eqref{tpp3} and \eqref{tpp4} we recall  Lemma 6.2 in \cite{CNGP13},
\begin{lemma}\label{lem:deflambda}
For $[q_1,q_2]\subset (0,\infty)$ there exists $N_0\in \N$ and there exist three positive constants $C',C_1,\lambda$ such that for $N\geq N_0$ and  for every integer $j\leq N/2$, the following bound holds
\begin{equation}
\mathbf{E}_{N,H_N^q}\bigl[e^{-\lambda V_j}\bigr]\leq C'e^{-C_1 j},\quad N\in\mathbb{N}.
\end{equation}
\end{lemma}

To prove \eqref{tpp3}, we apply lemma \ref{lem:deflambda} directly and we obtain
\begin{align}\label{since1} 
P_{N,H_{N}^{q}}(\exists i\in \{\tilde c \log N,\dots,\tfrac{N}{2}\}\colon\, V_i\leq 0)&
\leq \sum_{j=\tilde c \log N}^{N/2}  P_{N,H_{N}^{q}}\big( e^{-\lambda V_j}\geq 1\big)\leq  C' \sum_{j=\tilde c \log N}^{N/2} e^{-C_1 j},
\end{align}
which suffices to complete the proof of \eqref{tpp3}.
For the proof of \eqref{tpp4}, we note that
\be{incl}
\big\{A^-_{\,x \log N}( V)\geq y (\log N)^4\big\}\subset \big\{\exists i\leq x \log N\colon \,
V_i\leq-\tfrac{y}{x} (\log N)^3\big\}
\ee
and we apply    lemma \ref{lem:deflambda} again to obtain
\begin{align}\label{since}
P_{N,H_N^q}(\exists i\leq  x \log N\colon \,
V_i\leq-\tfrac{y}{x} (\log N)^3)&\leq \sum_{j=1}^{x \log N} P_{N,H_N^q}\big(e^{-\lambda V_j}\geq 
e^{\tfrac{y \lambda}{x} (\log N)^3}\big) \\
& \leq  \sum_{j=1}^{x\log N} E_{N,H_N^q}( e^{-\lambda V_j}) \, e^{-\tfrac{y \lambda}{x} (\log N)^3}\\
& \leq C' x \log N   e^{-\tfrac{y \lambda}{x} (\log N)^3},
\end{align}
and this completes the proof of \eqref{tpp4}.

%
%

\section{Scaling Limits in the critical phase}\label{sec:critic}
In this section we will prove the items (2) of Theorems \ref{pfa} and \ref{pfa2} which correspond to the critical case $(\beta=\beta_c)$. For the sake of notations, we will write $\beta$ instead of $\beta_c$ until the end of this proof.
In Section \ref{prep} below,
 we first exhibit a renewal
structure for the underlying geometric random walk, based on
``excursions''. Then we state a  local limit theorem for
the area of such an excursion. This Theorem has been proven recently in \cite{DKW13}. 
With these tools in hand we will be able to prove Theorem \ref{pfa} (2)
in Section \ref{pfy}  and Theorem \ref{pfa2} (2) in Section \ref{pfy2}.  Finally, in Section \ref{secder}, we identify the limiting law of the
rescaled  horizontal extension obtained in Section \ref{pfy2} with that of the Brownian stopping time $g_1$ under 
a proper conditioning. 

\subsection{Preparations}\label{prep}

\subsubsection{\bf The renewal structure}

We introduce a sequence of stopping times $(\tau_k)_{k\in\N}$ which
look a lot like ladder times by the prescription $\tau_0=0$ and
\begin{equation}\label{premm}
  \tau_{k+1} = \inf\ens{i> \tau_k : V_{i-1}\neq 0 \text{ and }
    V_{i-1}V_i \le 0}\,.
\end{equation}
To these we associate 
\begin{equation}\label{deuxx}
  \mathfrak{N}_k = \tau_k -\tau_{k-1} \quad(k\ge 1)\,,
\end{equation}
and the sequence of areas sweeped
\be{trois}
\mathfrak{A}_k =\valabs{V_{\tau_{k-1}}} + \cdots + \valabs{V_{\tau_k-1}} \quad(k\ge 1)\,.
\ee
We let $\tau=\tau_1=\mathfrak{N}_1$. Let us observe that $\mathfrak{A}_1=G_{\tau-1}(V)$. For simplicity, we will drop the $V$ dependency of $G$ in what follows.

\begin{proposition}
  \label{pro:renstruccritical}
The random variables $(\mathfrak{A}_i,\mathfrak{N}_i)_{i\ge 1}$ are
independent and the sequence $(\mathfrak{A}_i,\mathfrak{N}_i)_{i\ge
  2}$ is IID.
\end{proposition}

We shall first need to study the distribution of $V_\tau$. Let $T$ be
a random variable with distribution geometric of parameter
$p_\beta=1-e^{-\beta/2}$ that is
$$ \prob{T=k} = e^{-\frac{\beta}{2} k} (1-e^{-\beta/2})
\qquad(k\in\N)\,,$$
and let $\mu_\beta$ be the law of the associated symmetric random variable, that
is $\mu_\beta$ is the distribution of $\epsilon T$ with $\epsilon$
independent from $T$ and $\prob{\epsilon = \pm 1} = \undemi$ :
\be{defubeta}
 \mu_\beta(k) = \prob{\epsilon T=k} = \frac{1-e^{-\beta/2}}{2}
e^{-\frac{\beta}{2} \valabs{k}} \un{k\neq 0} +  (1-e^{-\beta/2})
\un{k=0}.
\ee
Finally we let $\bP_{\beta,x}$ be the law of the random walk starting
  from $V_0=x\in \Z$ and $\bP_{\mu_\beta,\beta}$ be the law of the
  random walk when  $V_0$ has
distribution $\mu_\beta$.

\begin{lemma}
  Under  $\bP_{\beta,x}$  with $x\in \Z$ or under $\bP_{\beta,\mu_\beta}$,   the random variable $V_\tau$ is independent from the
  couple $(G_{\tau -1},\tau)$. Moreover 
\begin{itemize}
\item $V_\tau=_{\text{law}}\ -T$ \ under $\bP_{\beta,x}$ with $x>0$, 
\item $V_\tau=_{\text{law}}\ T$ \ under $\bP_{\beta,x}$ with $x<0$,
\item $V_\tau=_{\text{law}}\ \mu_{\beta}$ \ under $\bP_{\beta}$,
\item $V_\tau=_{\text{law}}\ \mu_\beta$ \ under $\bP_{\beta,\mu_\beta}$.
\end{itemize}
  

%
%
\end{lemma}
\begin{proof}
  Let $x>0, y\ge 0$ and $a,n$ be integers. Under $\bP_{\beta,x}$  we compute the probability of $T_{n,a,y}:=\{G_{\tau -1}=a,\tau=n,V_\tau=-y\}$ by disintegrating it with respect to the 
 value $z>0$  taken by  $V_{n-1}$, i.e., 
\begin{align}\label{alift}
\nonumber \bP_{\beta,x}(G_{\tau -1}&=a,\tau=n,V_\tau=-y)\\
\nonumber &= \sum_{z>0}\probbetax{V_1+\cdots+V_{n-1}=a; V_i>0,1\le
  i\le n-2;V_{n-1}=z} \frac{e^{-\frac{\beta}{2}(z+y)}}{c_\beta}\\
&= \gamma_\beta \probbetax{G_{\tau -1}=a,\tau=n} e^{-\frac{\beta}{2}y}
\end{align}
where $\gamma_\beta$ can be seen as a normalizing constant for a
distribution on the non negative integers, we obtain that
$\gamma_\beta=p_\beta$ and $V_\tau$ is independent of
$(G_{\tau -1},\tau)$ and distributed as $-T$.
For $x<0$ the proof is exactly the same. 
For $x=0$,  we take into
account the possibility that
the walk sticks to zero for  a while. Thus, for $y>0$, we partition the event $\{G_{\tau -1}=a,\tau=n,V_\tau=-y\}$ depending on the 
 value $z>0$  taken by  $V_{n-1}$ and on the number of steps $k$ during which the random walks sticks to $0$, i.e.,
 
 \begin{multline}
  \probbeta{G_{\tau -1}=a,\tau=n,V_\tau=-y}\\
=\sum_{z>0,0\le k\le n-2}\probbetax{A_{n-1}=a; V_i=0, 0\le
  i\le k; V_i>0,k<
  i\le n-2;V_{n-1}=z} \frac{e^{-\frac{\beta}{2}(z+y)}}{c_\beta}\\
= \kappa e^{-\frac{\beta}{2}y}, \qquad \hfill
\end{multline}
and we obtain by symmetry $\probbeta{G_{\tau -1}=a,\tau=n,V_\tau=y}=\kappa  e^{-\frac{\beta}{2}y}$.
It is only for $x=y=0$ that we need to take into account positive and
negative excursions, and we obtain
$$ \probbeta{G_{\tau -1}=a,\tau=n,V_\tau=0}= 2 \kappa\,.$$
Summing all these probabilities yields
$$\kappa=\frac{1-e^{-\beta/2}}{2}\probbeta{G_{\tau -1}=a,\tau=n} $$
 and we
can conclude that the random variable $V_\tau$ is independent from the
  couple $(G_{\tau -1},\tau)$ with distribution $\mu_\beta$.

The final computation showing that $\mu_\beta$ is an ``invariant
measure'', is straightforward:
\begin{align}
  \probmubeta{V_\tau =y} &= \sum_{x} \mu_\beta(x)
\probbetax{V_\tau=y}\\
&= \sum_{x} \mu_\beta(x) (\prob{-T=y} \un{x>0} + \prob{T=y} \un{x<0} +
\mu_\beta(y) \un{x=0})\\
&= \ldots = \mu_\beta(y)
\end{align}
\end{proof}

\begin{proof}[Proof of Proposition~\ref{pro:renstruccritical}] The
  proof is based on the preceding Lemma, and uses induction in
  conjunction with the Markov property. Our induction assumption is
  thus that the sequence $(\mathfrak{A}_i,\mathfrak{N}_i)_{1\le i\le
    k}$ is independent, that the subsequence $(\mathfrak{A}_i,\mathfrak{N}_i)_{2\le i\le
    k}$ is IID, and that the random variable $V_{\tau_k}$ is
  independent of $(\mathfrak{A}_i,\mathfrak{N}_i)_{1\le i\le
    k}$, with distribution $\mu_\beta$. Then,
  \begin{multline*}
    \probbeta{(\mathfrak{A}_i,\mathfrak{N}_i)=(a_i,n_i), 1\le i \le
      k+1 ; V_{\tau_{k+1}}=y} = \\
\espbeta{\ind_{\{(\mathfrak{A}_i,\mathfrak{N}_i)=(a_i,n_i), 1\le i \le
      k\}}\bP_{V_{\tau_k}}\etp{(\mathfrak{A}_1,\mathfrak{N}_1)=(a_{k+1},n_{k+1})
    ;  V_{\tau}=y}}
\\
= \mu_\beta(y) \probmubeta{A_{\tau -1}=a_{k+1},\tau=n_{k+1}}\,  \probbeta{(\mathfrak{A}_i,\mathfrak{N}_i)=(a_i,n_i), 1\le i \le
      k}\,,
  \end{multline*}
and this concludes the induction step.
  
\end{proof}

\subsubsection{\bf Local limit theorem for the ``excursion'' area}

Let us state first the theorem. Here $f_{ex}$ stands for the density
of the standard Brownian excursion (see e.g. Janson~\cite{J07}).
\begin{theorem}[Theorem 1 of \cite{DKW13}]
  \label{theo:wachtellclt}
With the constant $C_\beta:=(\bE_\beta(V_1^2))^{-1/2}$ and with $w(x)=C_\beta\, f_{ex}(C_\beta\, x)$, we have
$$ \lim_{n\to +\infty} \sup_{a\in\Z} \valabs{n^{3/2} \probbeta{G_n=a \mid \tau=n} -
  w(a/n^{3/2})} = 0\,.$$
\end{theorem}

\begin{remark}\label{rem:cvdom}From the monograph~\cite{J07} we extract the asymptotics
  (formulas (93) and (96))
  \begin{gather}
    \label{eq:asympfex}
    f_{ex}(x) = e^{-\frac{C_1}{x^2}} ( C_2 x^{-5} + o(x^{-5})
    \quad(x\to 0)\\
f_{ex}(x) = C_3 x^2 e^{-6 x^2} (1 + o(1)) \quad(x \to +\infty).
  \end{gather}
A straightforward application of the dominated convergence theorem
entails that if $b_n\to \infty$ then for all $\eta>0$, 
\begin{equation} \label{eq:cvdominpuisnonentiere}
\unsur{n^{2/3}} \sum_{\eta n^{2/3} \le k\le b_n n^{2/3}}
\etp{\frac{k}{n^{2/3}}}^{-3} w\etp{\etp{\frac{k}{n^{2/3}}}^{-3/2}} \to
\int_\eta^{+\infty} t^{-3} w(t^{-3/2})\, dt.
\end{equation}
\end{remark}

It is easy to check from \cite{DKW13} that  this theorem still holds when started from the
``invariant measure'' $\mu_\beta$. More precisely, $R_n\to 0$ with

\be{errorterm} 
R_n :=\sup_{a\in\Z} \valabs{n^{3/2} \probmubeta{G_n=a \mid \tau=n} -
  w(a/n^{3/2})}
\ee
Without loss in generality we shall assume that $R_n$ non-increasing.


\subsection{Proof of Theorem \eqref{pfa} (2)  }\label{pfy}

Using the random walk representation \eqref{tgh}, we obtain, since
$\Gamma_\beta=1$, that the excess partition function is
\begin{align}\label{parffu}
\tfrac{1}{c_\beta} \tilde{Z}_{L,\beta} &= \sum_{N=1}^L
\probbeta{G_N=L-N,V_{N+1}=0}.
\end{align}
Then, we partition the event  $\{G_N=L-N,V_{N+1}=0\}$ depending on the length $r$
on which the random walk sticks at the origin before its right extremity, that is
\begin{align}
\tfrac{1}{c_\beta} \tilde{Z}_{L,\beta}   &=\sum_{r=0}^{L-1} \sum_{N=1}^{L-r}
\probbeta{G_{N}=L-N-r, V_N\neq 0, V_{N+1}=0}  \probbeta{V_1=\dots=V_r=0}\\
\nonumber &\hspace{5cm} +\probbeta{V_1=\dots=V_L=0}\\
\nonumber &=\big(\tfrac{1}{c_\beta}\big)^L\,+ \sum_{r=0}^{L-1}\big(\tfrac{1}{c_\beta}\big)^r \sum_{N=1}^{L-r}
\probbeta{G_{N}=L-N-r, V_N\neq 0, V_{N+1}=0} .
\end{align}
Then we use the fact that, for all $x\in \N$, we have 
$\probbeta{U_1=x}/\probbeta{U_1\geq x}=1-e^{\beta/2}$, and therefore

\begin{align}\label{TTF}
\nonumber \tfrac{1}{c_\beta} \tilde{Z}_{L,\beta}   &=\big(\tfrac{1}{c_\beta}\big)^L\,+ (1-e^{-\beta/2}) \sum_{r=0}^{L-1}\big(\tfrac{1}{c_\beta}\big)^r \sum_{N=1}^{L-r}
\probbeta{G_{N}=L-N-r, V_N\neq 0, V_N V_{N+1}\leq 0} \\
 &=\big(\tfrac{1}{c_\beta}\big)^L\,+ (1-e^{-\beta/2}) \sum_{r=0}^{L-1}\big(\tfrac{1}{c_\beta}\big)^r \, 
\probbeta{L-r+1 \in \mathfrak{X}},
\end{align}
where 
\begin{equation}\label{XX}
\mathfrak{X} =\Big\{\sum_{k\le n} \mathfrak{A}_k + \mathfrak{N}_k ; n
  \ge 1\Big\}
\end{equation}
is the renewal set associated to the  sequence of random variables
$X_k :=  \mathfrak{A}_k + \mathfrak{N}_k$ (recall (\ref{premm}--\ref{trois}).

It is clear that we are
going to obtain the same asymptotics for $\tilde{Z}_{L,\beta}$ if we substitute $\bP_{\beta,\mu_\beta}$ to 
 $\bP_{\beta}$ in the r.h.s. of \eqref{TTF}, that is  
if we consider a true
renewal process with the random variable $X_1$ having the same
distribution as the $X_i$ for $i\ge 2$.
Thus, the proof of Theorem \eqref{pfa} (2)   will be a consequence of the tail estimate of $X$ under $\bP_{\beta,\mu_\beta}$ in the next lemma.
\begin{lemma}\label{tes} For $\beta>0$, there exists a $c_{1,\beta}>0$ such that 
\begin{equation} 
  \label{eq:asympatauplustau}
  \bP_{\beta,\mu_\beta}(X_1= n)=\frac{c_{1,\beta}}{n^{4/3}} (1+o(1)),
\end{equation}
and $c_{1,\beta}=(1+e^{\beta/2})\,\sqrt{\frac{\espbeta{V_1^2}}{2 \pi}} \int_0^{+\infty}
  x^{-3} w(x^{-\frac32})\, dx$.
\end{lemma}
By applying  Doney's \cite{Doney97}, Theorem B (see also Theorem A.7 of \cite{cf:Gia}) we deduce from \eqref{eq:asympatauplustau} that
\be{attain}
\bP_{\beta,\mu_\beta}(L \in \mathfrak{X}) =  \frac{\sin(\pi/3)}{3 \pi c_{1,\beta} L^{2/3}} (1+o(1)).
\ee
Then, it suffices to recall  (\ref{TTF}-\ref{XX}) to complete the proof.

\medskip
  \begin{proof}[Proof of \ref{eq:asympatauplustau}]

We recall that $\tau_1=\mathfrak{N}_1$ and we drop the index $1$ for simplicity. First, we use that for all $j,k\in \N$
$$\frac{\probmubeta{V_{j+1}-V_j\geq k}}{\probmubeta{V_{j+1}-V_j= k}}\geq \frac{1}{1-e^{-\frac{\beta}{2}}}$$
to write 
\be{apr}
 \probmubeta{X_1 = n} =\probmubeta{A_{\tau -1} + \tau = n}=\probmubeta{A_{\tau -1} + \tau = n, V_\tau=0} \tfrac{1}{1-e^{-\frac{\beta}{2}}}
\ee
 and then  
split the probability of the r.h.s. in \eqref{apr} into two terms:
\begin{align}\label{auta}
\nonumber &\probmubeta{A_{\tau -1} + \tau = n, V_\tau=0} \\
\nonumber & \hspace{.2 cm}= \sum_{k=1}^{\eta n^{2/3}} \probmubeta{A_{k-1} = n-k, \tau=k, V_k=0} 
 + \sum_{ k=\eta\,  n^{2/3}}^{ n} \probmubeta{A_{k-1} = n-k, \tau=k,V_k=0} \\
&\hspace{.2 cm} =: u_n + v_n.
  \end{align}
From this equality, the proof will be divided into two steps. The first step consists in controlling 
$u_n$ and the second step $v_n$.

\subsubsection{Step 1}

Our aim is to show that for all $\gep>0$ there exists an $\eta>0$ such that 
\be{limsu}
 \limsup_{n\to \infty} n^{4/3} u_n\leq \gep
 \ee
 
\begin{proof} 
In this step, we will need an improved version of the local limit Theorem established in \cite[Proposition 2.3]{CD08} for $V_n$ and $A_n$
simultaneously. 
\begin{proposition}\label{lltalgarea}
\be{ala}
\sup_{n\in \N} \, \sup_{k,a\in \Z} n^3\,  \Big| \bP_\beta(V_n=k, A_n(V)=a)-g\Big(\tfrac{k}{\sigma_\beta \sqrt{n}}, \tfrac{a}{\sigma_\beta n^{3/2}}\Big)\Big|< \infty,
\ee
with $g(y,z)=\frac{6}{\pi} e^{-2 y^2-6 z^2 +6 yz}$ for $(y,z)\in \R^2$.
\end{proposition}

\begin{proof}[Proof of Proposition  \ref{lltalgarea}]
Compared to what is done in \cite{CD08}, the improvement in \eqref{ala} comes from the fact that the increments of the Random walk
under $\bP_\beta$ have a finite fourth moment and a third moment that is null.  The proof is performed by 
using Fourier inversion formula, which can be done for instance by mimicking the proof of 
\cite[Theorem 2.3.10]{LL10}. For this reason we will not repeat it here.  
  
\end{proof}

We resume the proof of \eqref{limsu} by bounding from above the probability that there exists a piece of the $V=(V_i)_{i=0}^n$ trajectory 
of length smaller than $\frac{n^{2/3}}{\log n}$ with an algebraic  area (seen from its starting point)  that is larger than $\frac{n}{2}$ and/or that 
one of the increments of $V$ is larger than $(\log n)^2$.  Thus, we set  $\cB_n:=\cC_n\cup \cD_n$ with
$$\cC_n:=\bigcup_{i\in \{0,\dots,n-1\}} \{V_{i+1}-V_i\geq (\log n)^2\}\quad \text{and} \quad  \cD_n:=  \bigcup_{(j_1,j_2)\in J_n }  \{A_{j_1,j_2} -V_{j_1} (j_2-j_1)\geq \tfrac{n}{2}\}$$
where 
$J_n=\big\{(j_1,j_2)\in \{0,\dots,n\}^2\colon 0\leq j_2-j_1\leq \frac{n^{2/3}}{\log n}\big \}$ and $A_{s,t}=\sum_{i=s}^{t-1} V_i$. Then, for each 
$(j_1,j_2)\in J_n$ we apply Markov property at $j_1$ and  we get
\be{trayu}
\bP_{\beta,\mu_\beta} (\cD_n)\leq \sum_{(j_1,j_2)\in J_n} \bP_\beta(A_{j_2-j_1}\geq \tfrac{n}{2}).
\ee
Since under $\bP_{\beta,\mu_\beta}$, the random variable $V_1$ has
small exponential moments, there exists a constant $C>0$ such that
\begin{equation}
  \label{eq:majvstarcritical}
  \probmubeta{\sup_{1\le k\le n} \valabs{V_k} \ge x n} \le 2 e^{-C n x
  \wedge x^2} \qquad(x\ge 0, n\in \N)\,.
\end{equation}
We note that $A_{j_2-j_1}\geq \tfrac{n}{2}$ implies $\max \{|V_i|, i=1,\dots,j_2-j_1\}\geq \frac{n}{2(j_2-j_1)}$ so that finally 
we can use \eqref{eq:majvstarcritical} to prove that there exists $C''>0$ such that
\begin{align}
\nonumber \sup_{(j_1,j_2)\in J_n} \bP_\beta(A_{j_2-j_1}\geq \tfrac{n}{2}) &\leq  \bP_\beta\bigg(\sup_{1\leq j\leq n^{2/3}/\log(n)} |V_j|\geq \tfrac{n^{1/3}}{2} \log (n) \bigg)\\
 &\leq 2 e^{-C'' \log(n)^3},
\end{align} 
which (recall \eqref{trayu}) suffices to claim that $\bP_{\beta,\mu_\beta} (\cD_n)=o(1/n^{4/3})$.
Moreover, $\bP_{\beta,\mu_\beta}(V_1\geq (\log n)^2)\leq c e^{-\frac{\beta}{2} \log(n)^2}$ suffices to conclude that $\bP_{\beta,\mu_\beta}(\cC_n)=o(1/n^{4/3})$ which,
in turn, implies that $\bP_{\beta,\mu_\beta}(\cB_n)=o(1/n^{4/3})$.  
  
At this stage,  for $k\leq \eta n^{2/3}$,  we can partition the set $\{A_k=n-k, \tau=k, V_k=0\}$ depending on the indices at which 
a trajectory passes above $\sqrt{k}$ for the first and the last time. Thus, we set
$\xi_{\sqrt{k}}=\inf \{i\geq 1 \colon\, V_i\geq \sqrt{k}\}$ and $\widehat \xi_{\sqrt{k}}=\max \{i\leq k \colon\, V_i\geq \sqrt{k}\}$. We also consider the 
positions of $V$ at  $\xi_{\sqrt{k}}$ and $\widehat \xi_{\sqrt{k}}$ and the algebraic areas below $V$ in-between  $0$ and $\xi_{\sqrt{k}}$,  
$\xi_{\sqrt{k}}$  and $\widehat \xi_{\sqrt{k}}$ as well as $\widehat \xi_{\sqrt{k}}$ and $k$. Thus, 
 we set $\bar t=(t_1,t_2)$, $\bar x=(x_1,x_2)$, $\bar a=(a_1,a_2)$  and we write   
 
 \begin{equation}\label{decsub}
\{A_k=n-k, \tau=k, V_k=0\}= \cup_{(\bar t,\bar x,\bar a)\in G_{k,n} }  Y_{n,k}(\bar t,\bar x,\bar a),
\end{equation}
with
\begin{align}
Y_{n,k}(\bar t,\bar x,\bar a):=\Big\{\tau=k,\, &  \xi_{\sqrt{k}}=t_1,\,  \widehat \xi_{\sqrt{k}}= k-t_2,\\
\nonumber & V_{t_1}=x_1, V_{k-t_2}=x_2, V_k=0,\\
\nonumber &\hspace{2cm}  A_{t_1}=a_1,
A_{t_1,k-t_2}=n-k-a_1-a_2, A_{k-t_2,k}=a_2\Big\},
\end{align}
and with

 \begin{align}
\nonumber  G_{k,n}=\Big\{(\bar t,\bar x,\bar a)\in \N^6\colon\,  0\leq t_1\leq k-t_2 &\leq k-1,\\
  \nonumber                                                                           0\leq a_1&\leq k^{3/2}, \ 0\leq a_2\leq k^{3/2},\\
                                                                         & x_1\geq \sqrt{k},\  x_2 \geq \sqrt{k}\Big\}
\end{align}  
Then we set 
 \begin{align}\label{gti}
  \widetilde{G}_{k,n}=\big\{(\bar t,\bar x,\bar a)\in G_{k,n}\colon\,  k-t_1-t_2\leq \tfrac{n^{2/3}}{\log n}\quad  \text{and/or}\quad  x_1-\sqrt{k}> \log (n)^2\big\}
\end{align} 
and we note that if $(\bar t,\bar x,\bar a)\in \cup_{k=1}^{\eta n^{2/3}}  \widetilde G_{k,n}$, then either $x_1-\sqrt{k}\geq (\log n)^2$ and 
 $Y_{n,k}(\bar t,\bar x,\bar a)\subset \cC_n$ or $x_1\leq \sqrt{k}+ (\log n)^2$ and $k-t_1-t_2\leq n^{2/3}/\log n$ and then 
 \begin{align}\label{tbrlb}
  A_{t_1,k-t_2}-V_{t_1} (k-t_1-t_2)&=n-k-a_1-a_2-x_1(k-t_1-t_2)\\
\nonumber   &\geq n-k-3k^{3/2}-(\log n)^{2} k\\
\nonumber &\geq n-\eta n^{2/3} -3 \eta^{3/2} n -\log(n)^2 \eta n^{2/3}\geq \frac{n}{2},
\end{align}
provided $\eta$ is chosen small enough. Thus,  $Y_{n,k}(\bar t,\bar x,\bar a)\subset \cD_n$ so that
$$\cup_{k=1}^{\eta n^{2/3}} \cup_{(\bar t,\bar x,\bar a)\in \widetilde G_{k,n} }  Y_{n,k}(\bar t,\bar x,\bar a)\subset \cB_n.$$
Clearly, for $k\leq n^{2/3}/\log(n)$ we have $G_{k,n}=\widetilde G_{k,n}$ so that we should simply focus on bounding from above
$$\bP_{\beta,\mu_\beta}\Big(\cup_{ k=\frac{n^{2/3}}{\log n}}^{ \eta n^{2/3}}\,  \cup_{(\bar t,\bar x,\bar a)\in G_{k,n}\setminus \widetilde{G}_{k,n}} 
Y_{n,k}(\bar t,\bar x,\bar a) \Big).$$

Pick $n^{2/3}/\log n\leq k\leq \eta n^{2/3}$ and $(\bar t,\bar x,\bar a)\in G_{k,n}\setminus \widetilde{G}_{k,n}$. 
By applying Markov property at $t_1$ and $k-t_2$ we can write

\begin{equation}\label{decompo}
  \probmubeta{Y_{n,k}(\bar t,\bar x,\bar a)}= S_{1} S_2 S_3
\end{equation}
with 
\begin{align}
\nonumber S_1&=  \bP_{\beta,\mu_\beta} (\tau>t_1, V_{t_1}=x_1,\xi_{\sqrt{k}}=t_1, A_{t_1}=a_1)\\
\nonumber S_2&=  \bP_{\beta, x_1}(\tau>k-t_1-t_2, V_{k-t_1-t_2}=x_2-x_1, A_{k-t_1-t_2}=n-k-a_1-a_2)\\
S_3&=  \bP_{\beta}(\tau>t_2, V_{t_2}=x_2,\xi_{\sqrt{k}}=t_2, A_{t_2}=a_2).
\end{align}
Since we are looking for an upper bound of the r.h.s. in \eqref{decompo}, we can easily remove the restriction  $\{\tau>k-t_1-t_2\}$ in   $S_2$ 
and write
\begin{align}\label{SS2}
S_2&\leq  \bP_{\beta}( V_{k-t_1-t_2}=x_2-x_1, A_{k-t_1-t_2}=n-k-a_1-a_2-x_1(k-t_1-t_2))
\end{align}

%
%

Therefore, it remains to bound 
$$\sum_{k=n^{2/3}/\log n}^{\eta n^{2/3}} \sum_{(\bar t,\bar x,\bar a)\in G_{k,n}\setminus\widetilde G_{k,n} } 
  S_{1} S_2 S_3$$
 and  we recall \eqref{SS2} and Proposition  \ref{lltalgarea} that yield $S_2\leq  \widetilde{S}_2+ \widehat{S}_2$ with  $\widetilde{S}_2 = \tfrac{C_1}{(k-t_1-t_2)^3}$ and 
  \begin{align}
 \widehat{S}_2 = \tfrac{C_2}{ (k-t_1-t_2)^2}\,  
 g(\tfrac{x_2-x_1}{\sigma_\beta \sqrt{k-t_1-t_2}}, \tfrac{n-k-a_1-a_2-x_1(k-t-1-t_2)}{(k-t_1-t_2)^{3/2}})
 \end{align}
 with $g(y,z)=\frac{6}{\pi} e^{-2 y^2-6 z^2 +6 yz}\leq \frac{6}{\pi} e^{-\frac{3}{2} z^2} $. We recall \eqref{tbrlb} and we write
 \begin{align}
 \widetilde{S}_2 = \tfrac{C_1}{(k-t_1-t_2)^3} \quad \quad   \widehat{S}_2 =\tfrac{C_2}{ (k-t_1-t_2)^2} e^{-\frac{3}{8} \frac{n^2}{(k-t_1-t_2)^3}} 
 \end{align}
 and then for all $(\bar t,\bar x,\bar a)\in  G_{k,n}\setminus\widetilde G_{k,n}$ we have $ \widetilde{S}_2 \leq C_1 \log(n)^3/n^2$ and at the same time 
\be{bobs1}
\sum_{(\bar t,\bar x,\bar a)\in G_{k,n} } 
S_{1} S_3\leq \bP_{\beta,\mu_\beta} (\xi_{\sqrt{k}}<\tau_0)  \bP_{\beta} (\xi_{\sqrt{k}}<\tau_0)\leq \frac{C_3}{k}. 
\ee
Thus,  $$\sum_{k=n^{2/3}/\log n}^{\eta n^{2/3}} \sum_{(\bar t,\bar x,\bar a)\in G_{k,n}\setminus\widetilde G_{k,n}} 
  S_{1} \widetilde{S}_2 S_3\leq \sum_{k=n^{2/3}/\log n}^{\eta n^{2/3}} \frac{C_4 \log(n)^3}{n^2 k}\leq \frac{C_5 \log(n)^4}{n^2}=o(1/n^{4/3}).$$
  It remains to bound from above
  $\sum_{k=n^{2/3}/\log n}^{\eta n^{2/3}} \sum_{(\bar t,\bar x,\bar a)\in G_{k,n}\setminus\widetilde G_{k,n} } 
  S_{1} \widehat{S}_2 S_3$.
  We rewrite 
  $$\widehat{S}_2=\tfrac{C_2}{n^{4/3}} \big[\tfrac{n^{2/3}}{ (k-t_1-t_2)}\big]^2 e^{-\frac{3}{8} \big(\frac{n^{2/3}}{k-t_1-t_2}\big)^3}$$
  and we note that $x^2 e^{-3 x^3/8}\leq e^{-x^3/4}$ for $x$ large enough. Since 
 $k\leq \eta n^{2/3}$ it comes that $n^{2/3}/(k-t_1-t_2)\geq n^{2/3}/k\geq 1/\eta$ so that by choosing $\eta$ small enough we get
\be{bobs2}
\widehat{S}_2\leq \tfrac{C_2}{n^{4/3}}  e^{-\frac{1}{4} \big(\frac{n^{2/3}}{k-t_1-t_2}\big)^3}\leq  \tfrac{C_2}{n^{4/3}}  e^{-\frac{1}{4} \big(\frac{n^{2/3}}{k}\big)^3}
\ee
 and then we use \eqref{bobs1} and \eqref{bobs2} to get
\begin{align}\label{finfo}
\nonumber\sum_{k=n^{2/3}/\log n}^{\eta n^{2/3}} \sum_{(\bar t,\bar x,\bar a)\in G_{k,n}\setminus\widetilde G_{k,n} } 
S_{1} \widehat{S}_2 S_3&\leq \sum_{k=n^{2/3}/\log n}^{\eta n^{2/3}} \tfrac{C_2}{n^{4/3}}  e^{-\frac{1}{4} \big(\frac{n^{2/3}}{k}\big)^3} \sum_{(\bar t,\bar x,\bar a)\in G_{k,n}} 
S_{1}  S_3\\
\nonumber & \leq \sum_{k=n^{2/3}/\log n}^{\eta n^{2/3}} \tfrac{C_2}{ k n^{4/3}}  e^{-\frac{1}{4} \big(\frac{n^{2/3}}{k}\big)^3}\\
& \leq  \tfrac{C_2}{n^{4/3}} \bigg[ \tfrac{1}{ n^{2/3}} \sum_{k=n^{2/3}/\log n}^{\eta n^{2/3}}  \tfrac{n^{2/3}}{ k }  e^{-\frac{1}{4} \big(\frac{n^{2/3}}{k}\big)^3}\bigg]
\end{align}
and the Riemann sum between brackets above converges to $\int_0^\eta (1/x) e^{-1/(4 x^3)} dx$ so that the r.h.s. in \eqref{finfo}
is smaller that $\gep/n^{4/3}$ as soon as $\eta$ is chosen small enough and this completes the proof.

  \end{proof}

\subsubsection{Step 2}    
Our aim is to show that for all $\eta>0$,
\be{condu}
\lim_{n\to \infty} n^{4/3} v_n= (1+e^{\beta/2})\,\sqrt{\frac{\espbeta{V_1^2}}{2 \pi}} \int_\eta^{+\infty}
  x^{-3} w(x^{-\frac32})\, dx
\ee
    

\begin{proof}
By Theorem 8 of
Kesten~\cite{Kes63} (see also Theorem A.11 of \cite{cf:Gia}) and  since
$\espbeta{V_1^2}<+\infty$, we can state that $\probbeta{\tilde \tau=n} \sim C n^{-3/2}\quad \text{with}\quad C=(\espbeta{V_1^2}/2 \pi)^{1/2}$ and with $\tilde\tau=\inf\{i\geq 1\colon\, V_i\leq 0\}$ which may differ from $\tau$ (recall \ref{premm}) when $V_0=0$ only. In Appendix \ref{lltt}, we extend this local limit theorem to the random walk with initial
distribution $\mu_\beta$ and we obtain
\begin{equation}
\label{eq:equivptauegaln}
\probmubeta{\tau=n} \sim C_\tau n^{-3/2}\quad \text{with}\quad C_\tau=(1+e^{\beta/2})\, \sqrt{\frac{\espbeta{V_1^2}}{2 \pi}}.
\end{equation}

 Let 

$$ v'_n := \sum_{\eta n^{2/3} \le k\le n} \probmubeta{\tau=k}
k^{-3/2} w\etp{\frac{n-k}{k^{3/2}}} \,.$$
We recall the definition of $R_n$  in \eqref{errorterm} and   we write
\begin{align}\label{vn}
\nonumber \valabs{v_n - v'_n} &\le \sum_{\eta n^{2/3} \le k\le n} \probmubeta{\tau=k}
k^{-3/2} R_k \le C' R_{\eta n^{2/3}}   \sum_{\eta n^{2/3} \le k\le n} k^{-3}\\
&\le C'' R_{\eta n^{2/3}}\unsur{\eta^2 n^{4/3}}  = o(n^{-4/3})\,.
\end{align}
We can establish  by  dominating convergence (see Remark \ref{rem:cvdom}) that 
\be{convint}
n^{4/3} v'_n \to C_\tau \int_\eta^{+\infty}t^{-3} w(t^{-3/2})\, dt.
\ee
By putting together \eqref{vn}, \eqref{convint} we obtain \eqref{condu}
and this completes the proof.
\end{proof}

\end{proof}

\subsection{Proof of Theorem \ref{pfa2} \ (2)}\label{pfy2}

Let  $(\mathfrak{U}_i)_{i=1}^\infty$ be the sequence of inter-arrivals of a $1/3$-stable regenerative set  \footnote{We refer to  \cite[Appendix A]{CSZ}  for a self-contained introduction of the 
 $\alpha$-stable regenerative sets on $[0,1]$ (see also
 \cite{Ber96}). In fact, it is useful to keep in mind that such a set is the limit in distribution of 
 the set $\frac{\tau}{N} \cap[0,1]$ when $\tau$ is a regenerative process on $\N$ with an inter-arrival law $K$
 that satisfies $K(n)\sim L(n)/n^{1+\alpha}$ and with $L$ a slowly varying function. The alpha-stable regenerative set can also be viewed as the 
 zero set of  a Bessel bridge on $[0,1]$ of dimension
 $d=2(1-\alpha)$.} 
 $\taufrak$ on $[0,1]$, conditioned on $1\in\taufrak$ and denote  by $(\mathfrak{U}^i)_{i=0}^\infty$ its order statistics.
 Let  $(Y_i)_{i=1}^\infty$ be an IID sequence of continuous random
 variables, independent of $(\mathfrak{U}_i)_{i=0}^\infty$, with density 
\be{density}
\textstyle \dd\bP_{Y_1}(x)\propto \frac{1}{x^3}\, w\big(\frac{1}{x^{3/2}}\big) \ind_{\R^+}(x).
\ee
We are first going to prove that 
\begin{equation}
  \label{eq:limmextcritique}
  \lim_{L\to +\infty} \frac{N_l}{L^{2/3}} =_{law} \sum_{i=1}^{+\infty}
  Y_i\,  (\mathfrak{U}^i)^{2/3}=_{law} \sum_{i=1}^{+\infty}
  Y_i\,  \mathfrak{U}_i^{2/3}
\end{equation}
where the second identity in law in \eqref{eq:limmextcritique} is straightforward
and then we shall identify the distribution of $\sum_{i=1}^{+\infty}
  Y_i\,  \mathfrak{U}_i^{2/3}$ with the distribution of $g_1$ conditionnaly on $B_{g_1}=0$.
  \smallskip

We recall (\ref{premm}-\ref{trois}) and we consider the i.i.d. sequence of random vectors $(\mathfrak{N}_i,\mathfrak{A}_i)_{i=1}^{\infty}$ and we recall 
that $X_i=\mathfrak{N}_i+\mathfrak{A}_i$ for $i \in \N$. We recall that, under $\bP_\beta$, the first excursion has law $\bP_{\beta,0}$ and the next excursions 
have law $\bP_{\beta,\mu_\beta}$.  Let us set $S_n=X_1 + \cdots + X_n$ and
$v_L:=\max\{i\geq 0\colon S_i\leq L\}$. We recall \eqref{XX} and we consider the sequence $$(\mathfrak{N}_i,\mathfrak{A}_i,X_i)_{i=1}^{v_L}$$
under the law $P_\beta(\cdot | L\in \mathfrak{X} )$. We denote by 
$X_{r_1} \geq \dots \geq X_{r_{v_L}}$  the order statistics of $(X_i)_{i=1}^{v_L}$ such that if $X_{r_i}=X_{r_j}$ and $i<j$ then $r_i<r_j$. 
To simplify notations  we set $(\mathfrak{N}^i,\mathfrak{A}^i, X^i)=(\mathfrak{N}_{r_i},\mathfrak{A}_{r_i}, X_{r_i})$ for $i\in \{1,\dots,L\}$.

To begin with, we will prove \eqref{eq:limmextcritique}  subject to Propositions \ref{1} and Claim \ref{2} below. Then, the
remainder of this section will be dedicated to the proof of Propositions \ref{1} and Claim \ref{2}.

\begin{proposition}\label{1}
\be{tobpr}
\lim_{L\to \infty} \frac{\sum_{i=1}^{v_L}  \mathfrak{N}_i}{L^{2/3}}=_{\text{Law}} \sum_{i=1}^\infty Y_i\,  (\mathfrak{U}^i)^{2/3}.
\ee

\end{proposition}

\begin{claim}\label{2}
For $\beta=\beta_c$, $t\in [0,\infty)$ and $L$ large enough, 
\be{mixmix}
P_{L,\beta}\bigg(\frac{N_l}{L^{2/3}}\leq t\bigg)  = \sum_{r=0}^{t L^{2/3}-1} \xi_{r,L} \, \bP_\beta\bigg( \frac{r+\sum_{i=1}^{v_{L-r+1}}  \mathfrak{N}_i}{L^{2/3}}\leq t
\, | \, L-r+1 \in \mathfrak{X}\bigg)
\ee
with 
\be{defalp}
\xi_{r,L} =\frac{(1-e^{-\beta/2}) \probbeta{L-r+1 \in \mathfrak{X}}}{ c_\beta^{r-1}\,  \widetilde Z_{L,\beta}}, \quad r\in \{0,\dots,L\}. 
\ee
\end{claim}

Pick $t\in [0,\infty)$ and $\gep>0$. With Claim \ref{2}  and with \eqref{TTF} and \eqref{attain} we obtain that there exists an $r_\gep \in \N$  such that, provided $L$ is chosen large enough, we have   $\sum_{r=0}^{r_\gep} \xi_{r,L}\in [1-\gep ,1]$
and $\sum_{r\geq r_\gep} \xi_{r,L} \leq \gep$. Then, it suffices to apply Proposition \ref{1} to each 
probability indexed by $r\in \{1,\dots,r_\gep\}$ in the r.h.s. of \eqref{mixmix} to conclude that, for $L$ large enough
\be{vibound}
(1-\gep) P\bigg(\sum_{i=1}^\infty Y_i\,  (\mathfrak{U}^i)^{2/3}\leq t\bigg)-\gep\leq P_{L,\beta}\Big(\frac{N_l}{L^{2/3}}\leq t\Big)\leq 2 \gep+ P\bigg( \sum_{i=1}^\infty Y_i\,  (\mathfrak{U}^i)^{2/3}\leq t\bigg),
\ee
which completes the proof of \eqref{eq:limmextcritique}.


\subsubsection{Proof of Proposition \ref{1}}
To begin with, let us distinguish between the $k$ excursions associated with the first $k$ variables of the order statistics $(X^i)_{i=1}^{v_L}$ and   the others, i.e.,   
\be{ortu}
\frac{\sum_{i=1}^{v_L}  \mathfrak{N}_i}{L^{2/3}}=A_{k,L}+B_{k,L}
\ee
with
\begin{align}\label{supra}
A_{k,L}&=\sum_{i=1}^k \frac{\mathfrak{N}^i}{L^{\frac23}}\quad \text{and}\quad B_{k,L}=\sum_{i= k+1}^{v_L} \frac{\mathfrak{N}^i}{(X^i)^{\frac23}}\, \Big(\frac{X^i}{L}\Big)^{\frac23},
\end{align}
Then, the proof  of Proposition  \ref{1} will be deduced from the following two steps.
\begin{itemize}
\item[\bf Step 1] {} Show that for all $k\in \N$ and under $\bP_{\beta}(\cdot\,|\, L\in \mathfrak{X})$, 
\be{liml}
\lim_{L\to \infty} A_{k,L}=_{\text{law}} \textstyle \sum_{i=1}^k   Y_i \, (\mathfrak{U}^i)^{\frac 23}.
\ee
  \item[\bf Step 2] {} Show that for all $\gep>0$,
\be{lims}
\lim_{k\to \infty} \limsup_{L\to \infty} \bP_\beta (B_{k,L} \geq \gep | L\in \mathfrak{X})=0.
\ee
\end{itemize} 
Before proving \eqref{liml} and \eqref{lims}, we need to settle some preparatory lemmas.
To begin with we let $F$ be the distribution function of $X$ under $\bP_{\beta,\mu_\beta}$ that is $F(t)=\bP_{\beta,\mu_\beta}(X\leq t)$ for $t\in \R$ 
and $F^{-1}$ its pseudo-inverse, that is $F^{-1}(u)=\inf\{t\in \R\colon\, F(t)\geq u\}$ for $u\in (0,1)$.
\bl{fm}
There exists $ C>0$ such that 
\be{pi}
F^{-1}(u)\sim\tfrac{C}{(1-u)^3}\quad \text{as}\  u\to 1^-.
\ee
\el
\begin{proof}
The proof is a straightforward consequence of \eqref{eq:asympatauplustau}.

\end{proof}

Recall \eqref{density}. The next lemma deals with the convergence in law, as $m \to \infty$, of the horizontal extension of an excursion renormalized by $m^{2/3}$
and conditioned on  the area of the excursion being equal to $m$.    
\bl{col}
For all $\beta>0$ and all $m\in \N$ we consider the random variable
$\frac{\mathfrak{N}}{m^{2/3}}$ under the laws $\bP_{\beta,a}( \cdot  | X=m)$ with $a\in \{0,\mu_\beta\}$. We have 
\be{impp}
\lim_{m \to \infty} \frac{\mathfrak{N}_1}{m^{2/3}}=_{\text{Law}} Y_1,
\ee
and also that the sequence $\big(\bE_{\beta,a}\big( \frac{\mathfrak{N}}{m^{2/3}} \big | X=m \big)\big)_{m\in \N}$ is bounded.

\el
\begin{proof}
With the help of Theorem \eqref{theo:wachtellclt}  we can use the following equality 
\be{equall}
P_{\beta,\mu_\beta}(\mathfrak{N}=n\, \big|\, X=m)=P_{\beta,\mu_\beta}(\mathfrak{A}=m-n\, \big|\, \mathfrak{N}=n)\ 
\frac{P_{\beta,\mu_\beta}(\mathfrak{N}=n)}{P_{\beta,\mu_\beta}(X=m)}, 
\ee
combined with \eqref{eq:equivptauegaln} and \eqref{eq:asympatauplustau}, to claim that  there exists a $D>0$ such that 
\be{tcll}
P_{\beta,\mu_\beta}(\mathfrak{N}=n\, \big|\, X=m)=D \frac{m^{4/3}}{n^3} w\Big(\frac{m-n}{n^{3/2}}\Big) + \frac{m^{4/3}}{n^3} (\gep_1(m)+\gep_2(n))
\ee
with $\gep_1(m)$ and $\gep_2(n)$ vanishing as $m,n\to \infty$.

To display an upper bound for the sequence $\big(\bE_{\beta,\mu_\beta}\big( \frac{\mathfrak{N}}{m^{2/3}} \big | X=m \big)\big)_{m\in \N}$ it suffices of course to consider 
\begin{align}\label{bbd}
\bE_{\beta,\mu_\beta}\Big( \frac{\mathfrak{N}}{m^{2/3}} \ind_{\{\mathfrak{N}\geq m^{2/3}\}}\big | X=m \Big)&=
\sum_{n=m^{2/3}}^l \frac{n}{m^{2/3}} P_{\beta,\mu_\beta}(\mathfrak{N}=n\, \big|\, X=m)\\
\nonumber &=\sum_{n=m^{2/3}}^m D \frac{m^{2/3}}{n^2} w\Big(\frac{m-n}{n^{3/2}}\Big)+ 
\sum_{n=m^{2/3}}^m \frac{m^{2/3}}{n^2} (\gep_1(m)+\gep_2(n))
\end{align}
where we have used \eqref{tcll}. Since the second term in the r.h.s. in \eqref{bbd} clearly vanishes as $m \to \infty$, we focus on the first term   and since $w$ is uniformly continuous because $s\to w(s)$ is continuous on $[0,\infty)$ and vanishes as $s\to \infty$,
we can write the first term  as a Riemann sum that converges to $\int_{1}^\infty \frac{D}{x^2}\,  w(\frac{1}{x^{3/2}}) dx$ plus a rest that vanishes as $m \to \infty$
and this gives us the expected boundedness.

Similarly, the convergence in law is obtained by  picking $t\in [0,\infty)$ and by writing\\ 
$\bP_{\beta,\mu_\beta}\Big(\frac{\mathfrak{N}_1}{l^{2/3}} \leq t\, \big| \, X=m \Big) =:  \tilde u_m + \tilde v_m$ where
\begin{align}\label{limss}
 \tilde u_m &=\frac{1}{\bP_{\beta,\mu_\beta}(X=m)}  \sum_{k=1}^{\eta\,  m^{2/3}} \probmubeta{A_{k-1} = m-k, \tau=k} \\
\tilde v_m &=\frac{1}{\bP_{\beta,\mu_\beta}(X=m)}  \sum_{k=\eta\, m^{2/3}}^{t m^{2/3} }
      \probmubeta{A_{k-1} = m-k, \tau=k} 
\end{align}
where $\eta\in (0,t)$. 
We note easily that $\tilde u_m=[(1-e^{-\beta/2})\,\bP_{\beta,\mu_\beta}(X=m)]^{-1} \, u_m$ with $u_m$ defined in \eqref{auta}. Therefore,  
\eqref{eq:asympatauplustau} and \eqref{limsu} tell us that $\tilde u_m$ can be made arbitrarily small provided $\eta$
is small enough and  $m$ large enough. Thus, it remains to deal with $\tilde v_m$, which,
with the help of \eqref{eq:asympatauplustau} is treated as the second term in the  r.h.s. in \eqref{auta}. Thus, \eqref{condu} tells us that there exists a $D>0$ 
such that 
$\lim_{m \to \infty} \tilde v_m= \int_\eta^{t} D t^{-3} w(t^{-3/2})\, dt$ and this suffices to complete the proof of \eqref{impp}.


\end{proof}

\bl{vl} For $\beta>0$ and $\gep>0$, there exists a $c_\gep>0$ such that  for $L\in \N$,
\be{pu}
\bP_{\beta}(v_L\geq c_\gep \, L^{1/3})\leq \gep.
\ee
\el
\begin{proof}
Since under $\bP_\beta$ only the first excursion has law $\bP_{\beta,0}$ and the others $\bP_{\beta,\mu_\beta}$ the proof of 
\eqref{pu} will be complete once we show for instance that for $c$ large enough and $L\in \N$
\be{nof}
\bP_{\beta,\mu_\beta}(\max\{X_i, i\leq cL^{1/3}\} \leq L)\leq \gep.
\ee
We recall that, if, $(\Gamma_i)_{i=1}^{c L^{1/3}+1}$ are the partial sums of $(\gamma_i)_{i=1}^{c L^{1/3}+1}$ a sequence of IID  exponential random variables 
with parameter $1$, we can state that 
\be{equp}
\max\{X_i, i\leq cL^{1/3}\}=_{\text{law}}\  F^{-1}\big(\ \tilde \Gamma_{cL^{1/3}}/\Gamma_{cL^{1/3}+1}\big).
\ee
with $\tilde \Gamma_{cL^{1/3}}=\gamma_2+\dots+\gamma_{cL^{1/3}+1}$. Thus we can rewrite the l.h.s. in \eqref{nof} as
\be{nof2}
\bP\bigg(F^{-1}\bigg(\frac{\tilde \Gamma_{cL^{1/3}}}{\Gamma_{cL^{1/3}+1}}\bigg) \leq L\bigg)=P\big(F^{-1}( D_L) (1-D_L)^3 \leq L (1-D_L)^3\big),
\ee
with $D_L=\tilde \Gamma_{cL^{1/3}}/\Gamma_{cL^{1/3}+1}$. After some easy simplifications we rewrite \eqref{nof2} as 
\be{nof3}
\bP\bigg(\gamma_1\geq  \big[ F^{-1}\big(D_L)]^{\frac13} (1-D_L) \  \frac{ \Gamma_{cL^{1/3}+1}}{L^{\frac13}} \bigg),
\ee
and then we use Lemma \ref{fm} to claim that,  as $L\to \infty$, the r.h.s. in \eqref{nof3} converges to $P(\gamma_1\geq c\,  C^{1/3})$. This completes the proof of 
the lemma.


\end{proof}

\bl{rem}
For every $\beta>0$, there exists a $M>0$ such that, for all function $G: \cP(\{0,\dots,L/2\})\to \R^+$, we have
\be{bounddd}
\bE_\beta\Big[ G\big(\mathfrak X\cap \big[0,\tfrac{L}{2}\big]\big)\,  \big|\, L\in \mathfrak{X}\Big] \leq M\   \bE_\beta\Big[ G\big(\mathfrak{X}\cap \big[0,\tfrac{L}{2}\big]\big)\Big]
\ee
\el

\begin{proof}
We compute the Radon Nikodym density of the image measure of $\bP_\beta(\cdot | L\in \mathfrak{X})$ by $\mathfrak{X}\cap[0,L/2]$ 
w.r.t. its counterpart without the $\{L\in \mathfrak{X}\}$ conditioning. Thus,  for $0=y_0< y_1<\dots<y_m\leq L/2$ 
we obtain
$$\frac{\bP_\beta(\tau\cap [0,\frac{L}{2}]=(y_0,\dots,y_m) | L\in \mathfrak{X})}{\bP_\beta(\tau\cap [0,\frac{L}{2}]=(y_0,\dots,y_m) )}:=G_L(y_m)+K_L(y_m)$$
with 
\begin{align}\label{deff}
\nonumber G_L(y)&= \frac{\sum_{n=0}^{L/4} P_{\beta,\mu_\beta}(n\in \mathfrak{X})\ \ P_{\beta,\mu_\beta}(X=L-n-y)}{P_{\beta}(L\in \mathfrak{X}) \ 
P_{\beta,\mu_\beta}(X\geq \frac{L}{2}-y)},\\
K_L(y)&= \frac{\sum_{n=L/4}^{L/2} P_{\beta,\mu_\beta}(n\in \mathfrak{X})\ \ P_{\beta,\mu_\beta}(X=L-n-y)}{P_{\beta}(L\in \mathfrak{X}) \ 
P_{\beta,\mu_\beta}(X\geq \frac{L}{2}-y)}.
\end{align}
Note that, for $y=0$, the terms $P_{\beta,\mu_\beta}(X=L-n-y)$ and $P_{\beta,\mu_\beta}(X\geq \frac{L}{2}-y)$ in the expression of $G_L(y)$ and $K_L(y)$ 
should be replaced by $P_{\beta,0}(X=L-n-y)$ and $P_{\beta,0}(X\geq \frac{L}{2}-y)$, respectively. However, this does not change anything in the sequel and this is why we will focus on 
$y>0$. It remains to prove that $G_L(y)$ and $K_L(y)$ are bounded above uniformly in $L\in \N$ and
$y\in \{0,\dots,L/2\}$. We will focus on the $G_L$ function since $K_L$ can be treated similarly.

The constants $c_1,\dots,c_4$ below are strictly positive and independent of $L,n,y$.
By recalling \eqref{eq:asympatauplustau} and since $L-n-y\geq L/4$ when $n\in \{0,\dots,L/4\}$ we can claim that in the numerator of $G_L(y)$, the term
$P_{\beta,\mu_\beta}(X=L-n-y)$ is bounded above by $c_1/L^{4/3}$  independently of $n$ while \eqref{attain}  tells us that  $\sum_{n=0}^{L/4} P_{\beta,\mu_\beta}(n\in \mathfrak{X})
\leq c_2 L^{1/3}$. Let us now deal with the denominator :  \eqref{attain}  tell us that $P_{\beta,\mu_\beta}(L\in \mathfrak{X})\geq  c_3/L^{2/3}$ while \eqref{eq:asympatauplustau}  gives
$$P_{\beta,\mu_\beta}(X\geq \tfrac{L}{2}-y)\geq P_{\beta,\mu_\beta}(X\geq \tfrac{L}{4})\geq \tfrac{c_4}{L^{1/3}}.$$  As a consequence, $G_L(y)$ is bounded above uniformly in $L\in \N$ and
$y\in \{0,\dots,L/2\}$.


\end{proof}

We resume the proof of  Proposition  \ref{1}.

\subsubsection{Proof of {\bf Step 1}  (\ref{liml})}

The proof of Step 1 will be complete once we show that 
\be{impp2}
\lim_{l \to \infty} \left(  \Big(\frac{X^1}{L}\Big)^{\frac23},\dots,\Big(\frac{X^k}{L}\Big)^{\frac23},  \frac{\mathfrak{N}^1}{(X^1)^{\frac23}},\dots,
\frac{\mathfrak{N}^k}{(X^k)^{\frac23}}\right)\,=_{\text{Law}} (\mathfrak{U}^1,\dots,\mathfrak{U}^k,Y_1,\dots,Y_k).
\ee
To obtain this convergence in law, we consider $g_1,\dots,g_k$ that are real Borel and bounded functions. We consider also $t\in \N$ and $(x_i)_{i=1}^t$ a sequence of strictly positive integers satisfying   $x_1+\dots+x_t=L$ with an order statistics $x_{r_1}\geq\dots\geq x_{r_t}$. The key observation is that, by independence of the $(\mathfrak{N}_i,X_i)_{i\in \N}$, we have
\begin{align}\label{keyo}
\bE_\beta \bigg[ \prod_{j=1}^k g_j\Big(\tfrac{\mathfrak{N}^j}{(X^j)^{\frac23}}\Big)\, & \Big |\,  X_i=x_i,\, 1\leq i\leq t \bigg]=\prod_{j=1}^k \bE_{\beta,\, 0\, \ind_{\{r_j=1\}}+\mu_\beta\, \ind_{\{r_j>1\}}} \bigg[g_j\Big(\tfrac{\mathfrak{N}}{X^{\frac23}}\Big) \, \Big |\, X=x_{r_j} \bigg].
\end{align}
We consider $f_1,\dots,f_k,g_1,\dots,g_k$ that are real Borelian and bounded functions  and we use \eqref{keyo} to observe that 
\begin{align}\label{keyoo}
 \bE_\beta\bigg[ \prod_{j=1}^k &f_j\Big(\big[\tfrac{X^j}{L}\big]^{\frac23}\Big)\    g_j\Big(\tfrac{\mathfrak{N}^j}{(X^j)^{\frac23}}\Big)\, \Big | \, L\in \mathfrak{X} \bigg]\\
\nonumber & = \bE_\beta\bigg[\prod_{j=1}^k f_j\Big(\big[\tfrac{X^j}{L}\big]^{\frac23}\Big)\    \bE_{\beta,\, 0 \,\ind_{\{r_j=1\}}+\mu_\beta \,\ind_{\{r_j>1\}}} \Big[g_j\Big(\tfrac{\mathfrak{N}}{X^{\frac23}}\Big) \Big |\, X=X^j\Big]\, \Big | \, L\in \mathfrak{X} \bigg].
 \end{align}
Because of Lemma \ref{tes}, we can assert that, under $\bP_\beta(\cdot | L\in \mathfrak X)$ the random set $(\mathfrak{X}/L)\cap [0,1]$ converges in law towards $\mathfrak{U}$, i.e., 
 the $1/3$ regenerative set on $[0,1]$ conditioned on $1\in \mathfrak{U}$ (the latter convergence is  proven e.g. in \cite{CSZ}, Proposition A.8). As a consequence $\big(\big(\frac{X^1}{L}\big),\dots,\big(\frac{X^k}{L}\big)\big)$
converges in law towards  $(\mathfrak{U}^1,\dots,\mathfrak{U}^k)$ which implies that $X^1,\dots,X^k$ tend to $\infty$ in probability and therefore
we can use Lemma \ref{col} to show that the l.h.s. in \eqref{keyoo} tends to (as $L\to \infty$)
\begin{align}\label{keyooo}
\prod_{j=1}^k \bE\big[g_j(Y)\big]  \bE\Big[\prod_{j=1}^k f_j(\mathfrak{U}^j)\Big].
 \end{align}
Thus, the proof of Step 1 is complete.

\subsubsection{Proof of {\bf Step 2} (\ref{lims})}


Under 
$\bP_\beta (\cdot |L\in \mathfrak{X})$, the reversibility yields that 
$$(\mathfrak{N}_i,\mathfrak{A}_i,X_i)_{i=1}^{v_L}=_{\text{law}} (\mathfrak{N}_{1+v_L-i},\mathfrak{A}_{1+v_L-i},X_{1+v_L-i})_{i=1}^{v_L}.$$
We set $v_{L/2}:=\max \{i\geq 1\colon\, S_i\leq L/2\}$ and $v'_{L/2}:=v_L-\min \{i\geq 1\colon\, S_i\geq L/2\}$ such that $v_{L/2}$ and 
$v'_{L/2}$ have the same law and 
\be{thk}
(\mathfrak{N}_i,\mathfrak{A}_i,X_i)_{i=1}^{v_{L/2}}=_{\text{law}} (\mathfrak{N}_{1+v_L-i},\mathfrak{A}_{1+v_L-i},X_{1+v_L-i})_{i=1}^{v'_{L/2}}.
\ee
We also denote by $(\mathfrak{N}_{\text{mid}},\mathfrak{A}_{\text{mid}},X_{\text{mid}})$ the features of the excursion containing $L/2$ in case 
$\tau_{v_{L/2}}<L/2$. In case $\tau_{v_{L/2}}=L/2$, we set  $(\mathfrak{N}_{\text{mid}},\mathfrak{A}_{\text{mid}},X_{\text{mid}})=(0,0,0)$.

By applying Lemma \ref{vl} 
 we can state that by choosing $c$ large enough, the quantity 
$\bP_\beta(v_{L/2}, v'_{L/2}\leq cL^{1/3}| L\in \mathfrak{X})$ is arbitrary close to $1$ uniformly in $L$ provided $c$ is chosen large enough. 
Thus, we set 
$$H_{c,L}= \{L\in \mathfrak{X}\} \cap \{v_L,v'_L\leq cL^{1/3}\}$$ 
and Step 2 will be complete once we show that, for each $c>0$ and $\gep>0$.
\be{lims2}
\lim_{k\to \infty} \limsup_{L\to \infty} \bP_\beta (B_{k,L} \geq \gep | H_{c,L})=0.
\ee
By Markov's inequality    Step 2 is a consequence of 
\be{lims1}
\lim_{k\to \infty} \limsup_{L\to \infty} \bE_\beta (B_{k,L}  | H_{c,L})=0.
\ee
We recall \eqref{supra} and we compute $\bE_\beta (B_{k,L}  | H_{c,L})$ by conditioning  on  $\sigma(X_i,i\in \N)$ as we did in \eqref{keyo}.  We recall that
$H_{c,L}$ is $\sigma(X_i,i\in \N)$-measurable  and that under $\bP_\beta$, the first excursion has law $\bP_{\beta,0}$ and the others $\bP_{\beta,\mu_\beta}$, i.e.,
\be{compu}
\bE_\beta (B_{k,L} | H_{c,L})= \bE_{\beta} \Bigg(\sum_{i\geq k+1}^{v_L}  \bE_{\beta,\, 0\,\ind_{\{r_i=1\}}+\mu_\beta \,\ind_{\{r_i>1\}}} \bigg[\frac{\mathfrak{N}}{X^{\frac23}}\Big | X=Xi\bigg]\, \bigg(\frac{X^i}{L}\bigg)^{\frac23}\bigg |\  H_{c,L}\Bigg),
\ee
but then, we can use Lemma \ref{col} which allows us to bound by $M>0$ each term $\bE_{\beta,a}\big[\mathfrak{N}/X^{\frac23}\big | X=X_i\big]$ with $a\in \{0,\mu_\beta\}$. By using again  the fact there exists an $\eta>0$ such that  $\bP_\beta(H_{c,L} | L\in \mathfrak X)\geq \eta$ uniformly in $L$ we can claim that  the proof will be complete once we show that 
\be{lims3}
\lim_{k\to \infty} \limsup_{L\to \infty} \bE_\beta \bigg(\sum_{i\geq k+1}^{v_L}  \Big(\frac{X^i}{L}\Big)^{\frac23}\ \ind_{\{v_{L/2}\leq cL^{1/3}\}} \ind_{\{v'_{L/2}\leq cL^{1/3}\}} \Big | L\in \mathfrak{X} \bigg)=0.
\ee
We note that, under $\bP_\beta(\cdot | L\in \mathfrak{X})$ we have necessarily $X^i\leq L/k$ for $i\geq k+1$. 
For simplicity we assume that $k\in 2\N$ and we denote by  $(\tilde X^1,\dots,\tilde X^{v_{L/2}})$ the order statistics of the variables $(X_1,\dots,X_{v_{L/2}})$ and by 
$(\bar X^1,\dots,\bar X^{v'_{L/2}})$ the order statistics of the variables $(X_{v_L}, X_{v_L-1},\dots,X_{1+v_L-v'_{L/2}} )$. Then,  we can easily note that 
$\sum_{i=1}^k X^i\geq \sum_{i=1}^{k/2} \tilde X^i+ \bar X^i$ so that the expectation in \eqref{lims3} is bounded above by
\begin{align}\label{decoup}
 \textstyle{ 2 \bE_\beta\big( \sum_{i=k/2+1}^{v_{L/2}}} & \textstyle{ (\frac{\tilde X^i}{L})^{\frac23}   \ind_{\{v_{L/2}\leq cL^{1/3}\}}\big| L\in \mathfrak{X} )}+{\textstyle  \bE_\beta\big( \big(\frac{X_{\text{mid}}}{L}\big)^{\frac23} \, \ind_{\{X_{\text{mid}}\leq L/k\}}\big| L\in \mathfrak{X})},
\end{align}
where the factor 2 in front of the first term is a direct consequence
of \eqref{thk}. The second term  in \eqref{decoup} is clearly bounded by  $(1/k)^{2/3}$ and therefore, it can be omitted.
As a consequence, it suffices to show that 
\be{lims4}
\lim_{k\to \infty} \limsup_{L\to \infty} \bE_\beta\bigg( \sum_{i=k}^{v_{L/2}}  \Big(\frac{\tilde X^i}{L}\Big)^{\frac23}   \ind_{\{v_{L/2}\leq cL^{1/3}\}}\Big| L\in \mathfrak{X} \bigg)=0.
\ee
At this stage, we note that $\sum_{i=k}^{v_{L/2}}  (\frac{\tilde X^i}{L})^{\frac23}   \ind_{\{v_{L/2}\leq cL^{1/3}\}}$ only depends on the random set of points 
$\mathfrak{X}\cap[0,L/2]$ and this allows us to use lemma \ref{rem} to claim that proving \eqref{lims4} without the conditioning by $\{L\in \mathfrak{X}\}$ is sufficient. Therefore,
we only need to estimate the quantity
$$\bE_\beta\bigg( \sum_{i=k}^{cL^{1/3}}  \Big(\frac{X^i}{L}\Big)^{\frac23}\bigg),$$
where $(X^1,\dots,X^{cL^{1/3}})$ is the order statistics of  $(X_1,\dots,X_{cL^{1/3}})$ under $\bP_\beta$ without any conditioning.
We recall that, if $(\Gamma_i)_{i=1}^{c L^{1/3}+1}$ are the partial sums of $(\gamma_i)_{i=1}^{c L^{1/3}+1}$ a sequence of IID  exponential random variables 
with parameter $1$, we can state that 
\be{equ}
(X^i)_{i=1}^{c L^{1/3}}=_{\text{law}} \bigg( F^{-1}\bigg[\frac{\Gamma_{cL^{1/3}+1-i}}{\Gamma_{cL^{1/3}+1}}\bigg]\bigg)_{i=1}^{c L^{1/3}}.
\ee
Moreover, by Lemma \ref{fm}, we can claim that there exists a $M>0$ such that $F^{-1}(u)\leq M/(1-u)^3$ for all $u\in (0,1)$ and consequently
\be{finm}
\bE_\beta\bigg( \sum_{i=k}^{cL^{1/3}}  \Big(\frac{X^i}{L}\Big)^{\frac23}\Big)\leq \frac{M^2}{L^{2/3}}  \sum_{i=k}^{cL^{1/3}} \bE\bigg(\bigg[\frac{\Gamma_{cL^{1/3}+1}}{\Gamma_{i}}\bigg]^2\bigg)
\ee
but then we can bound from above the general term in the sum of the r.h.s. in \eqref{finm} by 
\begin{align}\label{finm2}
\bE\bigg(\left[\frac{\Gamma_{cL^{1/3}+1}}{\Gamma_{i}}\right]^2\bigg)&= c^2 L^{\frac23} \bE\bigg(\left[\frac{\Gamma_{cL^{1/3}+1}}{cL^{1/3}}\right]^2 
\left[\frac{1}{\Gamma_{i}}\right]^2\bigg)\leq c^2 L^{\frac23} \bE\bigg(\left[\frac{\Gamma_{cL^{1/3}+1}}{cL^{1/3}}\right]^4 \bigg)^\undemi
\bE\bigg(\left[\frac{1}{\Gamma_{i}}\right]^4 \bigg)^\undemi.
\end{align}
It remains to point out, on the one hand, that $\bE_\beta\Big(\Big[\frac{\Gamma_{cL^{1/3}+1}}{cL^{1/3}}\Big]^4 \Big)$ converges to $1$ as $L\to \infty$ and, 
on the other hand, that for all $i\in \N\setminus \{0\}$, $\Gamma_i$ follows a law Gamma of parameter $(i,1)$ which entails that for $i\geq 5$, 
$\bE\Big(\Big[\frac{1}{\Gamma_{i}}\Big]^4 \Big)=\frac{(i-5)!}{(i-1)!} $. Consequently, we can use \eqref{finm} and \eqref{finm2}
to complete the proof of the step.

\subsubsection{Proof of Claim \ref{2}}

We use again the  random walk representation \eqref{tgh}, and since 
$\Gamma_\beta=1$, we obtain 
\begin{align}\label{lois}
P_{L,\beta}\Big(\frac{N_l}{L^{2/3}}\leq t\Big) &=\frac{c_\beta}{ \tilde{Z}_{L,\beta}} \sum_{N=1}^{t L^{2/3}}
\probbeta{G_N=L-N,V_{N+1}=0}.
\end{align}
Similarly to what we did in (\ref{parffu}--\ref{TTF}), we partition the event  $\{G_N=L-N,V_{N+1}=0\}$ depending on the length $r$
on which the random walk sticks at the origin before its right extremity, that is
\begin{align}
P_{L,\beta}\Big(\frac{N_l}{L^{2/3}}\leq t\Big) &= \sum_{r=0}^{tL^{2/3}-1}
\frac{(1-e^{-\beta/2}) }{ c_\beta^{r-1}\tilde{Z}_{L,\beta}}   \sum_{N=1}^{t L^{2/3}-r}
\probbeta{G_{N}=L-N-r, V_N\neq 0, V_N V_{N+1}\leq 0} \\
&= \sum_{r=0}^{tL^{2/3}-1}
\xi_{r,L}
\probbeta{\frac{r+\mathfrak{N}_1+\dots+\mathfrak{N}_{v_{L-r+1}}}{L^{2/3}} \leq t \, |\, L-r+1\in \mathfrak{X}},
\end{align}
where we recall the definition of $\xi_{r,L}$ in \eqref{defalp}. This ends the proof of Claim
\ref{2}.

\subsection{Identifying the distribution of $\lim_{L\to +\infty}
  \frac{N_l}{L^{2/3}}$ }\label{secder}
  Let $B$ be a standard Brownian motion on the line ; we consider its
  geometric area and its continuous inverse
$$ G_t(B) = \int_0^t \valabs{B_s}\, ds\,,\quad g_a=\inf\ens{t>0 :
  G_t(B) =a}\,.$$
We aim to identify \emph{formally} the distribution of $\lim_{L\to +\infty}
  \frac{N_l}{L^{2/3}}=_{law}\sum_{i=1}^{+\infty} Y_i
  \mathfrak{U}_i ^{2/3}$ with the distribution of $g_1$ conditionnally
  on $B_{g_1}=0$.
\subsubsection{Step 1 : Identifying the distribution of $Y_1$}. We
shall show that $Y_1$ is distributed as the extension of a Brownian
excursion normalized by its area. More precisely on the space of
excursions $(U_\delta,\cU_\delta)$ (see Revuz and Yor \cite[Chapter
XII, Section 2]{MR1083357}) we consider the extension (duration) and
area
$$ \zeta(w) := \inf\ens{t>0 : w(t) =0}\,,\quad A(w) :=
\int_0^{+\infty} w(s)\, ds\,,$$
and the scaling operator
$$ s_c(w)(t) = \unsur{\sqrt{c}}w(ct)\qquad( c>0)\,.$$

The operator that normalizes the area is 
$ \eta(w) := s_{A(w)^{2/3}}(w)$. It satisfies $A(\eta(w))=1$ and there
exists a probability $\gamma_A$ defined on $\ens{w\in U_\delta :
  A(w)=1}$, called \emph{the law of the Brownian excursion normalized by its
area}, that satisfies (with $n_+$ denoting the Itô measure of positive
excursions) for every positive measurable $F,\psi$:
$$ n_+\etc{F(\eta(w)) \psi(A(w))} = \gamma_A(F(w)) \eta_+(\psi(A(w)))\,.$$

It is now just a matter of playing with Brownian scaling, and with the
characterization of the normalized Brownian excursion (see e.g. \cite[Chapter
XII, Exercise 2.13]{MR1083357}) to show that $Y_1$ has distribution $\zeta(w)$
with $w$ sampled from $\gamma_A$.

\subsubsection{Step 2: a Brownian construction of a $\unsur{3}$-stable
  regenerative set}
Observe that if $(\tau_t, t\ge 0)$ is the inverse local time at level
$0$ of Brownian motion $B$, then by strong Markov property
$(G_{\tau_t}, t\ge 0)$ is a subordinator. Since it has the scaling
$(G_{\tau_{ct}}, t\ge 0) =_{law} (c^3 G_{\tau_t}, t\ge 0)$, its range $\cR$
is a stable $\unsur{3}$ regenerative set on $[0,+\infty[$. Therefore if
$(\mathfrak{U}_i)_{i=1}^{+\infty}$ are the interarrivals of
$\taufrak=\cR\cap\etc{0,1}$, we have the representation
$$\ens{\mathfrak{U}_i,1\le i} = \ens{G_{\tau_s} -G_{\tau_{s-}} : s>0,
  G_{\tau_s}\le 1}$$ and thus $\sum_{i=1}^{+\infty} Y_i
  \mathfrak{U}_i ^{2/3}$ is, in distribution,  the sum of all the extensions (durations)
  of the Brownian excursions, summed up until their cumulated areas do
  not exceed $1$, that is 
$$ \sum_{i=1}^{+\infty} Y_i
  \mathfrak{U}_i ^{2/3} =_{law} \sum_{s : G_{\tau_s}\le 1} \tau_s -\tau_{s-}
  = g_1\,.$$

\subsubsection{Step 3: a formal conditioning} 
We now have to take into account the conditioning. The
$(\mathfrak{U}_i)_{i=1}^{+\infty}$ are the interarrivals of
$\taufrak=\cR\cap\etc{0,1}$, conditionnaly on $1\in\taufrak$ that is 

$$\ens{\mathfrak{U}_i,1\le i} =_{law} \ens{G_{\tau_s} -G_{\tau_{s-}} : s>0,
  G_{\tau_s}\le 1} \quad \text{conditionnaly on $\ens{\exists s :
    G_{\tau_s}=1}$}$$

Since by definition of $g$ we have $\ens{\exists s :
    G_{\tau_s}=1} =\ens{B_{g_1}=0}$ we can conclude that 
$$ \sum_{i=1}^{+\infty} Y_i
  \mathfrak{U}_i ^{2/3} =_{law} g_1\, \quad\text{conditionnaly on
    $\ens{B_{g_1}=0}$}
$$
Let us explain why this conditioning is only formal. The sets on which
we condition are of zero probability measure and thus the law of
$\taufrak=\cR\cap\etc{0,1}$ conditioned by   $1\in\taufrak$ is defined
in \cite{CSZ} through regular conditional distributions (formulas
(1.19) and (1.20)).\footnote{There are well known  examples of negligible sets $A$
limits of two different sequences $A_n$ and $B_n$ of non negligible
sets, and which lead to two different limiting probabilities $\lim
\prob{. \mid A_n} \neq \lim \prob{. \mid B_n}$ by considering
different regular conditional probability measures}
\section{Appendix}\label{sec:appendice}

\subsection{Perfect simulation procedure}
We shall use the acceptance-reject algorithm. Let $X,X_1, \ldots, X_n$
be IID with values in $E$, let $A$ be a measurable set of $E$ such
$\prob{X \in A}>0$, and
$$ T := \inf\ens{ n \ge 1 : X_n \in A}\,.$$
Then $T$ has a geometric distribution of parameter $\prob{X\in A}$,
$X_T$ is independent of $T$, with distribution the conditional law
$$ \prob{X_T \in B} = \prob{X\in B \mid X \in A}\,.$$

For $\beta=\beta_c$, we have $\Gamma_\beta=1$ so the representation
formula\eqref{partfun} yields
$$ Z_{L,\beta} = e^{\beta L} \sum_{N=1}^L \probbeta{V \in \cV_{N,L-N}}
= e^{\beta L} \probbeta{V \in A}$$
with $A=\ens{ V : \exists N : G_N(V) = L-N, V_{N+1}=0}$.

Following the same steps as in Section \ref{sec:rep}
  we show that if
$B$ is a set of trajectories in $\Omega_L$ such that $\ens{l\in B,
  L_l(L)=N} = \ens{V\in C, V_{N+1}=0,G_N(V) = L-N}$ we have
\begin{align}
  P_{L,\beta}\etp{l \in B} = \unsur{Z_{L,\beta}} e^{\beta L}
  \sum_{N=1}^L \probbeta{V \in \cV_{N,L-N}, V \in C} = \probbeta{V \in C
    \mid V \in A}\,.
\end{align}
Therefore we can use an acceptance-reject to simulate a trajectory of
an IPDSAW under $P_{L,\beta}$ (for $\beta=\beta_c$). The mean number
of rejects for an acceptance is the mean of the geometric r.v. that
is, thanks to Theorem~\ref{pfa},
$$\unsur{P_{L,\beta}\etp{V\in A}} = \unsur{\tilde{Z}_{L,\beta}} \sim
\unsur{c} L^{2/3}$$
In a nutshell, we have a perfect simulation algorithm with complexity
$L^{2/3}$.

\medskip
Of course, one can try to use the same trick for $\beta \neq \beta_c$,
say $\beta>\beta_c$ so that $\Gamma_\beta<1$ We let $S$ be an
independent geometric r.v. of parameter $1-\Gamma_\beta$, and add a
cemetery point $\delta$ for trajectories : $S$ is now a lifetime
so that $V_N=\delta$ if $N\ge S$. Under this new probability
$\bar{\bP}_{L,\beta}$ we have
$$ P_{L,\beta}(B) = \bar{\bP}_{L,\beta}(V \in C \mid V \in A)$$
 and we have a perfect simulation algorithm. The problem is that now
 the mean number of acceptance for a reject is growing very fast with
 $L$ :
$$ \unsur{\tilde{Z}_{L,\beta}} \ge C \unsur{L^{\kappa}} e^{c
  \sqrt{L}}.$$

\subsection{Proof of \eqref{eq:equivptauegaln}}\label{lltt}
Since by definition  
\begin{align}
\tau&= \inf\ens{i\geq 1\colon\,  V_{i-1}\neq 0 \text{ and }
V_{i-1}V_i \le 0},\\
\nonumber \tilde \tau&= \inf\ens{i\geq 1\colon\,  V_{i}< 0},
\end{align}
we can claim that $V_0>0$ implies $\tau=\tilde \tau$.
Thus, we recall \eqref{defubeta} and can write 
\be{calcons}
\probmubeta{\tau=n}=\mu_\beta(0)\,  \probbeta{\tau=n}+2 \sum_{x=1}^\infty \mu_\beta(x) \, 
\bP_{\beta,x} (\tilde \tau=n).
\ee
By disintegrating the event $\{\tilde \tau=n+1\}$ with respect to the value taken by $V_1$ and by recalling that 
for $x>1$ we have $\bP_\beta(V_1=x)=2 \mu_\beta(x)/ c_\beta (1-e^{-\beta/2})$ we obtain that 
\be{eq:disi}
\bP_{\beta}(\tilde \tau=n+1)=\frac{2}{1+e^{-\beta/2}} \, \sum_{x=1}^\infty \mu_\beta(x) \, 
\bP_{\beta,x} (\tilde \tau=n).
\ee
Moreover, under $\bP_{\beta}$ we can disintegrate the event $\{\tau=n\}$ with respect to the time  $k$
during which $V$ sticks to $0$ before leaving it, i.e.,
\begin{align}\label{eq:disi2}
\probbeta{\tau=n}&=\, 2 \sum_{k=0}^{n-2} \bP_\beta( V_1=\dots=V_k=0,\,  V_{k+1}>0, \dots, V_{n-1}>0, V_n\leq 0),\\
\nonumber &=\, 2 \sum_{k=0}^{n-2} \frac{1}{(c_{\beta})^k} \, \bP_\beta( \tilde \tau= n-k).
\end{align}

At this stage, we recall that by
Theorem 8 of
Kesten~\cite{Kes63} we have that $\probbeta{\tilde \tau=n} \sim C n^{-3/2}\quad \text{with}\quad C=(\espbeta{V_1^2}/2 \pi)^{1/2}$. 
Then,  it remains to recall that $\mu_\beta(0)=1-e^{-\beta/2}$ and to put (\ref{calcons}--\ref{eq:disi2}) together 
 to obtain 
\begin{align}\label{eq:disi3}
\probmubeta{\tau=n}\sim \frac{C}{n^{3/2}} \, \Big[ 2 (1-e^{-\beta/2}) \frac{c_\beta}{c_\beta -1} + 1+e^{-\beta/2}\Big],
\end{align}
which after a straightforward computation gives  us 
\be{fina}
\probmubeta{\tau=n}\sim (1+e^{\beta/2}) \frac{C}{n^{3/2}}.
\ee

\bibliographystyle{amsplain}
\bibliography{cnp}

\end{document}